\documentclass{article}
\usepackage[utf8]{inputenc}
\usepackage[letterpaper,margin=1.0in]{geometry}
\usepackage{hyperref}       
\usepackage{url}            
\usepackage{booktabs}       
\usepackage{amsfonts}       
\usepackage{bbding}
\usepackage{xcolor}         
\usepackage{algorithm}
\usepackage{algorithmic}
\usepackage{graphicx}
\usepackage{enumitem}
\usepackage{multirow}
\usepackage[numbers,sort]{natbib}
\usepackage{color}

\usepackage{amsmath}\allowdisplaybreaks
\usepackage{amsfonts,bm}
\usepackage{amssymb}

\def\1{\bf{1}}

\newcommand{\Norm}[1]{\left\| #1 \right\|}
\newcommand{\norm}[1]{\left\| #1 \right\|_2}

\def\inner#1#2{{\langle #1, #2 \rangle}}
\def \EBP #1{\EB\left[#1\right]}

\def\va{{\bf{a}}}
\def\vb{{\bf{b}}}
\def\vc{{\bf{c}}}

\def\ve{{\bf{e}}}

\def\vg{{\bf{g}}}

\def\vu{{\bf{u}}}
\def\vv{{\bf{v}}}
\def\vw{{\bf{w}}}
\def\vx{{\bf{x}}}
\def\vy{{\bf{y}}}
\def\vz{{\bf{z}}}

\def\fO{{\mathcal{O}}}

\def\BN{{\mathbb{N}}}

\def\BP{{\mathbb{P}}}

\def\BR{{\mathbb{R}}}

\def\mX {{\bf X}}

\newcommand{\R}{\mathbb{R}}

\DeclareMathOperator*{\argmax}{arg\,max}

\usepackage{amsthm}
\usepackage{etoolbox}
\theoremstyle{plain}

\newtheorem{remark}{Remark}

\makeatletter
\def\Ddots{\mathinner{\mkern1mu\raise\p@
\vbox{\kern7\p@\hbox{.}}\mkern2mu
\raise4\p@\hbox{.}\mkern2mu\raise7\p@\hbox{.}\mkern1mu}}
\makeatother

\makeatletter
\newcommand*{\rom}[1]{\expandafter\@slowromancap\romannumeral #1@}
\makeatother

\newtheorem{theorem}{Theorem}[section]
\newtheorem{corollary}{Corollary}[theorem]
\newtheorem{lemma}[theorem]{Lemma}

\newtheorem{assumption}[theorem]{Assumption}
\newtheorem{definition}[theorem]{Definition}

\def\vx {{{\bf x}}}
\def\vw {{{\bf w}}}
\def\vy {{{\bf y}}}
\def\va {{{\bf a}}}

\def\BR{{\mathbb{R}}}

\def\A{{\bf A}}

\def\e{{\bf e}}

\def\g{{\bf g}}
\def\h{{\bf h}}
\def\G{{\bf G}}
\def\H{{\bf H}}
\def\I{{\bf I}}
\def \J{{\bf J}}

\def\LL{{\bf L}}

\def\Q{{\bf Q}}

\def\R{{\bf R}}

\def\u{{\bf u}}

\def\v{{\bf v}}

\def\w{{\bf w}}
\def\X{{\bf X}}
\def\x{{\bf x}}

\def\y{{\bf y}}

\def\z{{\bf z}}
\def\0{{\bf 0}}
\def\1{{\bf 1}}
\def\bu{{\bar \vu}}

\def\OM{{\mathcal O}}

\def\RB{{\mathbb R}}
\def\EB{{\mathbb E}}

\def \hH{{\hat{\H}}}

\def\defeq{\overset{\text{def}}{=}}

\title{Quasi-Newton Methods for Saddle Point Problems and Beyond}
\author{Chengchang Liu\thanks{School of Management, University of Science and Technology of China; 7liuchengchang@gmail.com} \qquad\qquad Luo Luo\thanks{School of Data Science, Fudan University; luoluo@fudan.edu.cn}}
\date{}

\begin{document}

\maketitle

\begin{abstract}
This paper studies 
quasi-Newton methods for solving strongly-convex-strongly-concave saddle point problems (SPP). 
We propose greedy and random Broyden family updates for SPP, which have explicit local superlinear convergence rate of ${\mathcal O}\big(\big(1-\frac{1}{n\kappa^2}\big)^{k(k-1)/2}\big)$, where $n$ is dimensions of the problem, $\kappa$ is the condition number and $k$ is the number of iterations.
The design and analysis of proposed algorithm are based on estimating the square of indefinite Hessian matrix, which is different from classical quasi-Newton methods in convex optimization.
We also present two specific Broyden family algorithms with BFGS-type and SR1-type updates, which enjoy the faster local convergence rate of $\mathcal O\big(\big(1-\frac{1}{n}\big)^{k(k-1)/2}\big)$.
Additionally, we extend our algorithms to solve general nonlinear equations and prove it enjoys the similar convergence rate.

\end{abstract}

\section{Introduction}
In this paper, we focus on the following smoothed saddle point problem
\begin{equation}\label{eq-general-obj}
    \min_{\x \in \RB^{n_x}}\max_{\y \in \RB^{n_y}} f(\vx,\vy),
\end{equation}
where $f$ is strongly-convex in $\vx$ and strongly-concave in $\vy$. We target to find the saddle point $\left(\x^*, \y^*\right)$ which holds that
\begin{equation*}
    f(\vx^*,\vy) \leq f(\vx^*,\vy^*) \leq f(\vx,\vy^*)
\end{equation*}
for all $\vx\in\BR^{n_x}$ and $\vy\in\BR^{n_y}$.
This formulation contains a lot of scenarios including game theory~\cite{bacsar1998dynamic,von2007theory}, AUC maximization~\cite{hanley1982meaning,ying2016stochastic}, robust optimization~\cite{ben2009robust,gao2016distributionally,shafieezadeh2015distributionally}, empirical risk minimization~\cite{zhang2017stochastic}~and reinforcement learning~\cite{du2017stochastic}.

There are a great number of first-order optimization algorithms for solving problem~\eqref{eq-general-obj}, including extragradient method \cite{korpelevich1976extragradien,tseng1995linear}, optimistic gradient descent ascent~\cite{DaskalakisISZ18}, proximal point method \cite{rockafellar1976monotone} and dual extrapolation~\cite{nesterov2006solving}. 
These algorithms iterate with first-order oracle and achieve linear convergence. \citet{lin2020near,wang2020improved} used Catalyst acceleration to reduce the complexity for unbalanced saddle point problem, nearly matching the lower bound of first-order algorithms~\cite{ouyang2018lower,zhang2019lower} in specific assumptions.  
Compared with first-order methods, second-order methods usually enjoy superior convergence in numerical optimization. \citet{huang2020cubic} extended  cubic regularized Newton (CRN) method~\cite{nesterov2006solving,nesterov2008accelerating} to solve saddle point problem~\eqref{eq-general-obj}, which has quadratic local convergence. 
However, each iteration of CRN requires accessing the exact Hessian matrix and solving the corresponding linear systems. These steps arise $\fO\left((n_x + n_y)^3\right)$ time complexity, which is too expensive for high dimensional problems.

Quasi-Newton methods~\cite{broyden1970convergence2,broyden1970convergence,shanno1970conditioning,broyden1967quasi,davidon1991variable} are popular ways to avoid accessing exact second-order information applied in standard Newton methods. 
They approximate the Hessian matrix based on the Broyden family updating formulas~\cite{broyden1967quasi}, which significantly reduces the computational cost. 
These algorithms are well studied for convex optimization.
The famous quasi-Newton methods including Davidon-Fletcher-Powell (DFP) method~\cite{davidon1991variable,fletcher1963rapidly}, Broyden-Fletcher-Goldfarb-Shanno (BFGS)    method~\cite{broyden1970convergence2,broyden1970convergence,shanno1970conditioning} and symmetric rank 1 (SR1) method~\cite{broyden1967quasi,davidon1991variable} enjoy local superlinear convergence~\cite{powell1971on,broyden1973on,Dennis1974A} when the objective function is strongly-convex. 
Recently,~\citet{rodomanov2021greedy,rodomanov2021new,rodomanov2021rates} proposed greedy and random variants of quasi-Newton methods, which first achieves non-asymptotic superlinear convergence. Later, \citet{lin2021faster} established a better convergence rate which is  condition-number-free. \citet{jin2020non,ye2021explicit} showed the non-asymptotic superlinear convergence rate also holds for classical DFP, BFGS and SR1 methods.

In this paper, we study quasi-Newton methods for saddle point problem~\eqref{eq-general-obj}. Noticing the Hessian matrix of our objective function is indefinite, the existing Broyden family update formulas and their convergence analysis cannot be applied directly. 
To overcome this issue, we propose a variant framework of greedy and random quasi-Newton methods for saddle point problems, which approximates the square of the Hessian matrix during the iteration. Our theoretical analysis characterizes the convergence rate by the gradient norm, rather than the weighted norm of gradient used in convex optimization~\cite{rodomanov2021greedy,rodomanov2021new,rodomanov2021rates,lin2021faster,ye2021explicit,jin2020non}. 
We summarize the theoretical results for proposed algorithms in Table \ref{table:convergence}. The local convergence behaviors for all of the algorithms have two periods. The first period has $k_0$ iterations with a linear convergence rate $\OM\big(\left(1-\frac{1}{\kappa^2}\right)^{k_0}\big)$. The second one enjoys superlinear convergence:
\begin{itemize}
    \item For general Broyden family methods, we have an explicit rate ${\mathcal O}\left(\left(1-\frac{1}{n\kappa^2}\right)^{k(k-1)/2}\right)$.
    \item For BFGS method and SR1 method, we have the faster explicit rate ${\mathcal O}\left(\left(1-\frac{1}{n}\right)^{k(k-1)/2}\right)$, which is condition-number-free. 
\end{itemize}
Additionally, our ideas also can be used for solving general nonlinear equations.

\paragraph{Paper Organization}
In Section~\ref{sec-notation}, we start with notation and preliminaries throughout this paper. In Section \ref{sec-greedy Q update}, we first propose greedy and random quasi-Newton methods for quadratic saddle point problem which enjoys local superlinear convergence. Then we extend it to solve general strongly-convex-strongly-concave saddle point problems.
In Section \ref{sec:extension}, we show our theory also can be applied to solve 
more general non-linear equations and give the corresponding convergence analysis. We conclude our work in Section \ref{sec:conclusion}. IIn Section \ref{sec:exp}, we provide numerical experiments to validate our algorithms on popular machine learning models. All proofs and experiment details are deferred to appendix.
\begin{table*}[!t]
\label{table:convergence}	
    \centering\renewcommand{\arraystretch}{1.8}
	\begin{tabular}{|c|c|c|c|}
		\hline
       \multicolumn{2}{|c|}{ Algorithms} & Upper Bound of $\lambda_{k+k_0}$ & $k_0$\\
		\hline
		\multirow{2}{*}{Broyden (Algorithm \ref{alg:BroydenSPP})}& Greedy & $\left(1-\frac{1}{n\kappa^2}\right)^{k(k-1)/2}\left(\frac{1}{2}\right)^k\left(1-\frac{1}{4\kappa^2}\right)^{k_0}$ & $\OM\left(n\kappa^2\ln(n\kappa)\right)$ \\ 
		&Random& \!$\left(1-\frac{1}{n\kappa^2+1}\right)^{k(k-1)/2}\left(\frac{1}{2}\right)^k\left(1-\frac{1}{4\kappa^2}\right)^{k_0}$\! & $\OM\left(n\kappa^2\ln(\frac{n\kappa}{\delta})\right)$\\
		\hline
		\multirow{2}{*}{\!BFGS/SR1 (Algorithm \ref{alg:BFGSSPP}/\ref{alg:SR1SPP})}\! & Greedy & $\left(1-\frac{1}{n}\right)^{k(k-1)/2}\left(\frac{1}{2}\right)^k\left(1-\frac{1}{4\kappa^2}\right)^{k_0}$ & $\OM\left(\max\left\{n,\kappa^2\right\}\ln(n\kappa)\right)$ \\ 
		&Random& $\left(1-\frac{1}{n+1}\right)^{k(k-1)/2}\left(\frac{1}{2}\right)^k\left(1-\frac{1}{4\kappa^2}\right)^{k_0}$ & $\OM\left(\max\left\{n,\kappa^2\right\}\ln\left(\frac{n\kappa}{\delta}\right)\right)$\\
		\hline

	
	\end{tabular}
	\caption{We summarize the convergence behaviors of proposed algorithms for solving saddle point problem in the view of gradient norm $\lambda_{k+k_0}\defeq\Norm{\nabla f(\z_{k+k_0})}$ after $(k+k_0)$ iterations, where $n\defeq n_x+n_y$ is dimensions of the problem and $\kappa$ is the condition number. The results come from Corollary \ref{co:Broydenconvergence-g} and \ref{co:Broydenconvergence-random} and the upper 
	bounds of random algorithms holds with high probability at least $1-\delta$.}
	\label{table:res}
\end{table*}

\section{Notation and Preliminaries}\label{sec-notation}

We use $\Norm{\cdot}$ to present spectral norm and Euclidean norm of matrix and vector respectively. We denote the standard basis for $\BR^n$ by $\{\ve_1,\dots,\ve_n\}$ and let $\I$ be the identity matrix. The trace of a square matrix is denoted by $\text{tr}(\cdot)$. 
Given two positive definite matrices $\G$ and $\H$, we define their inner product as  
\begin{equation*}
    \langle\G,\H\rangle\defeq\text{tr}(\G\H).
\end{equation*}
We introduce the following notation to measure how well does matrix $\G$ approximate matrix $\H$:
\begin{equation}
\label{eq-def-sigma-H}
    \sigma_{\H}(\G) \defeq \langle\H^{-1},\G-\H\rangle = \langle\H^{-1},\G\rangle-n.
\end{equation}
If we further suppose $\G\succeq\H$, it holds that
\begin{equation}
\label{eq-sigmaproperty-1}
    \G-\H \preceq \inner{\H^{-1}}{\G-\H} \H=\sigma_{\H}(\G) \H
\end{equation}
by \citet{rodomanov2021greedy}.

Using the notation of problem~\eqref{eq-general-obj}, we let $\z=[\x;\y]\in\BR^n$ where $n=n_x+n_y$ and denote the gradient and Hessian matrix of $f$ at $(\vx,\vy)$ as 
\begin{equation*}
\g(\vz)\defeq\begin{bmatrix}
\nabla_{\vx} f(\vx,\vy) \\ \nabla_{\vy} f(\vx,\vy)
\end{bmatrix} \in \BR^{n}
\quad\text{and}\quad
    \hH(\vz) \defeq \nabla^2 f(\vx,\vy) = \begin{bmatrix}
\nabla_{\vx\vx}^2 f(\vx,\vy) & \nabla_{\vx\vy}^2 f(\vx,\vy)  \\[0.1cm]
\nabla_{\vy\vx}^2 f(\vx,\vy) & \nabla_{\vy\vy}^2 f(\vx,\vy)
\end{bmatrix} \in \BR^{n\times n}.
\end{equation*}

We suppose the saddle point problem~\eqref{eq-general-obj} satisfies the following assumptions.

\begin{assumption}\label{assbasic}
The objective function $f(\vx,\vy)$ is twice differentiable and has $L$-Lipschitz continuous gradient and $L_2$-Lipschitz continuous Hessian , i.e., there exists constants $L>0$ and $L_2> 0$ such that 
\begin{align}
\Norm{\g(\vz) - \g(\vz')} \leq L\Norm{\vz-\vz'} \label{eq-obj-smooth}
\end{align}
and
\begin{align}
\Norm{\hH(\vz) - \hH(\vz')} \leq L_2\Norm{\vz-\vz'}.  \label{eq-Hessian-continous}
\end{align}
for any $\vz=[\x; \y], \vz'=[\x'; \y']\in\BR^n$.
\end{assumption}

\begin{assumption}
\label{assbasic2}
The objective function $f(\vx,\vy)$ is twice differentiable, $\mu$-strongly-convex in $\vx$ and $\mu$-strongly-concave in $\vy$, i.e., there exists $\mu>0$ such that $\nabla_{\vx\vx}^2 f(\vx,\vy) \succeq \mu \I$ and $\nabla_{\vy\vy}^2 f(\vx,\vy) \preceq -\mu \I$ for any $(\vx,\vy), (\vx',\vy')\in\BR^{n_x}\times\BR^{n_y}$.
\end{assumption}
\noindent Note that inequality~\eqref{eq-obj-smooth} means the spectral norm of Hessian matrix $\hH(\z)$ can be upper bounded, that is
\begin{align}
 \|\hH(\z)\|\leq L.   
\end{align}
for all $\z\in\BR^n$. Additionally, the condition number of the objective function is defined as
\begin{align*}
    \kappa \defeq \frac{L}{\mu}.
\end{align*}

\section{Quasi-Newton Methods for Saddle Point Problems}\label{sec-greedy Q update}

The update rule of standard Newton's method for solving problem~\eqref{eq-general-obj} can be written as
\begin{equation}
\label{eq-newton update}
    \z_{+}=\z -\big(\hH(\z)\big)^{-1} \g(\z). 
\end{equation}
This iteration scheme has quadratic local convergence, but solving linear system~\eqref{eq-newton update} takes $\fO\left(n^3\right)$ time complexity. For convex minimization, quasi-Newton methods including BFGS/SR1 \cite{broyden1970convergence2,broyden1970convergence,shanno1970conditioning,broyden1967quasi,davidon1991variable} and their variants~\cite{lin2021faster,ye2021explicit,rodomanov2021greedy,rodomanov2021new} focus on approximating the Hessian and reduce the computational cost to $\fO\left(n^2\right)$ for each round. However, all of these algorithms and related convergence analysis are based on the assumption that the Hessian matrix is positive definite, which is not suitable for our saddle point problems since $\hH(\z)$ is indefinite.

We introduce the auxiliary matrix $\H(\z)$ be the square of Hessian
\begin{equation*}
  \H(\z) \defeq \left(\hH(\z)\right)^2.
\end{equation*}
The following lemma means $\H(\z)$ is always positive definite. 
\begin{lemma}\label{lm-matrix-neq0}
Under Assumption~\ref{assbasic} and \ref{assbasic2}, we have
\begin{eqnarray}
  \label{eq6}
    \mu^2\I \preceq  \H(\z) \preceq  L^2 \I
\end{eqnarray}
for all $\vz\in\BR^n$, where $\H(\z)=\big(\hH(\z)\big)^2$.
\end{lemma}
\noindent Hence, we can reformulate the update of Newton's method \eqref{eq-newton update} by
\begin{align}
\begin{split}\label{eq-newton update2}
     \z_{+}
= & \z-\left[\big(\hH(\z)\big)^2\right]^{-1}\hH(\z)\g(\z)  \\
= & \z-\H(\z)^{-1}\hH(\z)\g(\z). 
\end{split}
\end{align}
Then it is natural to characterize the second-order information by estimating the auxiliary matrix $\H(\z)$, rather than the indefinite Hessian $\hH(\vz)$.  If we have obtained a symmetric positive definite matrix $\G\in\BR^{n\times n}$ as an estimator for $\H(\z)$, the update rule of \eqref{eq-newton update2} can be approximated by
\begin{equation}\label{eq-newton update3}
     \z_{+}=\z - \G^{-1} \hH(\z)\g(\z).
\end{equation}
The remainder of this section introduce several strategies to construct $\G$, resulting the quasi-Newton methods for saddle point problem with local superlinear convergence. 
We should point out the implementation of iteration \eqref{eq-newton update3} is unnecessary to construct Hessian matrix $\hH(\z)$ explicitly, since we are only interested in the Hessian-vector product $\hH(\z)\g(\z)$, which can be computed efficiently~\cite{pearlmutter1994fast,Schraudolph2002Fast}.

\subsection{The Broyden Family Updates}
\label{sec:Broyden update}
We first introduce the Broyden family~\cite[Section 6.3]{NoceWrig06} of quasi-Newton updates for approximating an positive definite matrix $\H\in\RB^{n\times n}$ by using the information of current estimator $\G\in\RB^{n\times n}$.

\begin{definition}\label{defbroyd}
Suppose two positive definite matrices $\H, \G\in\BR^{n\times n}$ satisfy $\H \preceq \G$. For any $\u \in \RB^{n}$, if $\G\u=\H\u$, we define
$\rm{Broyd}_{\tau}(\G,\H,\u)\overset{\rm{def}}{=}\H$. Otherwise, we define
\begin{align}\label{eqbroyd}
    \rm{Broyd}_{\tau}(\G,\H,\u)\overset{\rm{def}}{=} ~~& \tau \left[\G-\frac{\H\u\u^{\top}\G+\G\u\u^{\top}\H}{\u^\top\H\u}+\left(\frac{\u^\top\G\u}{\u^\top\H\u}+1\right)\frac{\H\u\u^{\top}\H}{\u^\top\H\u}\right] \\
    ~~& + (1-\tau)\left[\G-\frac{(\G-\H)\u\u^{\top}(\G-\H)}{\u^\top(\G-\H)\u} \right]. \nonumber
\end{align}
\end{definition}

The different choices of parameter $\tau$ for formula \eqref{eqbroyd} contain several popular quasi-Newton updates:
\begin{itemize}
\item For $\tau=\dfrac{\u^{\top}\H\u}{\u^{\top}\G\u}\in[0,1]$, it corresponds to BFGS update
    \begin{align}\label{eq-BFGS-update}
        \text{BFGS}(\G,\H,\u)\defeq \G-\frac{\G\u\u^{\top}\G}{\u^{\top}\G\u}+\frac{\H\u\u^{\top}\H}{\u^{\top}\H\u}.
    \end{align}
\item For $\tau=0$, it corresponds to SR1 update
    \begin{align}\label{eq:sr1-update}
        \rm{SR1}(\G,\H,\u)\overset{\rm{def}}{=}  \G-\frac{(\G-\H)\u\u^\top(\G-\H)}{\u^\top(\G-\H)\u}.
    \end{align}
\end{itemize}
 
The general Broyden family update as Definition \ref{defbroyd} has the following properties.
\begin{lemma}[{\cite[Lemma 2.1 and Lemma 2.2]{rodomanov2021greedy}}]\label{lmBroydupdate} Suppose two positive definite matrices $\H, \G\in\BR^{n\times n}$ satisfy $\H \preceq \G \preceq \eta \H$ for some $\eta\geq 1$, then for any $\u \in \RB^{n}$ and $\tau_1<\tau_2$, we have
\begin{equation*}
    {\rm Broyd}_{\tau_1}(\G,\H,\u)\preceq {\rm Broyd}_{\tau_2}(\G,\H,\u).
\end{equation*}
Additionally, for any $\tau \in [0,1]$, we have
\begin{equation}\label{eqgreedypartial}
    \H\preceq {\rm Broyd}_{\tau}(\G,\H,\u) \preceq \eta \H.
\end{equation}
\end{lemma}

We first introduce greedy update method~\cite{rodomanov2021greedy} by choosing $\u$ as 
\begin{align}
\begin{split}\label{eqgreedydirect}
\u_{\H}(\G) 
&\defeq\argmax_{\u \in \{\e_1,\cdots,\e_n\}}\frac{\u^\top (\G-\H)\u}{\u^{\top}\H\u}  
 = \argmax_{\u \in \{\e_1,\cdots,\e_n\}}\frac{\u^{\top}\G\u}{\u^{\top}\H\u};
\end{split}
\end{align}
and random update method \citet{rodomanov2021greedy,lin2021faster} by choosing $\u$ as 
\begin{eqnarray}\label{eq:random-uchose}	
\u\sim\mathcal{N}(\0, \I) \qquad\text{or}\qquad \u \sim \mathrm{Unif}\left(\mathcal{S}^{n-1}\right).
\end{eqnarray}
The following lemma shows applying update rule \eqref{eqbroyd} with formula \eqref{eqgreedydirect} or \eqref{eq:random-uchose} leads to a new estimator with tighter error bound in the measure of $\sigma_{\H}(\cdot)$.
\begin{lemma}
[{\cite[Theorem 6]{lin2021faster}}]
\label{thmbroydenupdate}
Suppose two positive definite matrices $\H, \G\in\BR^{n\times n}$ satisfy $\H\preceq \G\preceq \kappa^2\H$. Let $\G_+={\rm Broyd}_{\tau}\left(\G,\H,\u)\right)$, where $\u$ is chosen by greedy method  $\u=\u_{\H}(\G)$ as \eqref{eqgreedydirect} or random method as \eqref{eq:random-uchose}, then for any $\tau \in [0,1]$, we have
\begin{equation}
\label{eqbroythm1}
    \EB\left[\sigma_{\H}(\G_+)\right]\leq\left(1-\frac{1}{n\kappa^2}\right)\sigma_{\H}(\G).
\end{equation}
\end{lemma}
\noindent Note that we define $\EBP{\sigma_{\H}(\G_+)}=\sigma_{\H}(\G_+)$ for greedy method.

For specific Broyden family updates BFGS and SR1 shown in \eqref{eq-BFGS-update} and \eqref{eq:sr1-update}, 
we can replace \eqref{eqgreedydirect} by scaling greedy direction~\cite{lin2021faster}, which leads to a better convergence result. 
Concretely, for BFGS method, we first find $\LL$ such that $\G^{-1}=\LL^{\top}\LL$, where $\LL$ is an upper triangular matrix. 
This step can be implemented with $\fO(n^2)$ complexity \cite[Proposition 15]{lin2021faster}. We present the subroutine for factorizing $\G^{-1}$ in Algorithm \ref{alg:L update} and give its detailed implementation in Appendix \ref{faster}.

Then we use direction $\LL^{\top}\tilde{\u}_{\H}(\LL)$ for greedy BFGS update, where 
\begin{equation}\label{eqgreedydirect2}
    \tilde{\u}_{\H}(\LL)\overset{{\rm def}}{=}\argmax_{\u \in \{\e_1,\cdots,\e_n\}} \u^{\top}\LL^{-\top}\H^{-1}\LL^{-1}\u.
\end{equation}
For greedy SR1 method, we choose the direction by
\begin{equation}\label{eq:sr1direction}
    \bu_{\H}(\G)\overset{\rm{def}}{=}\argmax_{\u \in \{\e_1,\cdots,\e_n\}} \frac{\u^{\top}(\G-\H)^2\u}{\u^{\top}(\G-\H)\u}.
\end{equation}
Applying the BFGS update rule \eqref{eq-BFGS-update} with formula \eqref{eq:random-uchose} or \eqref{eqgreedydirect2}, we obtain a condition-number-free result as follows.
\begin{lemma}[{\cite[Theorem 13]{lin2021faster}}]
\label{thmBFGSupdate}
Suppose two positive definite matrices $\H, \G\in\BR^{n\times n}$ satisfy $\H\preceq\G$. Let $\G_+={\rm BFGS}\left(\G,\H,\LL^{\top}\u\right)$,
where $\u$ is chosen by greedy method $\u=\tilde{\u}_{\H}(\LL)$ in \eqref{eqgreedydirect2} or the random method in \eqref{eq:random-uchose}
and $\LL$ is an upper triangular matrix such that $\G^{-1}=\LL^{\top}\LL$. Then we have
\begin{equation}
\label{eqbfgsthm1}
    \EBP{\sigma_{\H}(\G_+)} \leq \left(1-\frac{1}{n}\right)\sigma_{\H}(\G).
\end{equation}
\end{lemma}
\begin{remark}
\label{rm:BFGS}
Note that the step of conducting $\G^{-1}=\LL^{\top}\LL$ requires QR decomposition of rank-1 change matrix which requires $\fO(26n^2)$ flops \citep[Section 12.5.1]{golub1996matrix}. We do not recommend using this BFGS update strategy in practice when $n$ is large.
\end{remark}

\noindent The effect of SR1 update can be characterized by the following measure
\begin{equation*}
    \tau_{\H}(\G)\overset{\rm{def}}{=}\rm{tr}(\G-\H).
\end{equation*}
Applying the SR1 update rule \eqref{eq:sr1-update} with formula \eqref{eq:random-uchose} or \eqref{eq:sr1direction}, we also hold a condition-number-free result.
\begin{lemma}[{\cite[Theorem 12]{lin2021faster}}]
\label{thm:sr1update}
Suppose two positive definite matrices $\H, \G\in\BR^{n\times n}$ satisfy $\H\preceq\G$. Let $\G_+={\rm SR1}\left(\G,\H,\u\right)$, where $\u$ is chosen by greedy method $\u=\bar{\u}_{\H}(\G)$ in \eqref{eq:sr1direction} or the random method in \eqref{eq:random-uchose}. Then we have
\begin{equation}
\label{eq:sr1tau}
    \EBP{\tau_{\H}\left(\G_+\right)} \leq \left(1-\frac{1}{n}\right)\tau_{\H}(\G).
\end{equation}
\end{lemma}

\begin{algorithm}[t]
\caption{Fast-Cholesky$(\H,\LL,\vu)$}\label{alg:L update}
\begin{algorithmic}[1]
\STATE \textbf{Input:} positive definite matrix $\H\in\BR^{n\times n}$, upper triangular matrix $\LL\in\BR^n$, greedy direction $\u\in\BR^n$ \\[0.2cm]
\STATE ${\displaystyle [\Q, \R]={\rm QR}\left(\LL\left(\I-\frac{\H\u\u^{\top}}{\u^{\top}\H\u}\right)\right)}$ \\[0.3cm]
\STATE $\v=\dfrac{\u}{\sqrt{\u^{\top}\H\u}}$ \\[0.25cm]
\STATE $[\Q', \R'] = {\rm QR}\left(\begin{bmatrix} \vv^\top \\ \R \end{bmatrix}\right)$ \\[0.15cm]
\STATE \textbf{Output:} $\hat\LL=\R'$
\end{algorithmic}
\end{algorithm}

\begin{algorithm}[t]
\caption{Greedy/Random Broyden Class Method for Quadratic Problems}
 \label{alg:BroydenQuadratic}
\begin{algorithmic}[1]
    \STATE \textbf{Input:} $\G_0=L^2\I$ and $\tau_k\in[0,1]$ \\[0.15cm]
	\STATE \textbf{for} $k=0,1,\dots$  \\[0.15cm]
		\STATE \quad $\z_{k+1}=\z_k-\G_k^{-1}\hH\g(\z_{k})$ \\[0.15cm]
	\STATE \quad Choose $\u_k$ from
	\begin{itemize}
	    \item Option I (greedy method): $\u_k=\u_{\H}(\G_k)$
	    \item Option II (random method): $\u_k\sim\mathcal{N}(\0, \I)$~~~or~~~$\u_k \sim \mathrm{Unif}\left(\mathcal{S}^{n-1}\right)$
	\end{itemize}
 	\STATE \quad $\G_{k+1}={\rm Broyd}_{\tau_k}\left(\G_k,\H,\u_k\right)$ \\[0.15cm]
 	\STATE \textbf{end for}
\end{algorithmic}
\end{algorithm}

\begin{algorithm}[t]
\caption{Greedy/Random BFGS Method for Quadratic Problem} \label{alg:BFGSQuadratic}
\begin{algorithmic}[1]
\STATE \textbf{Input:} $\G_0=L^2 \I$ and $\LL_0=L^{-1}\I$
\STATE \textbf{for} $k=0,1\dots$ \\[0.15cm]
\STATE \quad $\z_{k+1}=\z_k-\G_k^{-1}\hH\g(\z_k)$ \\[0.15cm]
\STATE \quad Choose $\tilde{\u}_k$ from
	\begin{itemize}
	    \item Option I (greedy method): $\tilde{\u}_k=\u_{\H}(\LL_k)$
	    \item Option II (random method): $\tilde{\u}_k\sim\mathcal{N}(\0, \I)$~~~or~~~$\tilde{\u}_k \sim \mathrm{Unif}\left(\mathcal{S}^{n-1}\right)$
	\end{itemize}
\STATE \quad $\u_k=\LL_{k}^{\top}\tilde{\u}_k$\\[0.15cm]
\STATE \quad $\G_{k+1}= {\rm BFGS}\left(\G_k,\H,\u_k\right)$ \\[0.15cm]
\STATE \quad $\LL_{k+1}=$ Fast-Cholesky$\left(\H,\LL_k,\vu_k\right)$ \\[0.15cm]
\STATE \textbf{end for}
\end{algorithmic}

\end{algorithm}
\begin{algorithm}[t]
\caption{Greedy/Random SR1 method for Quadratic Problem}\label{alg:SR1Quadratic}
\begin{algorithmic}[1]
\STATE \textbf{Input:} $\G_0=L^2 \I$ and $K\leq n$ \\[0.15cm]
\STATE \textbf{for} $k=0,1\dots,K$ \\[0.15cm]
\STATE \quad $\z_{k+1}=\z_k-\G_k^{-1}\hH\g(\z_k)$ \\[0.15cm]
	\begin{itemize}
	    \item Option I (greedy method): $\u_k=\bar{\u}_{\H}(\G_k)$
	    \item Option II (random method): $\u_k\sim\mathcal{N}(\0, \I)$~~~or~~~$\u_k \sim \mathrm{Unif}\left(\mathcal{S}^{n-1}\right)$
	\end{itemize}
\STATE \quad $\G_{k+1}= {\rm SR1}\big(\G_k,\H,\u_k\big)$ \\[0.15cm]
\STATE \textbf{end for}
\end{algorithmic}
\end{algorithm}

\subsection{Algorithms for Quadratic Saddle Point Problems} \label{sec-quadratic}

Then we consider solving quadratic saddle point problem of the form
\begin{equation}
\label{eqquadraticobj}
    \min_{\vx\in\BR^{n_x}}\max_{\vy\in\BR^{n_y}} f(\vx,\vy) \defeq \frac{1}{2} \begin{bmatrix} \vx^\top~\vy^\top \end{bmatrix} \A \begin{bmatrix} \vx \\ \vy \end{bmatrix} 
    - \vb^{\top}\begin{bmatrix} \vx \\ \vy \end{bmatrix},
\end{equation}
where $\vb\in\BR^{n}$, $\A\in\BR^{n\times n}$ is symmetric and $n=n_x+n_y$. We suppose $\A$ could be partitioned as
\begin{align*}
\A = \begin{bmatrix} 
\A_{\x\x} & \A_{\x\y} \\  
\A_{\y\x} & \A_{\y\y} \\  
\end{bmatrix}\in\BR^{n\times n}
\end{align*}
where the sub-matrices $\A_{\x\x}\in\BR^{n_x\times n_x}$, $\A_{\x\y}\in\BR^{n_x\times n_y}$, $\A_{\y\x}\in\BR^{n_y\times n_x}$ and $\A_{\y\y}\in\BR^{n_y\times n_y}$ satisfy $\A_{\x\x}\succeq \mu\I$, $\A_{\y\y} \preceq -\mu\I$ and $\|\A\|\leq L$. Recall the condition number is defined as $\kappa\defeq L/\mu$. Using notations introduced in Section \ref{sec-notation}, we have
\begin{align*}
\vz = [\vx; \vy], \quad 
\g(\z) = \A\z-\vb, \quad 
\hH \defeq \hH(\z)=\A  \quad \text{and} \quad
\H \defeq \H(\z)=\A^2.
\end{align*}

We present the detailed procedure of greedy and random quasi-Newton methods for quadratic saddle point problem by using Broyden family update, BFGS update and SR1 update in Algorithm \ref{alg:BroydenQuadratic}, \ref{alg:BFGSQuadratic} and \ref{alg:SR1Quadratic} respectively.
We define $\lambda_k$ as the Euclidean distance from $\z_k$ to the saddle point $\z^*$ for our convergence analysis, that is
\begin{equation*}
    \lambda_k \defeq \|\nabla f(\z_k)\|.
\end{equation*}
The definition of $\lambda_k$ in this paper is different from the measure used in convex optimization \cite{rodomanov2021greedy,lin2021faster}\footnote{In later section, we will see the measure $\lambda_k\defeq\Norm{\nabla f(\z_k)}$ is suitable to convergence analysis of quasi-Newton methods for saddle point problems.}, but it also holds the similar property as follows.

\begin{lemma}\label{lmquadraticlamb}
Assume we have $\eta_k \geq 1$ and $\G_k\in\BR^{n\times n}$ such that $\H\preceq\G_k\preceq\eta_k\H$ for Algorithm \ref{alg:BroydenQuadratic}, \ref{alg:BFGSQuadratic} and \ref{alg:SR1Quadratic}, then we have
\begin{equation*}
    \lambda_{k+1} \leq \left(1-\frac{1}{\eta_k}\right)\lambda_k.
\end{equation*}
\end{lemma}

\noindent The next theorem states the assumptions of Lemma \ref{lmquadraticlamb} always holds with $\eta_k=\kappa^2\geq 1$, which means $\lambda_k$ converges to 0 linearly.
\begin{theorem}
\label{thm-quadratic-lambda}
For all $k\geq0$, Algorithm \ref{alg:BroydenQuadratic},  \ref{alg:BFGSQuadratic} and \ref{alg:SR1Quadratic} hold that
\begin{equation}
\label{eqquadraticmatrix1}
    \H\preceq\G_k\preceq\kappa^2\H,
\end{equation}
and
\begin{equation}
\label{eqlambda1}
    \lambda_{k}\leq \left(1-\frac{1}{\kappa^2}\right)^k\lambda_0.
\end{equation}
\end{theorem}

Lemma \ref{lmquadraticlamb} also implies superlinear convergence can be obtained if there exists $\eta_k$ which converges to 1.
Applying Lemma \ref{thmbroydenupdate}, \ref{thmBFGSupdate} and \ref{thm:sr1update}, we can show it holds for proposed algorithms.
\begin{theorem} 
\label{thm:BroydenQuadratic}
Solving quadratic saddle point problem \eqref{eqquadraticobj} by proposed quasi-Newton Algorithm, we have the following results:
\begin{enumerate}
\item For Broyden family method (Algorithm \ref{alg:BroydenQuadratic}), we have
\begin{align}
\label{eq-qua-g-fth-result-2}
    \EBP{\frac{\lambda_{k+1} }{\lambda_k}} \leq \left(1-\frac{1}{n\kappa^2}\right)^kn\kappa^2
    \quad \text{for all}~ k \geq 0.
\end{align}
\item For BFGS method (Algorithm \ref{alg:BFGSQuadratic}), we have
\begin{align*}
    \EBP{\frac{\lambda_{k+1} }{\lambda_k}} \leq \left(1-\frac{1}{n}\right)^kn\kappa^2
    \quad \text{for all}~ k \geq 0. 
\end{align*}
\item For SR1 method (Algorithm \ref{alg:SR1Quadratic}), we have
\begin{align*}
    \EBP{\frac{\lambda_{k+1} }{\lambda_k}} \leq \left(1-\frac{k}{n}\right)n\kappa^4
    \quad \text{for all}~ 0\leq k\leq n.
\end{align*}
\end{enumerate}
\end{theorem}


Combing the results of Theorem \ref{thm-quadratic-lambda} and
\ref{thm:BroydenQuadratic}, we achieve the two-stages convergence behavior, that is, the algorithm has global linear convergence and local superlinear convergence. The formal description is summarized as follows.

\begin{corollary}
\label{co:quadraticresult}
Solving quadratic saddle point problem \eqref{eqquadraticobj} by proposed greedy quasi-Newton algorithms, we have the following results:
\begin{enumerate}
\item Using greedy Broyden family method (Algorithm \ref{alg:BroydenQuadratic}), we have
\begin{align}
    \lambda_{k_0+k} \leq \left(1-\frac{1}{n\kappa^2}\right)^{\frac{k(k-1)}{2}}\left(\frac{1}{2}\right)^k\left(1-\frac{1}{\kappa^2}\right)^{k_0}\lambda_0 \quad\text{for all}~k>0~\text{and}~k_0=\OM\left(n\kappa^2 \ln(n\kappa^2)\right).
\end{align}
\item Using greedy BFGS method (Algorithm \ref{alg:BFGSQuadratic}), we have
\begin{align}
\label{eq-q-f-convergence-result-faster}
    \lambda_{k_0+k} \leq \left(1-\frac{1}{n}\right)^{\frac{k(k-1)}{2}}\left(\frac{1}{2}\right)^k\left(1-\frac{1}{\kappa^2}\right)^{k_0}\lambda_0 \quad\text{for all}~k>0~\text{and}~k_0=\OM\left(n \ln(n\kappa^2)\right).
\end{align}
\item Using the greedy SR1 method (Algorithm \ref{alg:SR1Quadratic}), we have
\begin{equation}
\lambda_{k+k_0}\leq \frac{(n-k_0-k+1)!}{(n-k_0+1)!}\left(\frac{1}{2(n-k_0)}\right)^{k}\left(1-\frac{1}{\kappa^2}\right)^{k_0}\lambda_0
\end{equation}
for all $n-k_0+1\geq k>0$ and $k_0=\left\lceil\left(1-\dfrac{1}{2n\kappa^4}\right)n\right\rceil$.
\end{enumerate}
\end{corollary}

\begin{corollary}
\label{co:quadraticresult-random}
Solving quadratic saddle point problem \eqref{eqquadraticobj} by proposed random quasi-Newton algorithms, then with probability $1-\delta$ for any $\delta\in(0,1)$, we have the following results:
\begin{enumerate}
\item Using random Broyden family method (Algorithm \ref{alg:BroydenQuadratic}), we have
\begin{align}
    \lambda_{k_0+k} \leq \left(1-\frac{1}{n\kappa^2+1}\right)^{\frac{k(k-1)}{2}}\left(\frac{1}{2}\right)^k\left(1-\frac{1}{\kappa^2}\right)^{k_0}\lambda_0 \quad\text{for all}~k>0~\text{and}~k_0=\OM\left(n\kappa^2 \ln(n\kappa/\delta)\right).
\end{align}
\item Using random BFGS method (Algorithm \ref{alg:BFGSQuadratic}), we have
\begin{align}
\label{eq-q-f-convergence-result-faster-random}
    \lambda_{k_0+k} \leq \left(1-\frac{1}{n+1}\right)^{\frac{k(k-1)}{2}}\left(\frac{1}{2}\right)^k\left(1-\frac{1}{\kappa^2}\right)^{k_0}\lambda_0 \quad\text{for all}~k>0~\text{and}~k_0=\OM\left(n \ln(n\kappa/\delta)\right).
\end{align}

\item Using the random SR1 method (Algorithm \ref{alg:SR1Quadratic}), we have
\begin{equation}
\lambda_{k+k_0}\leq \frac{(n-k_0-k+1)!}{(n-k_0+1)!}\left(\frac{1}{2(n-k_0)}\right)^{k}\left(1-\frac{1}{\kappa^2}\right)^{k_0}\lambda_0
\end{equation}
for all $n-k_0+1\geq k>0$ and $k_0=\left\lceil\left(1-\delta/(4n^2(n+1)\kappa^4)\right)n\right\rceil$.
\end{enumerate}
\end{corollary}

\subsection{Algorithms for General Saddle Point Problems}
\label{sec-general}
In this section, we consider the general saddle point problem 
\begin{equation*}
    \min_{\x \in \RB^{n_x}}\max_{\y \in \RB^{n_y}} f(\vx,\vy).
\end{equation*}
where $f(\x,\y)$ satisfies Assumption \ref{assbasic} and \ref{assbasic2}. We target to propose quasi-Newton methods for solving above problem with local superlinear convergence and $\fO\left(n^2\right)$ time complexity for each iteration.

\subsubsection{Algorithms}
\label{sec-general-alg}

The key idea of designing quasi-Newton methods for saddle point problems is characterizing the second-order information by approximating the auxiliary matrix $\H(\z)\defeq\big(\hH(\z)\big)^2$. Note that we have supposed the Hessian of $f$ is Lipschitz continuous and bounded in Assumption \ref{assbasic} and \ref{assbasic2}, which means the auxiliary matrix operator $\H(\z)$ is also Lipschitz continuous.

\begin{lemma}\label{lm-H-continous}
Under Assumption \ref{assbasic} and \ref{assbasic2}, we have
$\H(\z)$ is $2LL_2$-Lipschitz continuous, that is 
\begin{equation}\label{eq10}
    \|\H(\vz)-\H(\vz')\| \leq 2L_2L \|\vz-\vz'\|.
\end{equation}
\end{lemma}

Combining Lemma \ref{lm-matrix-neq0} and \ref{lm-H-continous}, we achieve the following properties of $\H(\z)$, which analogizes the strongly self-concordance in convex optimization~\cite{rodomanov2021greedy}.
\begin{lemma}
\label{lm-self-concordant}
Under Assumption \ref{assbasic} and \ref{assbasic2}, the auxiliary matrix operator $\H(\cdot)$ satisfies
\begin{equation}    \label{eq-self-concordant}
    \H(\z)-\H(\z') \preceq M\|\z-\z'\|\H(\w),
\end{equation}
for all $\z,\z',\w \in\BR^n$, where $M=2\kappa^2L_2/L$.
\end{lemma}
\begin{corollary}
\label{co:concordant}
Let $\vz,\vz' \in \BR^n$ and $r=\|\vz'-\vz\|$. Suppose the objective function $f$ satisfies Assumption \ref{assbasic} and \ref{assbasic2}, then the auxiliary matrix operator $\H(\cdot)$ holds that
\begin{eqnarray}
\label{eq-H-concordant-app}
\frac{\H(\vz)}{1+Mr} \preceq \H(\vz') \preceq(1+Mr) \H(\vz),
\end{eqnarray}
for all $\z,\z' \in \BR^n$, where $M=2\kappa^2L_2/L$.
\end{corollary}

Different with the quadratic case, auxiliary matrix $\H(\z)$ is not fixed for general saddle point problem. Based on the smoothness of $\H(\z)$, we apply Corollary \ref{co:concordant} to generalize Lemma \ref{lmBroydupdate} as follows.

\begin{lemma}
\label{lm-G-H-neq}
Let $\z \in \RB^{n}$ and $\G\in \RB^{n\times n}$ be a positive definite matrix such that
\begin{equation}
\label{eq-G-H-neq-0}
    \H(\z)\preceq \G\preceq \eta\H(\z),
\end{equation}
for some $\eta \geq 1$. We additionally define $\z_{+}\in\BR^n$ and $r=\|\z_{+}-\z\|$, then we have
\begin{equation}
 \label{eq-G-update}
    \tilde{\G}\overset{\rm{def}}{=}(1+Mr)\G\succeq\H(\z_{+})
\end{equation}
and
\begin{equation}
\label{eq-G-H-neq}
    \H(\z_{+})\preceq {\rm Broyd}_{\tau}(\tilde{\G},\H(\z_{+}),\u) \preceq (1+Mr)^2\eta\H(\z_{+})
\end{equation}
for all $\u \in \BR^n$, $\tau \in [0,1]$ and $M=2\kappa^2L_2/L$.
\end{lemma}
  
The relationship \eqref{eq-G-H-neq} implies it is reasonable to establish the algorithms by update rule 
\begin{align}\label{update G general}
 \G_{k+1}={\rm Broyd}_{\tau_k} \left(\tilde{\G}_k,\H_{k+1},\u_k\right)   
\end{align}
with $\tilde{\G}_{k}=(1+Mr_k)\G_k$ and $r_k=\|\z_{k+1}-\z_{k}\|$. 
Similarly, we can also achieve $\G_{k+1}$ by such $\tilde{\G}_k$ for specific BFGS and SR1 update.
Combining iteration \eqref{eq-newton update3} with \eqref{update G general}, we propose the quasi-Newton methods for general strongly-convex-strongly-concave saddle point problems. The details is shown in Algorithm \ref{alg:BroydenSPP}, \ref{alg:BFGSSPP} and \ref{alg:SR1SPP} for greedy Broyden family, BFGS and SR1 updates respectively. 

\begin{algorithm}[t]
\caption{Greedy/Random Broyden Method for General Case}
 \label{alg:BroydenSPP}
	\begin{algorithmic}[1]
        \STATE \textbf{Input:} $\G_0\succeq \H$, $\tau_k \in[0,1]$ and $M\geq0$. \\[0.15cm]
		\STATE \textbf{for} $k \geq 0$ \\[0.15cm] 
 		\STATE \quad 
 		    $\z_{k+1}=\z_k-\G_k^{-1}\hH_k\g_k$\\[0.15cm]
 		\STATE \quad $r_k=\|\z_{k+1}-\z_{k}\|$ \\[0.15cm]
 		\STATE \quad 
 		     $\tilde{\G}_{k}=(1+Mr_k)\G_k$ \\[0.15cm]
	\STATE \quad Choose $\u_k$ from
	\begin{itemize}
	    \item Option I (greedy method): $\u_k=\u_{\H_{k+1}}(\tilde{\G}_k)$
	    \item Option II (random method): $\u_k\sim\mathcal{N}(\0, \I)$~~~or~~~$\u_k \sim \mathrm{Unif}\left(\mathcal{S}^{n-1}\right)$
	\end{itemize}
 		\STATE \quad 
 		    $\G_{k+1}={\rm Broyd}_{\tau_k} (\tilde{\G}_k,\H_{k+1},\u_k)$\\[0.15cm]
 		\STATE \textbf{end for}
	\end{algorithmic}
\end{algorithm}

\begin{algorithm}[t]
\caption{Greedy/Random BFGS Method for General Case}\label{alg:BFGSSPP}
\begin{algorithmic}[1]
\STATE \textbf{Input:} $\G_0\succeq\H$ and $M\geq0$. \\[0.15cm] 
\STATE $\LL_0=\G_0^{-1/2}$ \\[0.15cm]
\STATE \textbf{for} $k=0,1\dots$ \\[0.15cm]
\STATE \quad  $ \z_{k+1}=\z_k-\G_k^{-1}\hH_k\g_k$ \\[0.15cm]
\STATE \quad $r_k=\|\z_{k+1}-\z_{k}\|$  \\[0.15cm]
\STATE \quad $\tilde{\G}_{k}=(1+Mr_k)\G_k$ \\[0.15cm]
\STATE \quad $\tilde{\LL}_{k}=\LL_k/\sqrt{1+Mr_k}$ \\[0.15cm]
\STATE \quad Choose $\u_k$ from
	\begin{itemize}
	    \item Option I (greedy method): $\tilde{\u}_k=\u_{\H_{k+1}}(\tilde{\LL}_k)$
	    \item Option II (random method): $\tilde{\u}_k\sim\mathcal{N}(\0, \I)$~~~or~~~$\tilde{\u}_k \sim \mathrm{Unif}\left(\mathcal{S}^{n-1}\right)$
	\end{itemize}
\STATE \quad  ${\u}_k=\tilde{\LL}_k^{\top}\u_k$ \\[0.15cm]
\STATE \quad $\G_{k+1}= {\rm BFGS}(\tilde{\G}_k,\H_{k+1},\u_k)$  \\[0.15cm]
\STATE \quad $\LL_{k+1}=$ Fast-Cholesky$\left(\H_{k+1},\LL_k,\vu_k\right)$ \\[0.15cm]
\STATE \textbf{end for}
\end{algorithmic}
\end{algorithm}

\begin{algorithm}[t]
\caption{Greedy/Random SR1 Method for General Case}\label{alg:SR1SPP}
\begin{algorithmic}[1]
\STATE \textbf{Input:} $\G_0\succeq \H$ and $M\geq0$ \\[0.15cm]
\STATE \textbf{for} $k=0,1\dots$\\[0.15cm]
 		\STATE \quad
 		    $\z_{k+1}=\z_k-\G_k^{-1}\hH_k\g_k$ \\[0.15cm]
 		\STATE \quad $r_k=\|\z_{k+1}-\z_{k}\|$ \\[0.15cm]
 		\STATE \quad
 	        $\tilde{\G}_{k}=(1+Mr_k)\G_k$\\[0.15cm]
\STATE \quad Choose $\u_k$ from
	\begin{itemize}
	    \item Option I (greedy method): $\u_k=\u_{\H_{k+1}}(\bar{\G}_k)$
	    \item Option II (random method): $\u_k\sim\mathcal{N}(\0, \I)$~~~or~~~$\u_k \sim \mathrm{Unif}\left(\mathcal{S}^{n-1}\right)$	
	 \end{itemize}
 		\STATE \quad 
 		    $\G_{k+1}= {\rm SR1}(\tilde{\G}_k,\H_{k+1},\u_k)$\\[0.15cm]
 		
\STATE \textbf{end for}
\end{algorithmic}
\end{algorithm}

\subsubsection{Convergence Analysis}\label{sec-general-convergence}

We turn to consider building the convergence guarantee for algorithms proposed in Section \ref{sec-general-alg}.

We first introduce the following notations to simplify the presentation.
We let $\{\z_k\}$ be the sequence generated from Algorithm \ref{alg:BroydenSPP}, \ref{alg:BFGSSPP} or \ref{alg:SR1SPP} and denote
\begin{align*}
\hH_k\defeq\hH_k(\z_k), \qquad
\H_k\defeq \big(\hH_k(\z_k)\big)^2, \qquad
\g_k\defeq\g(\z_k) \qquad \text{and} \qquad
r_k\defeq \|\z_{k+1}-\z_{k}\|.
\end{align*}
We still use gradient norm $\lambda_k\overset{\rm def}{=}\|\nabla f(\z_k)\|$ for analysis and establish the relationship between $\lambda_k$ and $\lambda_{k+1}$, which is shown in Lemma \ref{thm-lambda-general}.

\begin{lemma}
\label{thm-lambda-general}
Using Algorithm \ref{alg:BroydenSPP},  \ref{alg:BFGSSPP} and \ref{alg:SR1SPP}, suppose we have
\begin{equation}
\label{eq-GkHk}
    \H_k \preceq \G_k \preceq \eta_k \H_k,
\end{equation}
for some $\eta_k \geq 1$.
Then we have 
\begin{eqnarray}
\label{6-5}
    \lambda_{k+1}&\leq& \left(1-\frac{1}{\eta_k}\right) \lambda_{k} + \beta \lambda_k^2,\\
\label{6-5-r}
    r_k &\leq& \lambda_k/\mu,
\end{eqnarray}
where $\beta=\dfrac{L_2}{2\mu^2}$.
\end{lemma}
\noindent \citet[Lemma 4.3]{rodomanov2021greedy} derive a result similar to Lemma \ref{thm-lambda-general} for minimizing the strongly-convex function $\hat f(\cdot)$, but depends on the different measure $\lambda_{\hat f}(\cdot)\defeq \inner{\nabla \hat f(\cdot)}{\big(\nabla^2 \hat f(\cdot)\big)^{-1}\nabla \hat f(\cdot)}$.\footnote{The original notations of \citet{rodomanov2021greedy} is minimizing the strongly-convex function $f(\cdot)$ and establishing the convergence result by $\lambda_{f}(\cdot)\defeq \inner{\nabla f(\cdot)}{\big(\nabla^2 f(\cdot)\big)^{-1}\nabla f(\cdot)}$. To avoid ambiguity, we use notations $\hat f(\cdot)$ and $\lambda_{\hat f}$ to describe their work in this paper.}
Note that our algorithms are based on the iteration rule
\begin{align*}
\z_{k+1} = \z_k-\G_k^{-1}\hH_k\g_k.
\end{align*}
Compared with quasi-Newton methods for convex optimization, 
there exists an additional term $\hH_k$ between $\G_k^{-1}$ and $\g_k$, which leads to the convergence analysis based on $\lambda_{\hat f}(\z_k)$ is difficult. Fortunately, we find directly using gradient norm $\lambda_k$ makes the analysis achievable.

For further analysis, we also denote
\begin{equation}\label{eq-def-gamma}
\rho_k \overset{\rm{def}}{=}M\lambda_k/\mu, 
\end{equation}
and
\begin{eqnarray}\label{eq-def-sigma}
\sigma_k \overset{\rm{def}}{=}\sigma_{\H_k}(\G_k)=\langle\H_k^{-1},\G_k\rangle-n.
\end{eqnarray}
Then we establish linear convergence for the first period of iterations, which can be viewed as the generalization of Theorem \ref{thm-quadratic-lambda}. 
Note that the following result holds for all Algorithm \ref{alg:BroydenSPP}, \ref{alg:BFGSSPP} and \ref{alg:SR1SPP}; and it does not depend on the choice of $\vu_k$.

\begin{theorem}
   \label{thm:G-linear-convergence}
Using Algorithm \ref{alg:BroydenSPP}, \ref{alg:BFGSSPP} and \ref{alg:SR1SPP} with $M=2\kappa^2L_2/L$ and $\G_0=L^2\I$, we suppose the initial point $\vz_0$ is sufficiently close to the saddle point such that
\begin{equation}\label{eq-initial-ass-0}
    \frac{M\lambda_0}{\mu} \leq \frac{\ln b}{4b\kappa^2}
\end{equation}
where $1<b<5$. 
Let $\rho_k=M\lambda_k/\mu$, then for all $k\geq0$, we have
\begin{equation*}
    \H_k\preceq\G_k\preceq \exp\left(2\sum_{i=0}^
    {k-1}\rho_i\right)\kappa^2\H_k\preceq b\kappa^2\H_k
\end{equation*}
and 
\begin{equation*}
    \lambda_k \leq\left(1-\frac{1}{2b\kappa^2}\right)^{k}\lambda_{0}.
\end{equation*}
\end{theorem}

\noindent Then we analyze how does $\sigma_k$ change after one iteration to show the local superlinear convergence for Broyden family method (Algorithm \ref{alg:BroydenSPP}) and  BFGS method (Algorithm \ref{alg:BFGSSPP}). Recall that $\sigma_k$ is defined in \eqref{eq-def-sigma} to measure how well does matrix $\G_k$ approximate $\H_k$.

\begin{lemma}
\label{lm:Broydensigma-update} 
Solving general strongly-convex-strongly-concave saddle point problem \eqref{eq-general-obj} under Assumption \ref{assbasic} and \ref{assbasic2} by proposed greedy quasi-Newton algorithms and supposing the sequences $\G_0,\cdots,\G_k$ generated by the Algorithm \ref{alg:BroydenSPP} and \ref{alg:BFGSSPP} are given, then we have the following results for $\sigma_k$:
\begin{enumerate}
    \item For greedy and random Broyden family method (Algorithm \ref{alg:BroydenSPP}), we have 
    \begin{eqnarray}
\label{eq-sigma-update-1}
   \EB_{\u_k}\left[\sigma_{k+1}\right]\leq \left(1-\frac{1}{n\kappa^2}\right)(1+Mr_k)^2\left(\sigma_k + \frac{2nMr_k}{1+M r_k}\right) \quad \text{for all} ~k\geq0.
\end{eqnarray}
    \item For greedy and random BFGS method (Algorithm \ref{alg:BFGSSPP}), we have 
    \begin{eqnarray}
\label{eq-sigma-update-2}
   \EB_{\u_k} \left[\sigma_{k+1}\right]\leq \left(1-\frac{1}{n}\right)(1+Mr_k)^2\left(\sigma_k + \frac{2nMr_k}{1+M r_k}\right)\quad \text{for all} ~k\geq0.
\end{eqnarray} 
\end{enumerate}
Note that for greedy method, we always have $\sigma_{k+1}=\EB_{\u_k}\left[\sigma_{k+1}\right].$
\end{lemma}

\noindent The analysis for SR1 method is based on constructing $\eta_k$ such that
\begin{equation}
\label{eq:defetak0}
    {\rm tr}(\G_k-\H_k)=\eta_k{\rm tr}(\H_k),
\end{equation}
whose details are showed in appendix.
Based on Lemma \ref{lm:Broydensigma-update} and $\eta_k$ satisfies \eqref{eq:defetak0}, we can show the proposed algorithms enjoy the local superlinear convergence for the general saddle point problems.

\begin{theorem} 
\label{thm-Broyden-superlinear}
Solving general saddle point problem \eqref{eq-general-obj} under Assumption \ref{assbasic} and \ref{assbasic2} by proposed greedy and random quasi-Newton methods (Algorithm \ref{alg:BroydenSPP}, \ref{alg:BFGSSPP} and \ref{alg:SR1SPP}) with $M=2\kappa^2L_2/L$ and $\G_0=L^2\I$, we have the following results:
\begin{enumerate}
\item For greedy and random Broyden family method (Algorithm \ref{alg:BroydenSPP}), if the initial point is sufficiently close to the saddle point such that
\begin{eqnarray}
\label{eq-general-initial-2}
    \frac{M\lambda_0}{\mu} \leq \frac{\ln2}{8(1+2n)\kappa^2}, 
\end{eqnarray}
we have
\begin{eqnarray}
\label{eq-g-fth-result-2}
    \EBP{\frac{\lambda_{k+1}}{\lambda_k}}  \leq \left(1-\frac{1}{n\kappa^2}\right)^k2n\kappa^2
    \quad \text{for all} ~k\geq0.
 \end{eqnarray}
\item For greedy and random BFGS method (Algorithm \ref{alg:BFGSSPP}), initial point is sufficiently close to the saddle point such that
\begin{eqnarray}
\label{eq76}
    \frac{M\lambda_0}{\mu} \leq \frac{\ln2}{8(1+2n)\kappa^2}, 
\end{eqnarray}
we have
\begin{eqnarray}
\label{eq77}
    \EBP{\frac{\lambda_{k+1} }{\lambda_k}} \leq \left(1-\frac{1}{n}\right)^k2n\kappa^2\quad \text{for all} ~k\geq0.
\end{eqnarray}
\item For greedy and random SR1 method (Algorithm \ref{alg:SR1SPP}), if the initial point is sufficiently close to the saddle point such that
\begin{eqnarray}
\label{eq-sr1-initial}
    \frac{M\lambda_0}{\mu} \leq \frac{\ln2}{8(1+2n\kappa^2)\kappa^2}, 
\end{eqnarray}
we have
\begin{eqnarray}
\label{eq:sr1thm}
    \EBP{\frac{\lambda_{k+1}}{\lambda_{k}} } \leq \left(1-\frac{1}{n}\right)^k2n\kappa^4, \quad \text{for all}~k\geq0.
\end{eqnarray}
\end{enumerate}
\end{theorem}

Finally, combining the results of Theorem \ref{thm:G-linear-convergence} and \ref{thm-Broyden-superlinear}, we can prove the algorithms achieve the two-stages convergence behaviors as follows.

\begin{corollary}\label{co:Broydenconvergence-g}
Solving general saddle point problem \eqref{eq-general-obj} under Assumption \ref{assbasic} and \ref{assbasic2} by proposed \textbf{greedy} quasi-Newton methods (Algorithm \ref{alg:BroydenSPP}, \ref{alg:BFGSSPP} and \ref{alg:SR1SPP}) with $M=2\kappa^2L_2/L$ and $\G_0=L^2\I$, if the initial point is sufficiently close to the saddle point such that
\begin{eqnarray*}
    \frac{M\lambda_0}{\mu} \leq\frac{\ln2}{8\kappa^2},
\end{eqnarray*}
we have the following results:
\begin{enumerate}
\item For greedy Broyden family method (Algorithm \ref{alg:BroydenSPP}), we have
\begin{align}
\label{eq-broyden-conv-result}
    \lambda_{k_0+k} \leq \left(1-\frac{1}{n\kappa^2}\right)^{\frac{k(k-1)}{2}}\left(\frac{1}{2}\right)^k\left(1-\frac{1}{4\kappa^2}\right)^{k_0}\lambda_0 \quad\text{for all}~k>0~\text{and}~k_0=\OM\left(n\kappa^2 \ln(n\kappa)\right).
\end{align}

\item For greedy BFGS/SR1 method (Algorithm \ref{alg:BFGSSPP}, \ref{alg:SR1SPP}), we have
\begin{align}
\label{eq-bfgs-conv-result}
    {\lambda_{k_0+k}} \leq \left(1-\frac{1}{n}\right)^{\frac{k(k-1)}{2}}\left(\frac{1}{2}\right)^k\left(1-\frac{1}{4\kappa^2}\right)^{k_0}\lambda_0 \quad\text{for all}~k>0~\text{and}~k_0=\OM\left(\max\{n,\kappa^2\} \ln(n\kappa)\right).
\end{align}

\end{enumerate}
\end{corollary}

\begin{corollary}\label{co:Broydenconvergence-random}
Solving general saddle point problem \eqref{eq-general-obj} under Assumption \ref{assbasic} and \ref{assbasic2} by proposed \textbf{random} quasi-Newton methods (Algorithm \ref{alg:BroydenSPP}, \ref{alg:BFGSSPP} and \ref{alg:SR1SPP}) with $M=2\kappa^2L_2/L$ and $\G_0=L^2\I$, if the initial point is sufficiently close to the saddle point such that
\begin{eqnarray*}
    \frac{M\lambda_0}{\mu} \leq\frac{\ln2}{8\kappa^2},
\end{eqnarray*}
then with probability $1-\delta$ for any $\delta\in(0,1)$, we have the following results:
\begin{enumerate}
\item For random Broyden family method (Algorithm \ref{alg:BroydenSPP}), we have
{\small\begin{align}\label{eq-broyden-conv-result-random}
    \lambda_{k_0+k} \leq \left(1-\frac{1}{n\kappa^2+1}\right)^{\frac{k(k-1)}{2}}\left(\frac{1}{2}\right)^k\left(1-\frac{1}{4\kappa^2}\right)^{k_0}\lambda_0 \quad\text{for all}~k>0~\text{and}~k_0=\OM\left(n\kappa^2 \ln(n\kappa/\delta)\right).
\end{align}}

\item For random BFGS/SR1 method (Algorithm \ref{alg:BFGSSPP}, \ref{alg:SR1SPP}), we have
{\small\begin{align}
\label{eq-bfgs-conv-result-random}
    \lambda_{k_0+k} \leq \left(1-\frac{1}{n+1}\right)^{\frac{k(k-1)}{2}}\left(\frac{1}{2}\right)^k\left(1-\frac{1}{4\kappa^2}\right)^{k_0}\lambda_0 \quad\text{for all}~k>0~\text{and}~k_0=\OM\left(\max\{n,\kappa^2\} \ln(n\kappa/\delta)\right).
\end{align}}

\end{enumerate}
\end{corollary}

\section{Extension for Solving Nonlinear Equations}\label{sec:extension}

In this section, we extend our algorithms for solving general nonlinear equations. Concretely, we consider finding the solution of the system
\begin{equation}
\label{eq:neqsys}
    \g(\z)=\0.
\end{equation}
The remainders of this paper do not require the operator $\g(\cdot)$ related to some minimax problem and it could be any differentiable function from $\BR^n$ to $\BR^n$.
We use $\hH(\z)$ to present the Jacobian of $\g(\cdot)$ at $\z\in\BR^n$ and still follow the notation $\H(\z)\defeq \big(\hH(\z)\big)^2$.
We suppose the nonlinear equation satisfies the following conditions.

\begin{assumption}
\label{ass:neqsysass}
The function $\g:\RB^n\to\RB^n$ is differentiable and its Jacobian $\hH:\BR^n\to\BR^{n\times n}$ is $L_2$-Lipschitz continuous. That is, for all $\z,\z' \in \BR^n$, we have
\begin{equation}
   \Norm{ \hH(\z)-\hH(\z')}\leq L_2\|\z-\z'\|.
\end{equation}
\end{assumption}
\begin{assumption}

\label{ass:neqsysass2}
There exists a solution $\z^*$ of equation \eqref{eq:neqsys} such that $\hH(\z^*)$ is non-singular. Additionally, we assume the smallest and largest singular values of $\hH(\z^*)$ are $\mu$ and $L$ respectively.
\end{assumption}

\noindent We still denote the condition number as $\kappa \defeq L/\mu$. Note that the saddle point problem \eqref{eq-general-obj} under Assumption \ref{assbasic} and \ref{assbasic2} is a special case of solving nonlinear equation \eqref{eq:neqsys} under Assumption \ref{ass:neqsysass} and \ref{ass:neqsysass2}. 

Similar to previous section, the design of the algorithms is based on approximating the auxiliary matrix $\H(\z)$. Hence, we start from considering its smoothness.

\begin{lemma}
\label{lem:localNLS}
Under Assumption \ref{ass:neqsysass} and \ref{ass:neqsysass2}, we have
\begin{eqnarray}
\label{eq:nonlinar-obj-conti}
\|\H(\z)-\H(\z')\|\leq 4L_2L\|\z-\z'\|,
\end{eqnarray}
and
\begin{eqnarray}
\label{eq:nonlinear-obj-bound}
  \frac{\mu^2}{2} \I\preceq  \H(\z) \preceq 4L^2 \I,
\end{eqnarray}
for all $\z,\z'\in\left\{ \z: \|\z-\z^*\|\leq \frac{L}{8\kappa^2L_2}\right\}$.
\end{lemma}

\noindent Then we have the property similar to strongly self-concordance like Lemma \ref{lm-self-concordant}.
\begin{lemma}\label{lm-self-concordant2}
For all $\z,\z',\vw \in \left\{\z: \|\z-\z^*\|\leq \frac{L}{8\kappa^2L_2}\right\}$, we have:
\begin{equation}
\label{eq:Nonlinear-concor}
    \H(\z)-\H(\z')\preceq M\|\z-\z'\|\H(\vw)
\end{equation}
with $M=8\kappa^2L_2/L$
\end{lemma}

After above preparation, we can directly apply Algorithm \ref{alg:BroydenSPP}, \ref{alg:BFGSSPP} and \ref{alg:SR1SPP} with $M=8\kappa^2L_2/L$ and $\G_0=4L^2\I$ to find the solution of \eqref{eq:neqsys}. Our convergence analysis is still based on the measure of $\lambda_k\defeq\|\nabla(\z_k)\|$. 
Different from the setting of saddle point problems, the properties shown in Lemma \ref{lem:localNLS} and \ref{lm-self-concordant2} only hold locally. Hence, we introduce the following lemma to show $\z_k$ generated from the algorithms always lies in the neighbor of solution $\vz^*$.

\begin{lemma}\label{lm:zallin}
Solving general nonlinear equations \eqref{eq:neqsys} under Assumption \ref{ass:neqsysass} and \ref{ass:neqsysass2} by proposed greedy and random quasi-Newton methods (Algorithm \ref{alg:BroydenSPP}, \ref{alg:BFGSSPP} and \ref{alg:SR1SPP}) with $M=8\kappa^2L_2/L$ and $\G_0=4L^2\I$, if the initial point $\z_0$ is sufficiently close to the solution $\z^*$ such that
\begin{equation}
     \frac{M\lambda_0}{\mu} \leq \frac{\ln2}{64\sqrt{2}\kappa^2}, 
\end{equation}
then for all $k \geq 0$, we have
\begin{equation}
\label{eq:lamneq-nonlinear}
    \lambda_k \leq\left(1-\frac{1}{32\kappa^2}\right)^{k}\lambda_{0}, 
\end{equation}
which means $\z_k\in\left\{ \z: \|\z-\z^*\|\leq \frac{L}{8\kappa^2L_2}\right\}$.
\end{lemma}

\noindent Based on Lemma \ref{lm:zallin}, we establish the following theorem to show the algorithms also have local superlinear convergence for solving nonlinear equations.

\begin{theorem}
\label{thm:Broydenconvergence-Nonlinear}
Solving general nonlinear equations \eqref{eq:neqsys} under Assumption \ref{ass:neqsysass} and \ref{ass:neqsysass2} by proposed {\textbf {greedy}} quasi-Newton methods (Algorithm \ref{alg:BroydenSPP}, \ref{alg:BFGSSPP} and \ref{alg:SR1SPP}) with $M=8\kappa^2L_2/L$ and $\G_0=4L^2\I$, if the initial point $\z_0$ is sufficiently close to the solution $\z^*$ such that
\begin{eqnarray*}
    \frac{M\lambda_0}{\mu} \leq\frac{\ln2}{64\sqrt{2}\kappa^2},
\end{eqnarray*}
we have the following results.
\begin{enumerate}
\item For greedy Broyden family method (Algorithm \ref{alg:BroydenSPP}), we have
\begin{align*}
    \lambda_{k_0+k} \leq \left(1-\frac{1}{8n\kappa^2}\right)^{\frac{k(k-1)}{2}}\left(\frac{1}{2}\right)^k\left(1-\frac{1}{32\kappa^2}\right)^{k_0}\lambda_0 \quad\text{for all}~k>0~\text{and}~k_0=\OM\left(n\kappa^2 \ln(n\kappa)\right).
\end{align*}

\item For greedy BFGS/SR1 method (Algorithm \ref{alg:BFGSSPP}, \ref{alg:SR1SPP}), we have
\begin{align*}
    \lambda_{k_0+k} \leq \left(1-\frac{1}{n}\right)^{\frac{k(k-1)}{2}}\left(\frac{1}{2}\right)^k\left(1-\frac{1}{32\kappa^2}\right)^{k_0}\lambda_0 \quad\text{for all}~k>0~\text{and}~k_0=\OM\left(\max\{n,\kappa^2\} \ln(n\kappa)\right).
\end{align*}

\end{enumerate}
\end{theorem}

\begin{theorem}
\label{co::Broydenconvergence-Nonlinear-random}
Solving general nonlinear equations \eqref{eq:neqsys} under Assumption \ref{ass:neqsysass} and \ref{ass:neqsysass2} by proposed {\textbf{random}} quasi-Newton methods (Algorithm \ref{alg:BroydenSPP}, \ref{alg:BFGSSPP} and \ref{alg:SR1SPP}) with $M=8\kappa^2L_2/L$ and $\G_0=4L^2\I$, if the initial point $\z_0$ is sufficiently close to the solution $\z^*$ such that
\begin{eqnarray*}
    \frac{M\lambda_0}{\mu} \leq\frac{\ln2}{64\sqrt{2}\kappa^2},
\end{eqnarray*}
with probability $1-\delta$ for any $\delta\in(0,1)$, we have the following results.
\begin{enumerate}
\item For random Broyden family method (Algorithm \ref{alg:BroydenSPP}), we have
\begin{align*}
    \lambda_{k_0+k} \leq \left(1-\frac{1}{8n\kappa^2+1}\right)^{\frac{k(k-1)}{2}}\left(\frac{1}{2}\right)^k\left(1-\frac{1}{32\kappa^2}\right)^{k_0}\lambda_0 ~\text{for all}~k>0~\text{and}~k_0=\OM\left(n\kappa^2 \ln(n\kappa/\delta)\right).
\end{align*}

\item For random BFGS/SR1 method (Algorithm \ref{alg:BFGSSPP}, \ref{alg:SR1SPP}), we have
\begin{align*}
    \lambda_{k_0+k} \leq \left(1-\frac{1}{n+1}\right)^{\frac{k(k-1)}{2}}\left(\frac{1}{2}\right)^k\left(1-\frac{1}{32\kappa^2}\right)^{k_0}\lambda_0 ~\text{for all}~k>0~\text{and}~k_0=\OM\left(\max\{n,\kappa^2\} \ln(n\kappa/\delta)\right).
\end{align*}

\end{enumerate}
\end{theorem}

\begin{table*}[!t]
\label{table:comparison}	\centering\renewcommand{\arraystretch}{1.8}
	\begin{tabular}{|c|c|c|c|}
		\hline
       \multicolumn{2}{|c|}{ Algorithms} & Upper Bound of $\lambda_{k+k_0}$ & $k_0$\\
		\hline
		\multirow{2}{*}{Broyden (Algorithm \ref{alg:BroydenSPP})}& Greedy & $\left(1-\frac{1}{8n\kappa^2}\right)^{k(k-1)/2}\left(\frac{1}{2}\right)^k\left(1-\frac{1}{32\kappa^2}\right)^{k_0}$ & $\OM\left(n\kappa^2\ln(n\kappa)\right)$ \\ 
		&Random& \!\!$\left(1-\frac{1}{8n\kappa^2+1}\right)^{k(k-1)/2}\left(\frac{1}{2}\right)^k\left(1-\frac{1}{32\kappa^2}\right)^{k_0}$\!\! & $\OM\left(n\kappa^2\ln(\frac{n\kappa}{\delta})\right)$\\
		\hline
		\multirow{2}{*}{\!\!BFGS/SR1 (Algorithm \ref{alg:BFGSSPP}/\ref{alg:SR1SPP})}\!\! & Greedy & $\left(1-\frac{1}{n}\right)^{k(k-1)/2}\left(\frac{1}{2}\right)^k\left(1-\frac{1}{32\kappa^2}\right)^{k_0}$ & $\OM\left(\max\left\{n,\kappa^2\right\}\ln(n\kappa)\right)$ \\ 
		&Random& $\left(1-\frac{1}{n+1}\right)^{k(k-1)/2}\left(\frac{1}{2}\right)^k\left(1-\frac{1}{32\kappa^2}\right)^{k_0}$ & $\OM\left(\max\left\{n,\kappa^2\right\}\ln(\frac{n\kappa}{\delta})\right)$\\
		\hline
	\end{tabular}
	\caption{We summarize the convergence behaviors of three proposed algorithms for non-linear equations system in the view of gradient norm $\|\nabla f(\z_{k+k_0})\|$ after $(k+k_0)$ iterations. The results come from Theorem \ref{thm:Broydenconvergence-Nonlinear} and Theorem \ref{co::Broydenconvergence-Nonlinear-random} and the upper bounds of random algorithms holds with high probability at least $1-\delta$.} \vskip0.3cm
\end{table*}

\begin{figure}[t]
\centering
\begin{tabular}{ccc}
\includegraphics[scale=0.33]{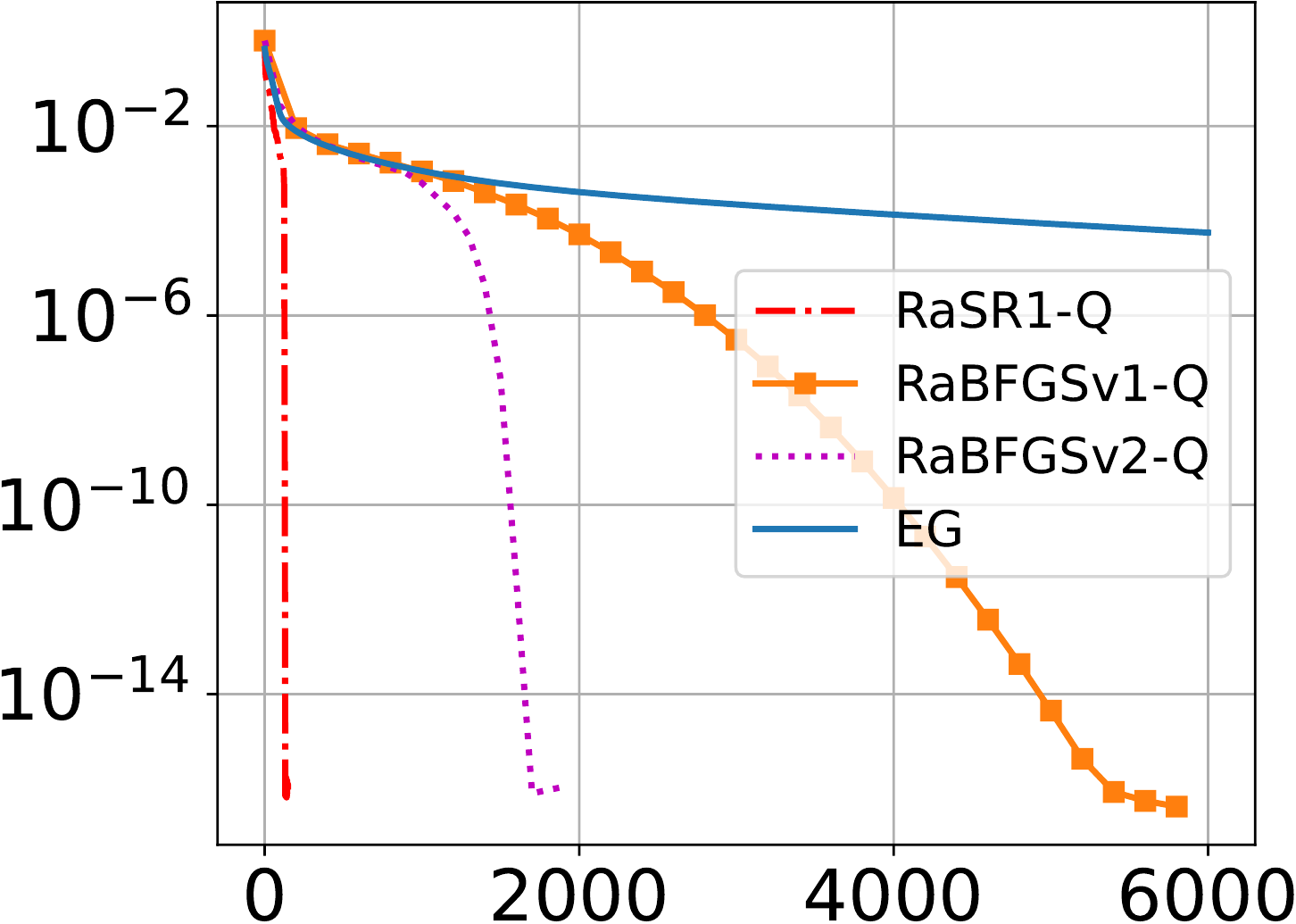} &
\includegraphics[scale=0.33]{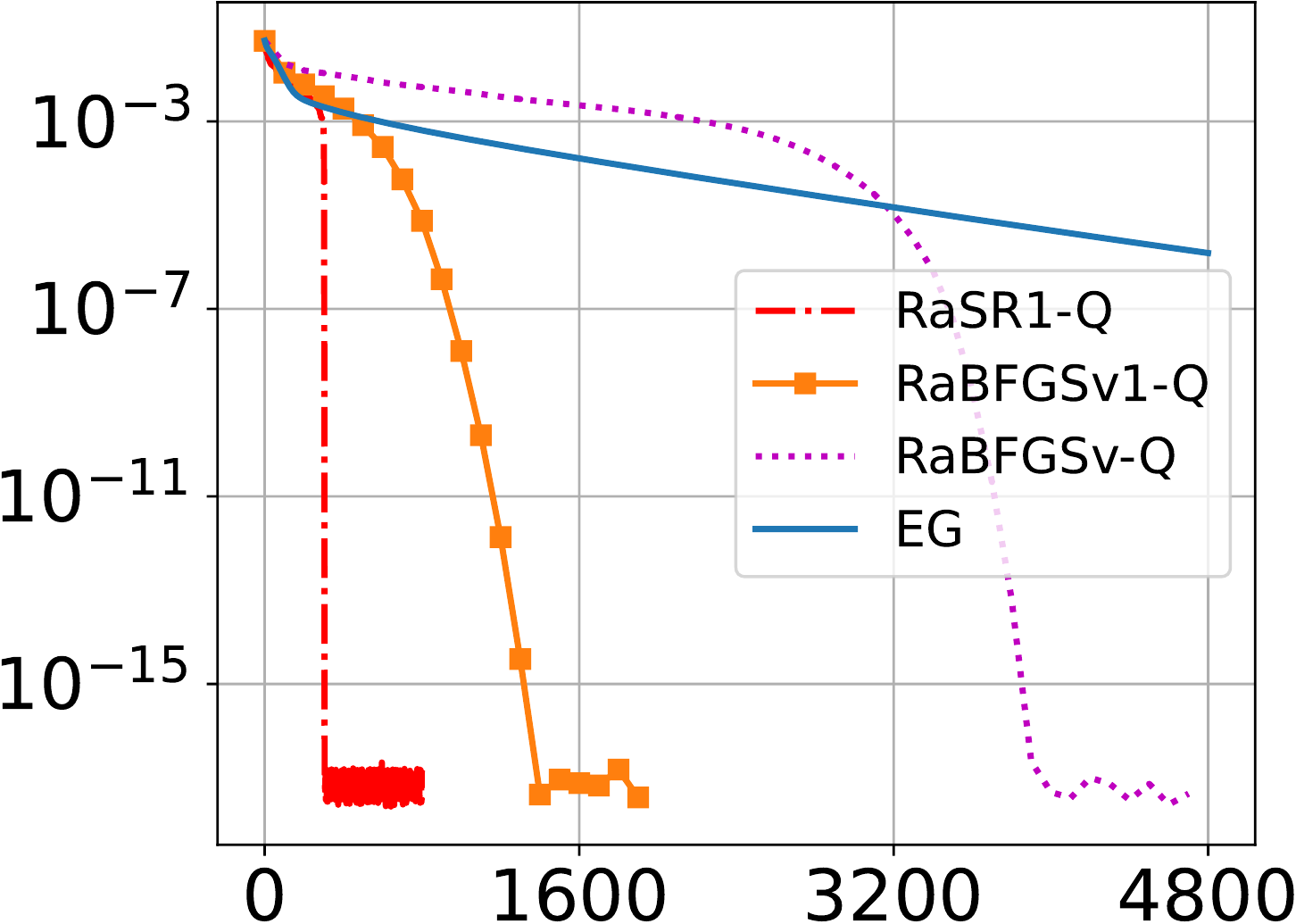} &
\includegraphics[scale=0.33]{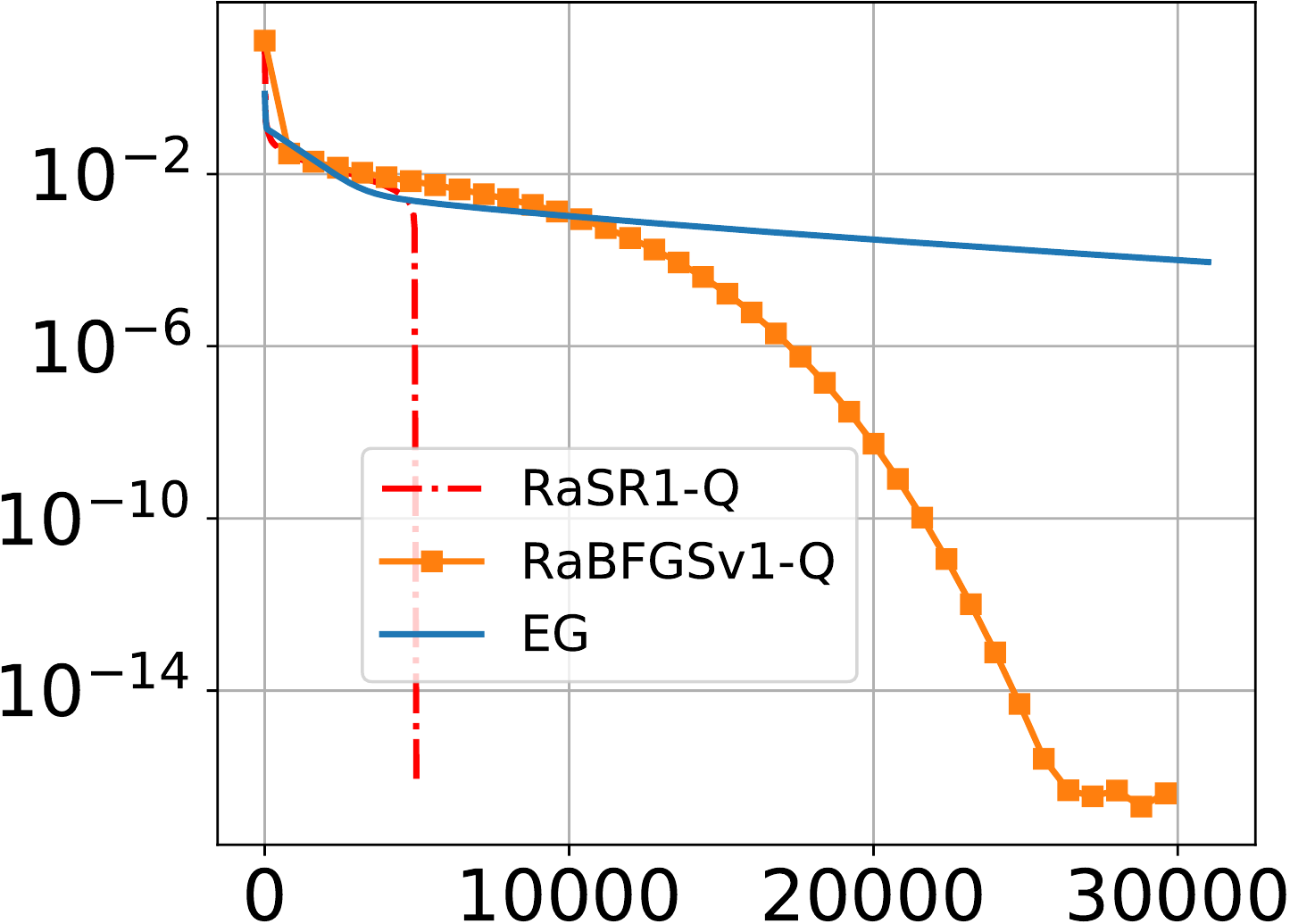}
\\[-0.1cm]
\small (a) a9a (iteration) & \small  (b) w8a (iteration) & \small  (c) sido0 (iteration) \\[0.2cm]
\includegraphics[scale=0.33]{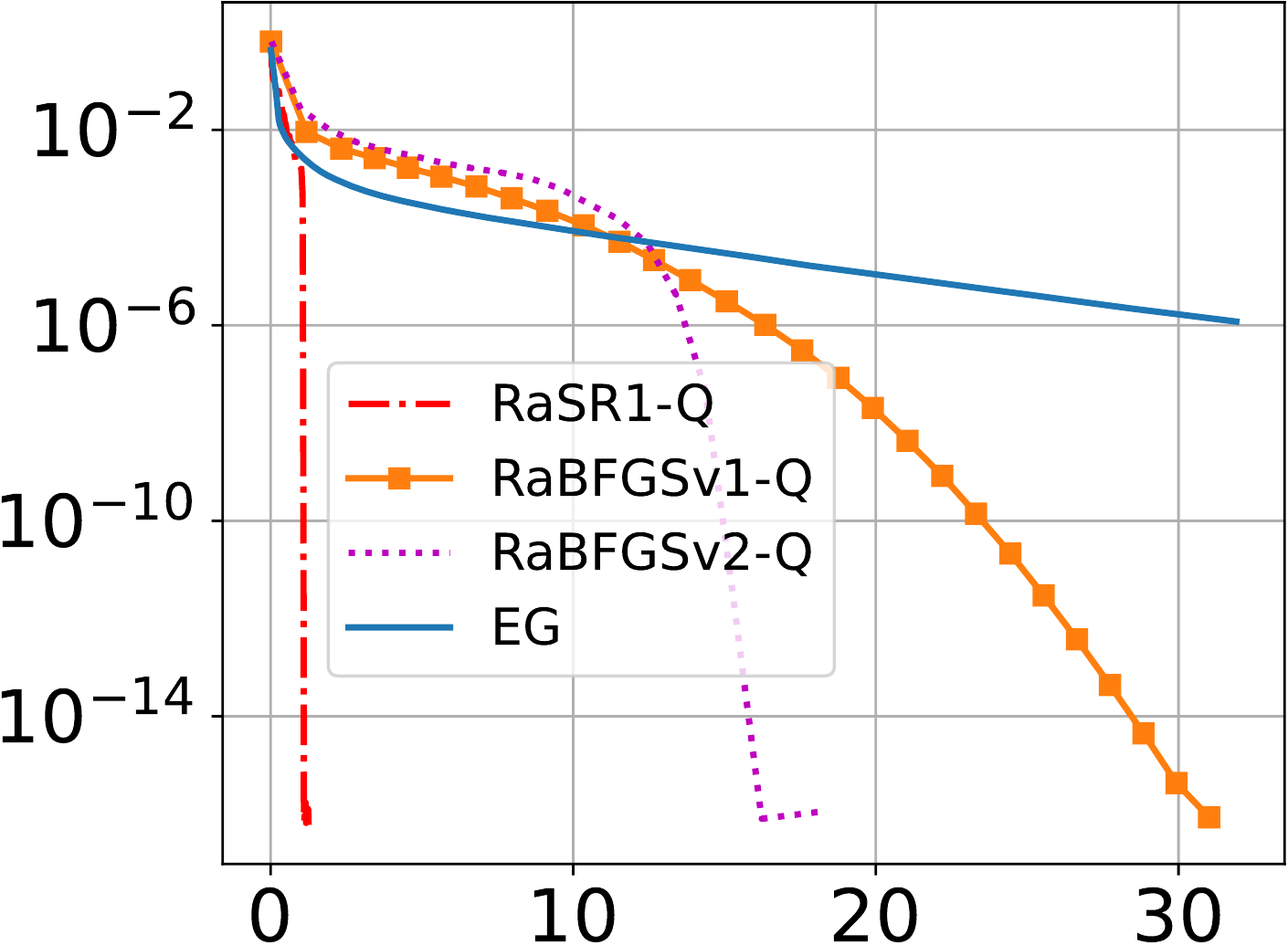} & 
\includegraphics[scale=0.33]{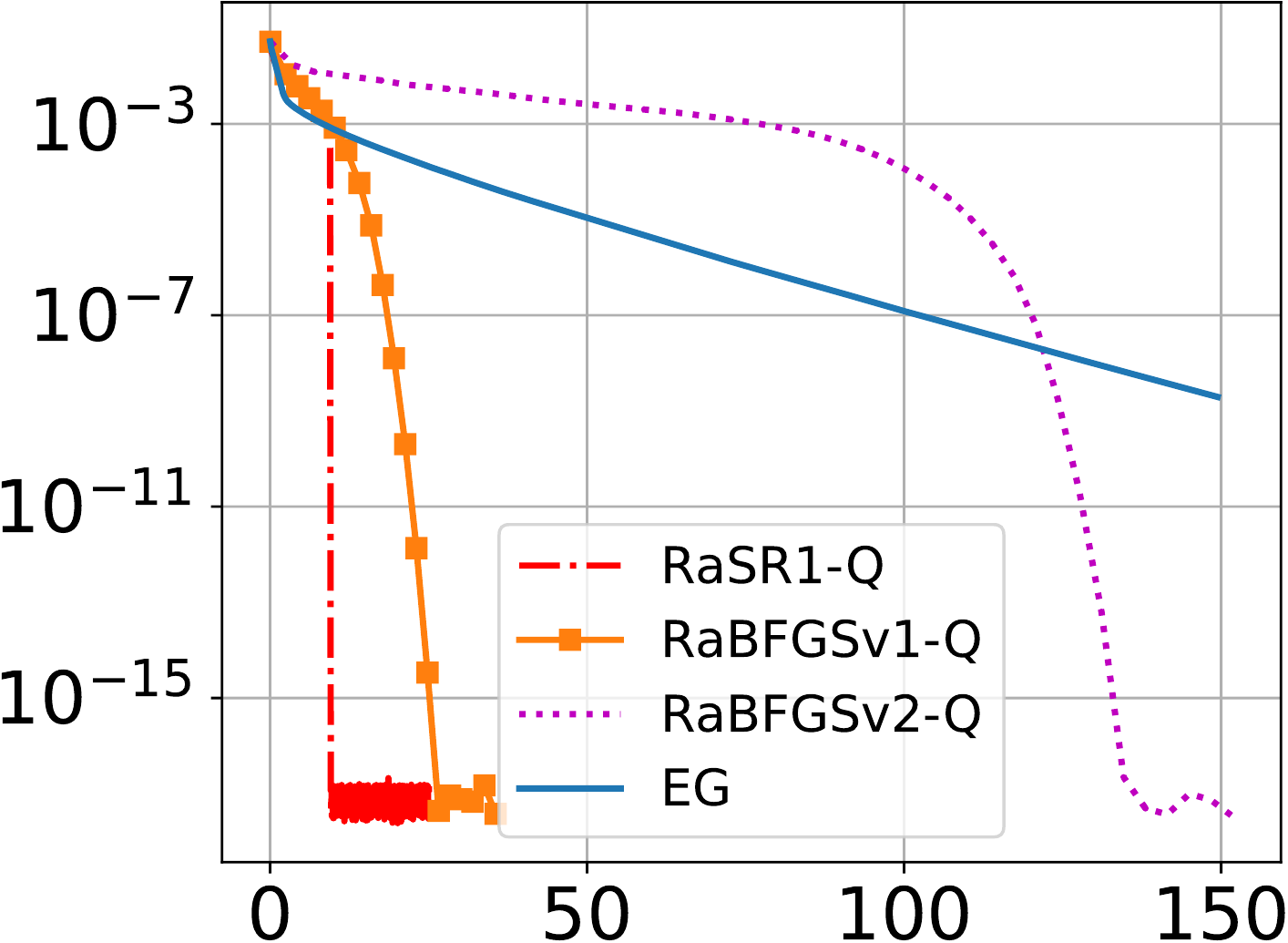} &
\includegraphics[scale=0.33]{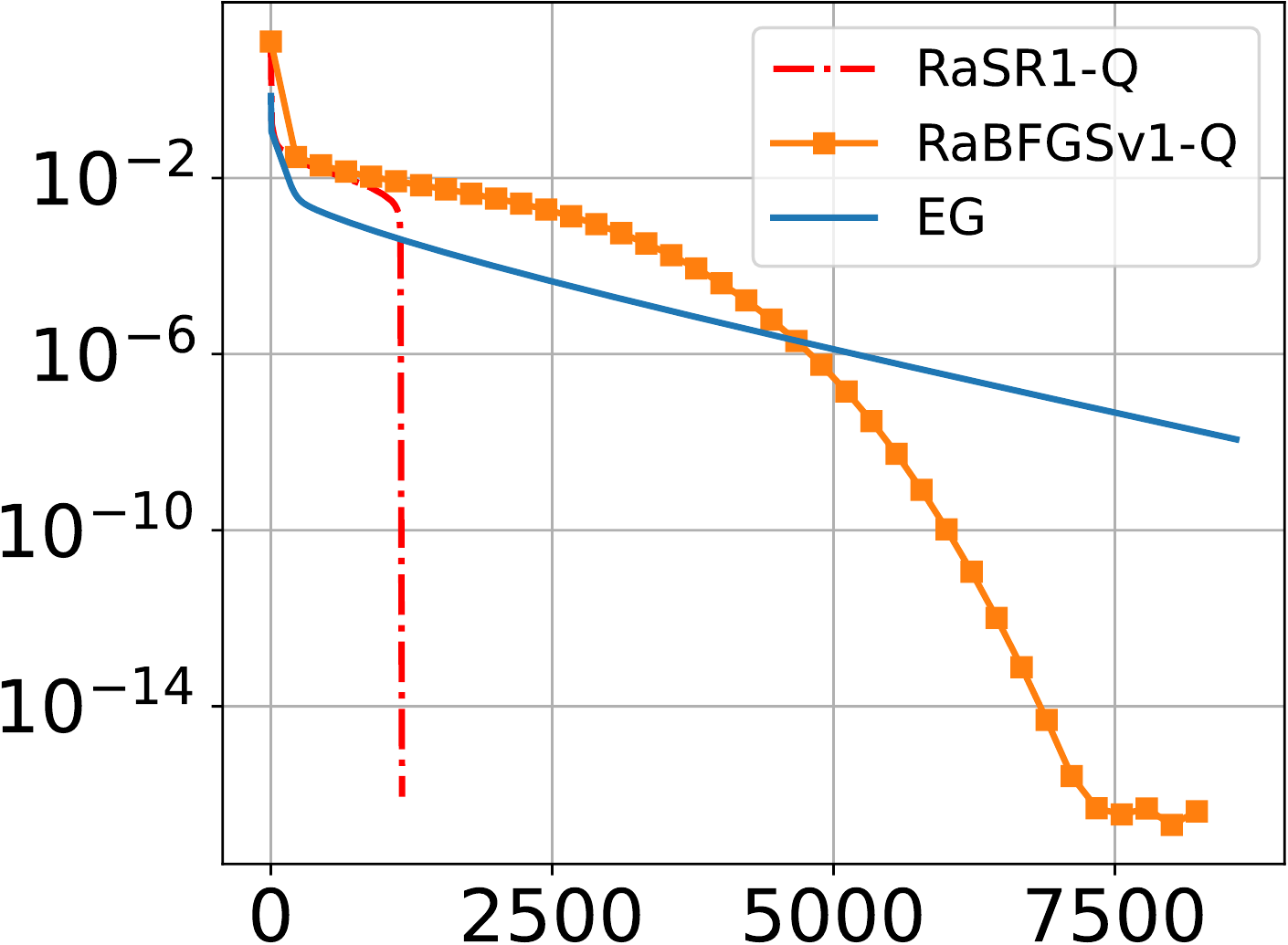} \\[-0.1cm]
\small  (d) a9a (time) & \small  (e) w8a (time) &\small  (f) sido0 (time) 
\end{tabular}\vskip-0.15cm
\caption{We demonstrate iteration numbers vs. $\|\vg(\vz)\|_2$ and CPU time (second) vs. $\|\vg(\vz)\|_2$ for AUC model on datasets ``a9a'' ($n=126$, $m=32561$), ``w8a'' ($n=303$, $m=45546$) and ``sido'' ($n=4935$, $m=12678$).}\label{fig:aucepoch}\vskip-0.3cm
\end{figure}

\paragraph{Related Work}
Very recently, \citet{lin2021explicit,ye2021greedy} showed Broyden's methods have explicit local superlinear convergence rate for solving nonlinear equations, but their assumptions are quite different from ours.
Specifically, they only suppose $\|\hH(\z)-\hH(\z^*)\|\leq L_2\Norm{\z-\z^*}$ for any $\vz\in\BR^n$, which is weaker than Assumption~\ref{ass:neqsysass} of this paper.
However, their analysis additionally requires the approximate Jacobian matrix $\G_0$ for initial point $\z_0$ must be sufficiently close to the exact Jacobian matrix $\hH(\z_0)$, which is a strong assumption and unnecessary for our algorithms.
\section{Numerical Experiments}\label{sec:exp}
In this section, we conduct the experiments on machine learning applications of AUC Maximization and adversarial debiasing to verify our theory.
We refer to Algorithm \ref{alg:BroydenQuadratic} and \ref{alg:BroydenSPP} with parameter $\tau_k=\u^{\top}\H_{k+1}/(\u^{\top}\G_k\u)$ as RaBFGSv1-Q and RaBFGSv1-G; refer to Algorithm \ref{alg:BFGSQuadratic}, \ref{alg:SR1Quadratic},  \ref{alg:BFGSSPP} and \ref{alg:SR1SPP} as RaBFGSv2-Q, RaSR1-Q, RaBFGSv2-G and RaSR1-G respectively. 
We compare these proposed algorithms with classical first-order method extragradient (EG) \cite{korpelevich1976extragradien,tseng1995linear}. 
Some detailed settings of the experiments can be found in appendix.

\subsection{AUC Maximization}
AUC maximization~\cite{hanley1982meaning,ying2016stochastic} aims to find the classifier $\w\in\RB^{d}$ on the training set $\{\va_i,b_i\}_{i=1}^m$ where $\va_i \in \RB^{d}$ and $b_i\in \{+1,-1\}$. We denote $m^{+}$ be the number of positive instances and $p=m^{+}/m$. The minimax formulation for AUC maximization can be written as 
\begin{equation*}
\min_{\vx\in\BR^{d+2}}\max_{y\in\BR} f(\vx,y) \defeq \frac{1}{m} \sum_{i=1}^m f_i(\vx,\vy;\va_i,b_i,\lambda),
\end{equation*}
where $\vx=[\vw;u;v]\in\RB^{d+2}$, $\lambda>0$ is the regularization parameter and $f_i$ is defined as
\begin{align*}
f_i(\vx,\vy;\va_i,b_i,\lambda) = & \frac{\lambda}{2}\|\vx\|_2^2-p(1-p)y^2
+p((\vw^T\va_i-v)^2+2(1+y)\vw^T\va_i)\mathbb{I}_{b_i=-1} \\
& +(1-p)((\vw^T\va_i-u)^2-2(1+y)\vw^T\va_i)\mathbb{I}_{b_i=1}.
\end{align*}
Note that the objective function of AUC maximization is quadratic, hence we conduct the algorithms in Section \ref{sec-quadratic} (Algorithm \ref{alg:BroydenQuadratic}, \ref{alg:BFGSQuadratic} and \ref{alg:SR1Quadratic}) for this model.
We set $\lambda = 100/m$ and evaluate all algorithms on three real-world datasets ``a9a'', ``w8a'' and ``sido0''. 
The dimension of the problem is $n = d + 3$.
The results of iteration numbers against $\|\vg(\vz)\|_2$ and CPU time against $\|\vg(\vz)\|_2$ are presented in Figure \ref{fig:aucepoch}.

\begin{figure}[t]
\centering
\begin{tabular}{ccc}
\includegraphics[scale=0.33]{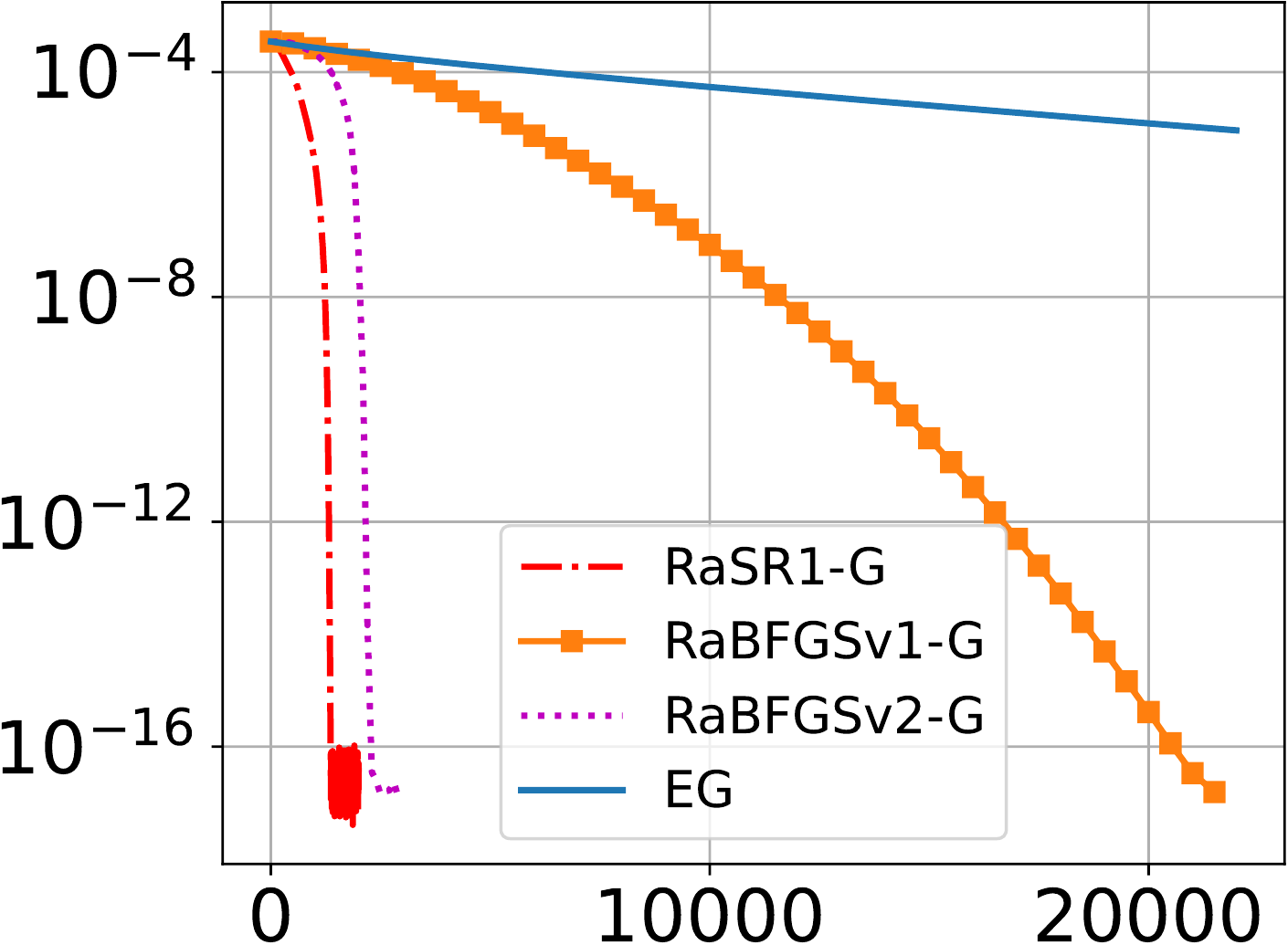} &
\includegraphics[scale=0.33]{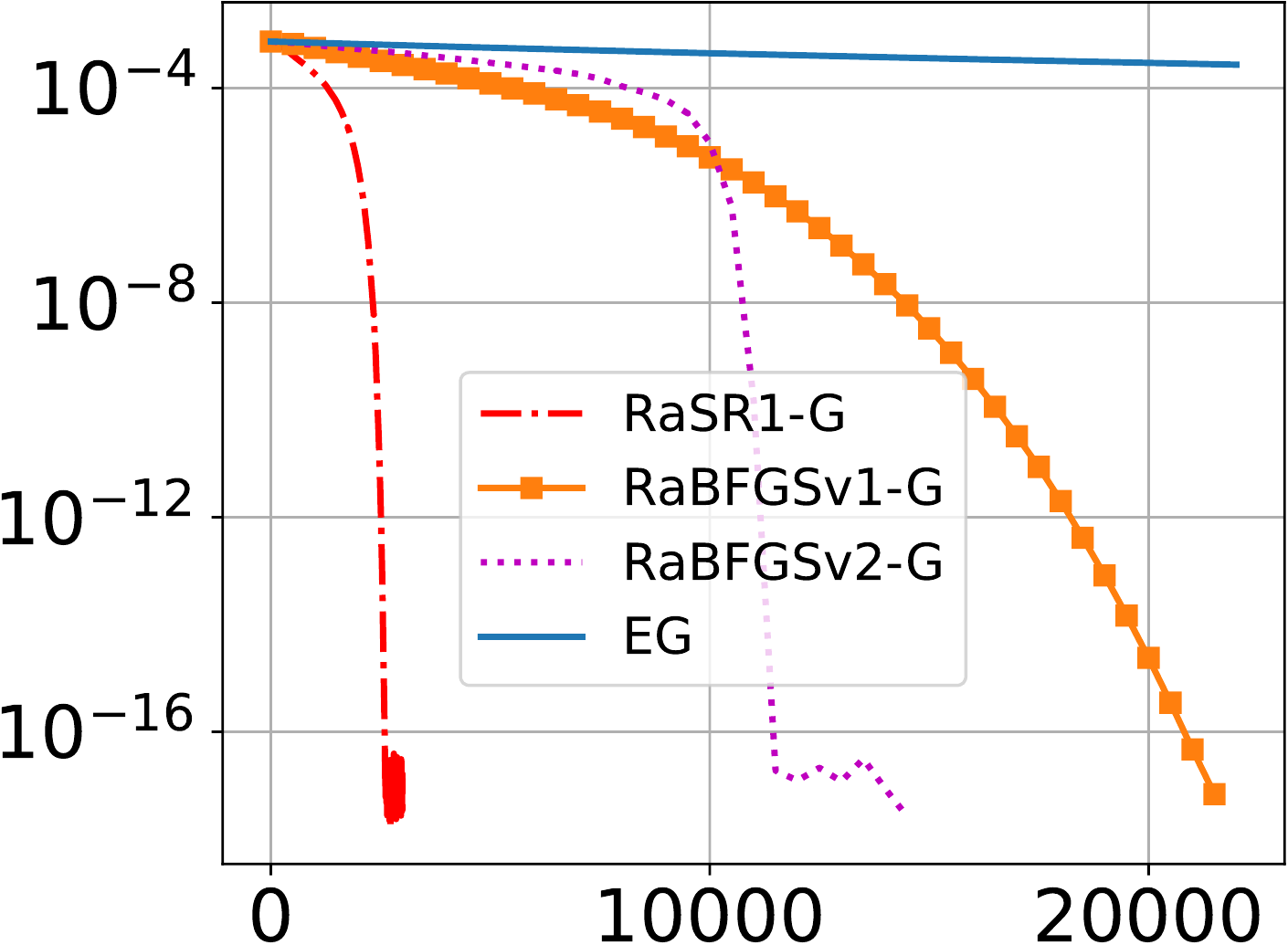} & 
\includegraphics[scale=0.33]{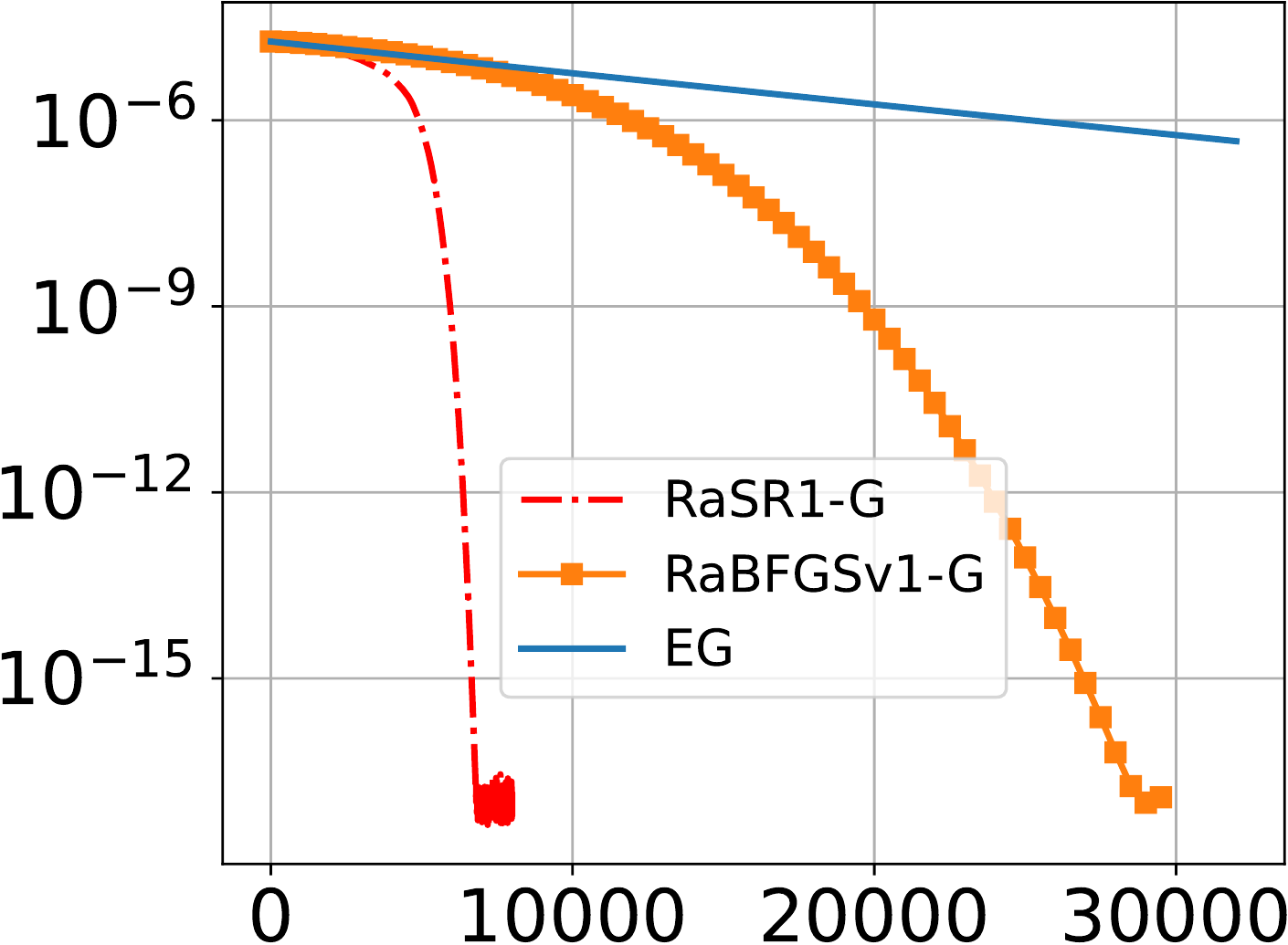} \\[-0.1cm]
\small  (a) adults (iteration) &\small  (b) law school (iteration) & \small  (c) bank market (iteration) \\[0.2cm]
\includegraphics[scale=0.33]{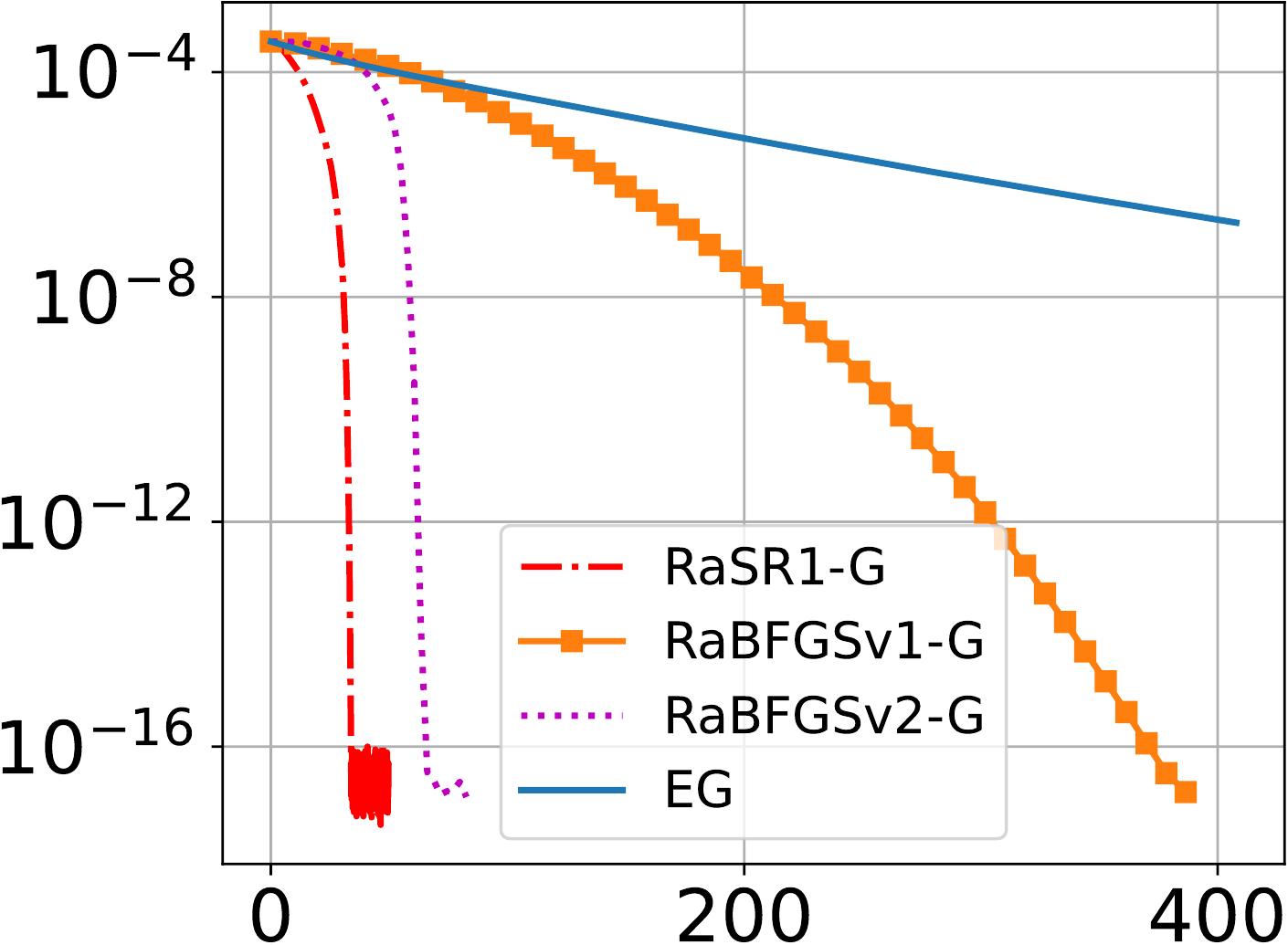} &
\includegraphics[scale=0.33]{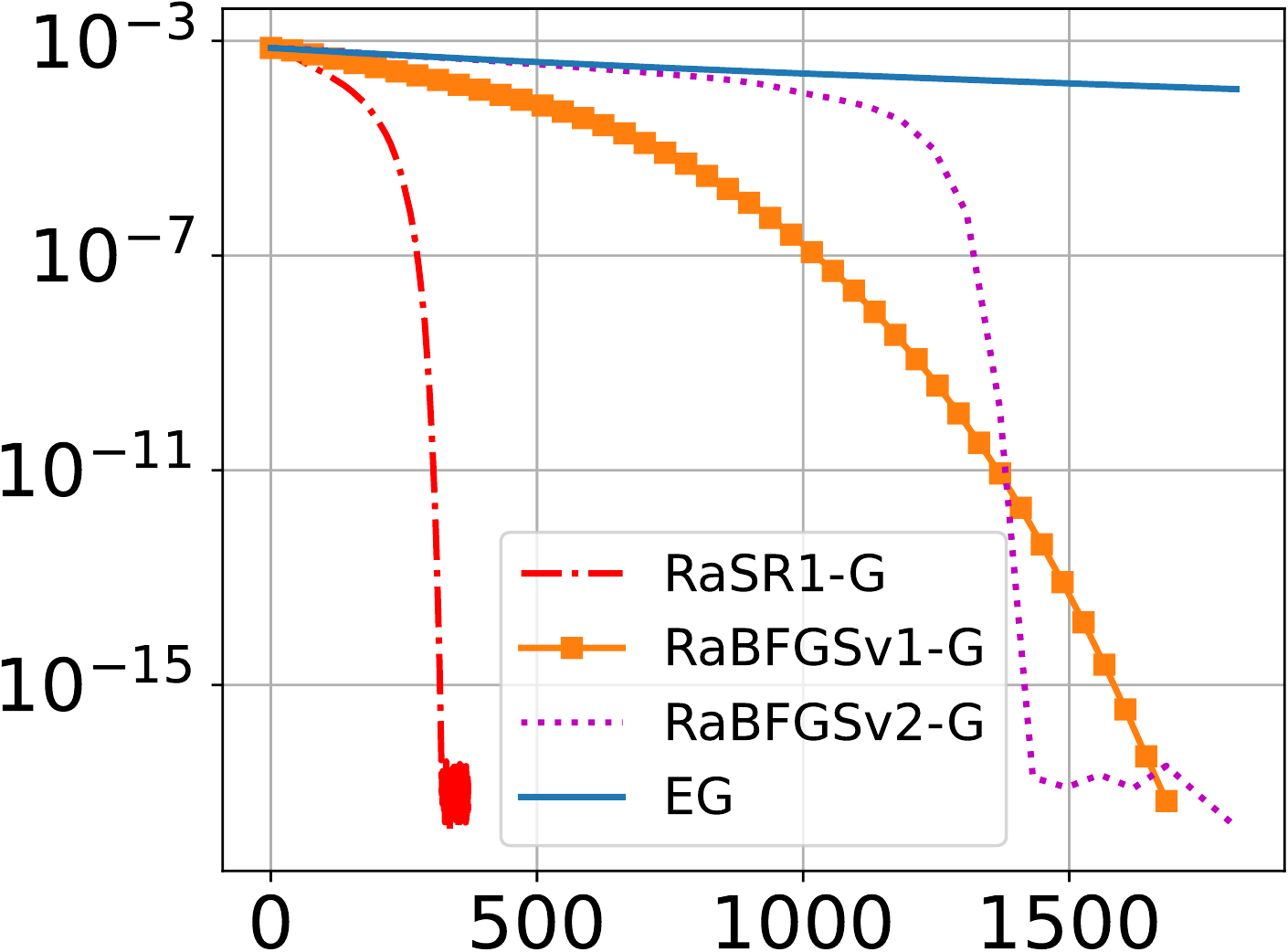}&
\includegraphics[scale=0.33]{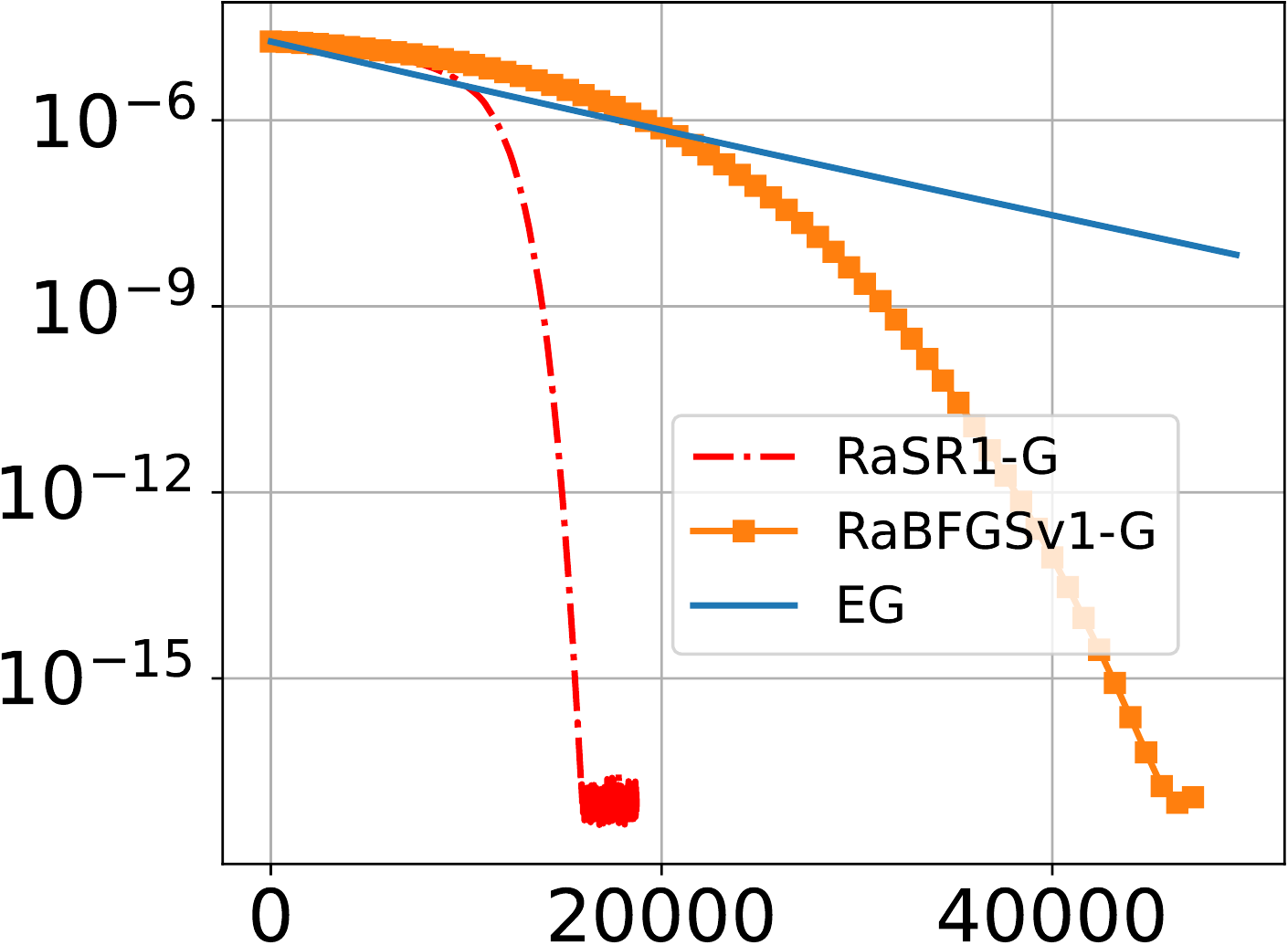}
\\[-0.1cm]
\small  (d) adults (time) & (e) law school (iteration) & \small  (f) bank market (time) 
\end{tabular}\vskip-0.15cm
\caption{We demonstrate iteration numbers vs. $\|\vg(\vz)\|_2$ and CPU time (second) vs. $\|\vg(\vz)\|_2$ for adversarial debiasing model on datasets ``adults'' ($n=123$, $m=32561$), ``law school'' ($n=380$, $m=20427$) and ``bank market'' ($n=3880$, $m=45211$).}\label{fig:fairepoch}\vskip-0.15cm
\end{figure}


\subsection{Adversarial Debiasing}
Adversarial learning \cite{lowd2005adversarial,zhang2018mitigating} can be used on fairness-aware machine learning issues. 
Give the training set $\{\va_i,b_i,c_i\}_{i=1}^m$ where $\va_i\in\RB^{d}$ contains all input variables, $b_i\in \RB$ is the output and $c_i\in \RB$ is the input variable which we want to protect and make it unbiased.
Our experiments are based on the fairness-aware binary classification dataset ``adult'', ``bank market'' and ``law school''\cite{quy2021survey}, leading to $b_i,c_i\in\{+1,-1\}$.
The model has the following minimax formulation
\begin{align*}
\mbox{\small
$\displaystyle{\min_{\x\in\RB^{d}}\max_{y\in \RB} \frac{1}{m}\sum_{i=1}^{m}(l_1(\va_i,b_i,\x)\!-\!\beta l_2(\va_i^\top\x,c_i,y))\!+\!\lambda\|\x\|^2\!-\!\gamma y^2}$}
\end{align*}
where $l_1$, $l_2$ are the logit functions: ${\rm logit}(\va,b,\vc)= \log(1+\exp(-b\va^\top\vc))$. We set the parameters $\beta,\lambda$ and $\gamma$ as $0.5, 10^{-4}$ and $10^{-4}$ respectively. The dimension of the problem is $n=d+1$.
Since the objective function is non-quadratic, we conduct the proposed algorithms in Section \ref{sec-general} (Algorithm \ref{alg:BroydenSPP}, \ref{alg:BFGSSPP} and \ref{alg:SR1SPP}) here.
We use extragradient as warm up to achieve the local condition for proposed algorithms. The results of iteration numbers against $\|\vg(\vz)\|_2$ and CPU time against $\|\vg(\vz)\|_2$ are presented in Figure \ref{fig:fairepoch}.


\section{Conclusion}\label{sec:conclusion}
In this work, we have proposed quasi-Newton methods for solving strongly-convex-strongly-concave saddle point problems. 
We have characterized the second-order information by approximating the square of Hessian matrix, which avoid the issue of dealing with the indefinite Hessian directly.
We have presented the explicit local superlinear convergence rates for Broyden's family update and a faster convergence rates for two specific methods: SR1 and BFGS updates. 
Moreover, we have also extend our theory to solve the general nonlinear equation systems and provided the similar convergence results for the algorithms.


\bibliographystyle{plainnat}
\bibliography{ref}

\pagebreak 
\appendix

\section{Efficient Implementation for Algorithm \ref{alg:L update}}
\label{faster}

For the self-completeness of this paper, we present Proposition 15 of \citet{lin2021faster} to show Algorithm \ref{alg:L update} can be implemented with $\fO\left(n^2\right)$ flops.

\begin{lemma}[{\cite[Proposition 15]{lin2021faster}}]
\label{lemma:l}
    In this Lemma, we show how to construct upper triangle matrix $\hat\LL$ from $\LL$, $\H$ and the greedy direction $\u$ with $\OM(n^2)$ flops.
	From the inverse BFGS update rule of $\A=\G^{-1}$, we have
	\begin{equation}\label{eq:bfgs-inv-update}
	    \A_+=\left({\rm BFGS}\left(\G,\H,\u\right)\right)^{-1}=\left(\I-\dfrac{\u\u^\top \H}{\u^\top \H\u}\right) \A \left(\I-\dfrac{\H \u \u^\top}{\u^\top \H \u}\right)+\dfrac{\u\u^\top}{\u^\top \H \u}.
	\end{equation}
	Suppose we already have $\A=\LL^\top \LL$ where $\LL$ is an upper triangular matrix, now we construct $\hat{\LL}$ such that $\A_+=\LL^\top\LL$ with $\OM(n^2)$ flops.
	\begin{enumerate}
	    \item First we can obtain $QR$ decomposition of 
	    \[ \LL\left(\I-\dfrac{\H\u\u^\top}{\u^\top \H\u}\right)=\LL-\dfrac{\LL(\H\u)}{\u^\top \H\u}\u^\top \]
	    with $\OM(n^2)$ flops since it is a rank-one changes of $\LL$.
	    \item Second, we have $\LL\left(\I-\dfrac{\H\u\u^\top}{\u^\top \H\u}\right)=\Q \R$, with an orthogonal matrix $\Q\in\RB^{n\times n}$ and an upper triangular matrix $\R\in\RB^{n\times n}$. We denote $\v=\frac{\u}{\sqrt{\u^\top \H \u}}$, then we can view
	    \begin{equation*}
	        \A_{+} = \R^\top \Q^\top \Q \R + \dfrac{\u\u^\top}{\u^\top \H \u} = 
	        \R^\top \R +\v\v^\top = \begin{bmatrix} \v & \R^\top  \end{bmatrix}\begin{bmatrix} \v^\top \\ \R \end{bmatrix}.
	    \end{equation*}
	     we still can obtain $QR$ decomposition of $\begin{bmatrix} \v^\top \\ \R\end{bmatrix}$ with only $\OM(n^2)$ flops, leading to
	    $\begin{bmatrix} \v^\top \\ \R\end{bmatrix} = \Q'\R'$, with an column orthogonal matrix $\Q'\in\RB^{(n+1) \times n}$ and an upper triangular matrix $\R'\in\RB^{n\times n}$, 
	    and 
	    \[ \A_+ = \R'^\top \R'. \] 
	    Thus $\R'$ satisfies our requirements.
	\end{enumerate}
\end{lemma}

\section {A Useful Lemma for Convergence Analysis and Proof of Lemma \ref{lm-matrix-neq0}}

This section first provides a useful lemma of symmetric positive definite matrices for our further analysis.

\begin{lemma}\label{lm-matrix-neq}
Suppose $\H, \G\in\BR^{n\times n}$ are symmetric positive definite matrices, $\hH\in \BR^{n\times n}$ are symmetric non-degenerate matrix where $\H=(\hH)^2$ and
\begin{equation*}
 \H\preceq\G\preceq\eta \H ,  
\end{equation*} where $\eta >1$.
Then we have
\begin{eqnarray}
\label{eq32}
    \|\I-\hH\G^{-1}\hH\|\leq 1-\frac{1}{\eta}.
\end{eqnarray}
\end{lemma}
\begin{proof}
We have the following inequality for $\G^{-1}$ and $\H^{-1}$
\begin{equation}
    \frac{1}{\eta}\H^{-1}\preceq\G^{-1}\preceq\H^{-1},
\end{equation}
which means
\begin{equation}
    \0\preceq\H^{-1}-\G^{-1}\preceq\left(1-\frac{1}{\eta}\right)\H^{-1}.
\end{equation}
Thus we have
\begin{equation}
    \0\preceq \hH(\H^{-1}-\G^{-1})\hH\preceq\left(1-\frac{1}{\eta}\right)\hH\H^{-1}\hH\preceq\left(1-\frac{1}{\eta}\right)\I.
\end{equation}
So we have
\begin{equation}
    \|\I-\hH\G^{-1}\hH\|=\|\hH(\H^{-1}-\G^{-1})\hH\|\leq\left(1-\frac{1}{\eta}\right).
\end{equation}
\end{proof}

\subsection{The Proof of Lemma \ref{lm-matrix-neq0}}
\begin{proof}
We partition $\hH(\vz)\in\BR^{n\times n}$ as
\begin{equation*}
    \hH(\vz)=\begin{bmatrix}
\hH_{\vx\vx}(\vz) & \hH_{\vx\vy}(\vz)   \\
\hH_{\vy\vx}(\vz)^{\top} & \hH_{\vy\vy}(\vz)
\end{bmatrix}\in\RB^{n\times n}
\end{equation*}
where the sub-matrices $\hH_{\x\x}(\vz)\in\BR^{n_x\times n_x}$, $\hH_{\x\y}(\vz)\in\BR^{n_x\times n_y}$, $\hH_{\y\x}(\vz)\in\BR^{n_y\times n_x}$ and $\hH_{\y\y}(\vz)\in\BR^{n_y\times n_y}$ satisfy $\hH_{\x\x}(\vz)\succeq \mu\I$, $\hH_{\y\y}(\vz) \preceq -\mu\I$ and $\|\hH(\vz)\|\leq L$ for all $\vz\in\BR^n$ by Assumption \ref{assbasic} and \ref{assbasic2}.
We denote
\begin{equation*}
\J=\begin{bmatrix}
\I_{n_x} & \0   \\
\0 & -\I_{n_y}
\end{bmatrix}  \in \BR^{n\times n} ,
\end{equation*}
where $\I_{n_x}$ and $\I_{n_y}$ are $n_x\times n_x$ identity matrix and $n_y\times n_y$ identity matrix respectively.
We can verified that $\J^{\top}\J=\I$ and
\begin{equation*}
    \J\hH(\vz)=\begin{bmatrix}
    \hat{\H}_{\x\x}(\vz) & \hat{\H}_{\x\y}(\vz)\\
    -\hH_{\x\y}(\vz) & -\hH_{\y\y}(\vz)
    \end{bmatrix}.
\end{equation*}
Thus we have 
\begin{align*}
\H(\vz)=\hH(\vz)^\top\hH(\vz)=\hH^\top(\vz)(\J^{\top}\J)\hH(\vz)=(\J\hH)^{\top}(\J\hH).
\end{align*}
Following Lemma 4.5 of \citet{abernethy2019lastiterate}, we have 
\begin{equation*}
    \H(\vz) \succeq \mu^2 \I \succ \0.
\end{equation*}
Since $\big\|\hH\big\|\leq L$ and $\H\succ\0$, we have
\begin{equation*}
    \Norm{\H(\z)} = \big\|\hH(\z)^2\big\| 
\leq \big\|\hH(\z)\big\|^2 \leq L^2.
\end{equation*}
and
\begin{equation*}
    \H(\z) \preceq L^2 \I.
\end{equation*}
 \end{proof}

\section{The Proofs for \ref{sec-quadratic}}

This section provides the proofs of the results in Section \ref{sec:Broyden update} and \ref{sec-quadratic}.

\subsection{The Proof of Lemma \ref{lmquadraticlamb}}
\begin{proof}
We have $\g_k=\A\z_k-b$ and $\g_{k+1}=\A\z_{k+1}-b$ where $\A=\hH$, which means
\begin{equation}
    \g_{k+1}-\g_k=\A(\z_{k+1}-\z_{k})=-\hH\G_k^{-1}\hH\g_k.
\end{equation}
So we have
\begin{equation}
   \lambda_{k+1}=\|(\I-\hH\G_{k}^{-1}\hH)\g_k\|\leq\|\I-\hH\G_k^{-1}\hH\|\lambda_k.
\end{equation}
Since $\H\preceq\G_k\preceq \eta_k\H$ and $\H=(\hH)^2$, According to Lemma \ref{lm-matrix-neq}, we have $\|\I-\hH\G_k^{-1}\hH\|\leq\left(1-\frac{1}{\eta_k}\right)$, which means
\begin{equation}
    \lambda_{k+1}\leq \left(1-\frac{1}{\eta_k}\right)\lambda_k.
\end{equation}
\end{proof}

\subsection{The Proof of Theorem \ref{thm-quadratic-lambda}}
\begin{proof}
Lemma \ref{lm-matrix-neq0} means $\mu^2\I\preceq\H\preceq L^2\I$.
Combining with $\G_0=L^2\I$, we have
\begin{equation*}
    \H\preceq\G_0\preceq\kappa^2\H.
\end{equation*}
According to Lemma~\ref{lmBroydupdate}, we achieve \eqref{eqquadraticmatrix1}. Applying Lemma \ref{lmquadraticlamb}, we obtain
\begin{equation}
\label{eq-qua-lambda_0}
    \lambda_{k+1}\leq \left(1-\frac{1}{\kappa^2}\right)\lambda_k.
\end{equation}
for all $k\geq 0$ which leads to \eqref{eqlambda1}
\end{proof}

\subsection{The Proof of Theorem \ref{thm:BroydenQuadratic}}
\begin{proof}
The proof is modified from \cite[Theorem 16]{lin2021faster}.
Denote $\eta_k=\|\H^{-1/2}\G_k\H^{-1/2}\|\geq1$, 
then $\H\preceq\G_k\preceq\eta_k\H$
and we have
\begin{eqnarray}
\label{eq:etasigmaneq}
1-\frac{1}{\eta_k}\leq\eta_k-1\leq \sigma_{\H}(\G_k)\leq \frac{{\rm tr}(\G_k-\H)}{\mu^2}=\frac{\tau_{\H}(\G_k)}{\mu^2}.
\end{eqnarray}
\textbf{BFGS method:}
According to Lemma \ref{thmBFGSupdate}, for both random and greedy method we have:
\begin{equation*}
   \EBP{\sigma_{\H}(\G_{k+1})}\leq\left(1-\frac{1}{n}\right)\EBP{\sigma_{\H}(\G_{k})}.
\end{equation*}
which implies
\begin{equation*}
      \EBP{\sigma_{\H}(\G_k)} \leq \left(1-\frac{1}{n}\right)^k\sigma_{\H}(\G_0).
\end{equation*}
Thus we have
\begin{equation}
\label{eq:expectetak}
 \EBP{(\eta_k-1)} \overset{\eqref{eq:etasigmaneq}}{\leq} \EBP{\sigma_{\H}(\G_k)}\leq\left(1-\frac{1}{n}\right)^k\sigma_{\H}(\G_0),   
\end{equation}
The upper bound of $\sigma_\H(\G_0)$ means
\begin{equation}
\label{eq:sigmaupperbound}
    \sigma_{\H}(\G_0)=\langle\H^{-1},\G_0 \rangle-n \leq \langle\H^{-1},\kappa^2\H \rangle -n =n\left(\kappa^2-1\right)\leq n\kappa^2,
\end{equation}
Combining with Lemma \ref{lmquadraticlamb}, we have
\begin{equation*}
\EBP{\frac{\lambda_{k+1}}{\lambda_{k}}}\leq \EBP{1-\frac{1}{\eta_k}} \leq \EBP{\eta_k-1} \overset{\eqref{eq:expectetak}}{\leq} \left(1-\frac{1}{n}\right)^k\sigma_{\H}(\G_0)\overset{\eqref{eq:sigmaupperbound}}{\leq}\left(1-\frac{1}{n}\right)^k n\kappa^2
\end{equation*}

\paragraph{Broyden Family Method:}
The proof for the Broyden family method is similar to the one of BFGS. According to Lemma \ref{thmbroydenupdate}, we have
\begin{equation*}
    \EBP{\sigma_{\H}(\G_{k+1})} \leq \left(1-\frac{1}{n\kappa^2}\right)\EB~\sigma_{\H}(\G_k).
\end{equation*}
which means
\begin{equation*}
     \EBP{\sigma_{\H}(\G_k)} \leq \left(1-\frac{1}{n\kappa^2}\right)^k\sigma_{\H}(\G_0).
\end{equation*}
The upper bound bound of $\sigma_{\H}(\G_0)$ is the same with the BFGS which implies
 \begin{equation*}
\EBP{\frac{\lambda_{k+1}}{\lambda_{k}} }\leq \EBP{\eta_k-1} \leq \left(1-\frac{1}{n\kappa^2}\right)^k\sigma_{\H}(\G_0)\overset{\eqref{eq:sigmaupperbound}}{\leq}\left(1-\frac{1}{n\kappa^2}\right)^k n\kappa^2
 \end{equation*}

\paragraph{SR1 Method:}
Using Theorem 12 of \citet{lin2021faster}, we have
\begin{equation*}
\EBP{\tau_{\H}(\G_k)}\leq\left(1-\frac{k}{n}\right)\tau_{\H}(\G_0).
\end{equation*}
The upper bound of $\tau_{\H}(\G_k)$ means
\begin{align*}
  \tau_{\H}(\G_0) =  {\rm tr}(\G_0-\H) \leq (\kappa^2-1) {\rm tr}(\H) \leq n\kappa^2 L^2
\end{align*}
Combining with Lemma \ref{lmquadraticlamb}, we have
\begin{equation*}
    \EBP{\frac{\lambda_{k+1}}{\lambda_k}} \leq \EBP{\eta_k-1} \overset{\eqref{eq:etasigmaneq}}{\leq} \EBP{\frac{\tau_{\H}(\G_k)}{\mu^2}}\leq \left(1-\frac{k}{n}\right)\frac{n\kappa^2L^2}{\mu^2}=\left(1-\frac{k}{n}\right) n\kappa^4
\end{equation*}
\end{proof}

\subsection{The proof of Corollary \ref{co:quadraticresult}}
\begin{proof}
The convergence behaviors of the algorithms have two stages, the first one is global linear convergence and the second one is local superlinear convergence.

\paragraph{Greedy BFGS Method:}
Let $k_0$ be the iteration number of the first stage. We hope it holds that
\begin{equation}
\label{eq:bfgsresultneq}
    \left(1-\frac{1}{n}\right)^{k_0}n\kappa^2\leq \frac{1}{2}.
\end{equation}
after $k_0$ iterations.
Clearly, we require $k_0\leq n\ln\left(2n\kappa^2\right)$.
According to inequality \eqref{eq-qua-lambda_0} in Theorem \ref{thm-quadratic-lambda}, during the first $k_0$ iterations, we have
\begin{equation*}
\lambda_k \leq \left(1-\frac{1}{\kappa^2}\right)^k\lambda_0.
\end{equation*}
After that, Theorem \ref{thm:BroydenQuadratic} means we have 
\begin{eqnarray*}
\lambda_{k+k_0}&\leq&\lambda_{k_0}\prod_{i=0}^{k-1}\left[\left(1-\frac{1}{n}\right)^{i}\frac{1}{2}\right]\\
               &=& \left(1-\frac{1}{n}\right)^{\frac{k(k-1)}{2}}\left(\frac{1}{2}\right)^k\lambda_{k_0}    \\
               &\leq&\left(1-\frac{1}{n}\right)^{\frac{k(k-1)}{2}}\left(\frac{1}{2}\right)^k \left(1-\frac{1}{\kappa^2}\right)^{k_0}\lambda_0.
\end{eqnarray*}
for all $k\geq0$.

\paragraph{Greedy Broyden Family Method:}
The proof for Broyden method is almost the same as the one of BFGS. 
We only need to replace inequality \eqref{eq:bfgsresultneq} to 
\begin{equation*}
    \left(1-\frac{1}{n\kappa^2}\right)^{k_0}n\kappa^2\leq \frac{1}{2},
\end{equation*}
and replace all the terms of $\left(1-\frac{1}{n}\right)$ to  $\left(1-\frac{1}{n\kappa^2}\right)$. The reason is Theorem \ref{thm:BroydenQuadratic} only provides a slower convergence result
\begin{align*}
    \lambda_{k+1}\leq\left(1-\frac{1}{n\kappa^2}\right)^kn\kappa^2\lambda_{k}
\end{align*}
for Broyden family update, rather than 
\begin{align*}
    \lambda_{k+1}\leq\left(1-\frac{1}{n}\right)^kn\kappa^2\lambda_{k}  
\end{align*}
for BFGS.

\paragraph{Greedy SR1 Method:}
The convergence of SR1 update also has two stages. We hope it holds that 
\begin{equation*}
    \left(1-\frac{k}{n}\right)n\kappa^4\leq \frac{1}{2},
\end{equation*}
after $k_0$ iterations. Clearly, let $k_0=\left\lceil\left(1-\dfrac{1}{2n\kappa^4}\right)n\right\rceil$ satisfies the condition.
And according to Theorem \ref{thmbroydenupdate} we have  $\lambda_{k}\leq\left(1-\frac{k}{n}\right)n\kappa^4\lambda_k$, which means that at most after n iterations, we have $\lambda=0$.

\end{proof}

\subsection{The proof of Corollary \ref{co:quadraticresult-random}}
\begin{proof}
The convergence behaviors of the random algorithms also have two stages, the first one is global linear convergence and the second one is local superlinear convergence.
We consider the random variable $\X_k=\lambda_{k+1}/\lambda_k, \forall k \geq 0 $ in the following derivation.
\paragraph{Random Broyden Family Method:}
	Note that $\mX_k \geq 0$, using Markov's inequality, we have for any $\epsilon>0$, 
	\begin{equation}\label{eq:qua_pp}
	\BP\left(\X_k \geq \frac{n\kappa^2}{\epsilon} \left(1-\frac{1}{n\kappa}\right)^{k} \right) \leq \frac{\EBP{\ \mX_k}}{\frac{n\kappa^2}{\epsilon} \left(1-\frac{1}{n\kappa^2}\right)^{k}} \stackrel{\eqref{eq-qua-g-fth-result-2}}{\leq} \epsilon.
	\end{equation}
	Choosing $\epsilon_k = \delta (1-q)q^k$ for some positive $q < 1$, then we have 
	\begin{equation*}
	\begin{aligned}
	\BP\left(\mX_k \geq \frac{n\kappa^2}{\epsilon_k} \left(1-\frac{1}{n\kappa^2}\right)^{k}, \exists  k\in\BN\right) 
	\leq&  \sum_{k=0}^\infty \BP\left(\mX_k
	\geq \frac{n\kappa^2}{\epsilon_k}\left(1-\frac{1}{n\kappa^2}\right)^{k} \right) \\
	\stackrel{(\ref{eq:qua_pp})}{\leq}&  \sum_{k=0}^\infty \epsilon_k = \sum_{k=0}^\infty \delta (1-q)q^k = \delta.
	\end{aligned}
	\end{equation*}
	Therefore, we obtain with probability $1-\delta$,
	\begin{equation*}
	\mX_k \leq \left(\frac{1-\frac{1}{n\kappa^2}}{q}\right)^{k} \cdot \frac{n\kappa^2}{(1-q)\delta}, \ \forall k\in\BN.
	\end{equation*}
	If we set $q = 1 - \frac{1}{n^2\kappa^4}$, we could obtain with probability $1-\delta$, for all $k \in \BN$,
	\begin{equation*}
	\mX_k \leq \frac{n^3\kappa^6}{\delta} \left(1+\frac{1}{n\kappa^2}\right)^{-k} = \frac{n^3\kappa^6}{\delta} \left(1-\frac{1}{n\kappa^2+1}\right)^{k}.
	\end{equation*}
	Furthermore, it holds with probability $1-\delta$ that
	\begin{equation}\label{eq:r_qua_tel}
	\frac{\lambda_{k+1}}{\lambda_k} \le \frac{2n^3\kappa^6}{\delta} \left(1-\frac{1}{n\kappa^2+1}\right)^{k}, \forall k \in \BN, 
	\end{equation}
	Telescoping from $k$ to $0$ in Eq.~\eqref{eq:r_tel}, we get
	\begin{align*}
	\lambda_{k} &= \lambda_0 \cdot \prod_{i=1}^k \frac{\lambda_i}{\lambda_{i-1}} \leq \lambda_0 \cdot \left(\frac{2n^3\kappa^6}{\delta}\right)^{k} \prod_{i=1}^k \left(1 - \frac{1}{n\kappa^2+1}\right)^{i-1} \\
	&= \left(\frac{2n^3\kappa^6}{\delta}\right)^{k}\left(1-\frac{1}{n\kappa^2+1}\right)^{k(k-1)/2} \lambda_0.
	\end{align*}

Let $k_0$ be the iteration number of the first stage. We hope it holds that
\begin{equation}
\label{eq:bfgsresultneq-random}
    \left(1-\frac{1}{n\kappa^2+1}\right)^{k_0}\frac{2n^3\kappa^6}{\delta}\leq \frac{1}{2}.
\end{equation}
after $k_0$ iterations.
Clearly, we require $k_0=\OM\left( n\kappa^2\ln\left(n\kappa\right/\sigma)\right)$.
According to inequality \eqref{eq-qua-lambda_0} in Theorem \ref{thm-quadratic-lambda}, during the first $k_0$ iterations, we have
\begin{equation*}
\lambda_k \leq \left(1-\frac{1}{\kappa^2}\right)^k\lambda_0.
\end{equation*}
After that, \eqref{eq:r_qua_tel} means that with probability $1-\delta$ we have 
\begin{eqnarray*}
\lambda_{k+k_0}&\leq&\lambda_{k_0}\prod_{i=0}^{k-1}\left[\left(1-\frac{1}{n\kappa^2+1}\right)^{i}\frac{1}{2}\right]\\
               &=& \left(1-\frac{1}{n\kappa^2+1}\right)^{\frac{k(k-1)}{2}}\left(\frac{1}{2}\right)^k\lambda_{k_0}    \\
               &\leq&\left(1-\frac{1}{n\kappa^2+1}\right)^{\frac{k(k-1)}{2}}\left(\frac{1}{2}\right)^k \left(1-\frac{1}{\kappa^2}\right)^{k_0}\lambda_0.
\end{eqnarray*}
for all $k\geq0$.

\paragraph{Random BFGS Method:}
Similar to the analysis for random Broyden family method, 
we obtain with probability $1-\delta$,
	\begin{equation*}
	\X_{k} \leq \left(\frac{1-\frac{1}{n}}{q}\right)^k \cdot \frac{n\kappa^2}{(1-q)\delta}, \forall k\in\BN.
	\end{equation*}
	If we set $q = 1 - 1 / n^2$, we could obtain with probability $1-\delta$,
	\begin{equation}
	\label{eq:lambda-qua-randomsuper}
 \lambda_{k+1} \leq  \frac{2n^3\kappa^2}{\delta}\left(1-\frac{1}{n+1}\right)^{k} \lambda_k, \text{ for all } k \in \BN.
	\end{equation}
We require the point $\vz_{k_0}$ sufficient close to the saddle point such that
\begin{equation*}
    \left(1-\frac{1}{n+1}\right)^{k_0}2n^3\kappa^2/\delta \leq \frac{1}{2},
\end{equation*}
which can be guaranteed by setting $k_0=\OM\left( n\ln(n\kappa/\delta)\right)$.

The remainder of the proof can follow the analysis in random Broyden methdos. 
We only need to replace all the term of $\left(1-\frac{1}{n\kappa^2+1}\right)$ to  $\left(1-\frac{1}{n+1}\right)$. The reason is that \eqref{eq:lambda-qua-randomsuper}  provides a faster convergence result for random BFGS update, rather than \eqref{eq:r_qua_tel}
for random Broyden case.

\paragraph{SR1 Method:}
Similar to the analysis for random Broyden family method, 
we obtain with probability $1-\delta$,
	\begin{equation*}
	X_{k} \leq \left(\frac{1-\frac{k}{n}}{q}\right)\cdot \frac{n\kappa^4}{(1-q)\delta} \qquad \text{ for all } k \in \BN.
	\end{equation*}
	If we set $q = 1 - 1 / n^2$, we could obtain with probability $1-\delta$,
	\begin{equation}
	\label{eq:lambda-qua-srq-randomsuper}
    \lambda_{k+1} \leq  \frac{n^2(n+1)\kappa^4}{\delta}\left(1-\frac{k}{n}\right)\lambda_k \qquad \text{ for all } 0\leq k\leq n .
	\end{equation}
Recall that we denote $k_0$ the first iteration such that
\begin{equation*}
    \left(1-\frac{k_0}{n}\right)\frac{n^2(n+1)\kappa^4}{\delta}\leq \frac{1}{2}.
\end{equation*}
Clearly, we can set
\begin{align*}
k_0=\left\lceil\left(1-\dfrac{\delta}{2n^2(n+1)\kappa^4}\right)n\right\rceil.
\end{align*}
Then it holds that
\begin{eqnarray*}
     \lambda_{k_0+k+1}
&\overset{\eqref{eq:lambda-qua-srq-randomsuper}}{\leq}& \frac{n^2(n+1)\kappa^4}{\delta} \left(1-\frac{k+k_0}{n}\right)\\
&=& \left(\frac{n-k-k_0}{n-k_0}\right)\left(1-\frac{k_0}{n}\right)\frac{n^2(n+1)\kappa^4}{\delta}\\
&\leq&\frac{1}{2}\left(\frac{n-k-k_0}{n-k_0}\right)\lambda_{k_0+k}.
\end{eqnarray*}

Thus for $k_0=\left\lceil\left(1-\delta/(2n^2(n+1)\kappa^4)\right)n\right\rceil$ and $0<k\leq n-k_0+1$, we have
\begin{eqnarray*}
        \lambda_{k+k_0}\leq \left(\frac{n-k+1-k_0}{n-k_0}\cdots\frac{n-k_0-1}{n-k_0}\right)\left(\frac{1}{2}\right)^k\lambda_{k_0} \leq \frac{(n-k_0-k+1)!}{(n-k_0+1)!}\left(\frac{1}{2(n-k_0)}\right)^{k}\left(1-\frac{1}{\kappa^2}\right)^{k_0}\lambda_0
\end{eqnarray*}

\end{proof}

\section{The Proofs for Section \ref{sec-general}}
This section provides the proofs of the results in Section \ref{sec-general}.

\subsection{The Proof of Lemma \ref{lm-H-continous}}
\begin{proof}
Since Assumption \ref{assbasic} and \ref{assbasic2} mean operator $\hH(\cdot)$ is $L$-Lipschitz continuous and $\|\hH(\z)\|\leq L$ for all $\z\in\BR^n$, we have
\begin{eqnarray*}
\|\H(\vz)-\H(\vz')\| &=& \|\hat{\H}(\vz)\hat{\H}(\vz)-\hat{\H}(\vz')\hat{\H}(\vz')\| \\
&\leq& \|\hat{\H}(\vz)(\hat{\H}(\vz)-\hat{\H}(\vz'))\| + \|(\hat{\H}(\vz)-\hat{\H}(\vz'))\hat{\H}(\z')\| \\
&\leq& 2\|\hat{\H}\|\cdot\|\hat{\H}(\vz)-\hat{\H}(\vz')\|\\
&\leq& 2L_2L \|\vz-\vz'\|.
\end{eqnarray*}
\end{proof}

\subsection{The Proof of Lemma \ref{lm-self-concordant}}
\begin{proof}
According to Lemma \ref{lm-H-continous}, we have
\begin{eqnarray*}
\H(\vz)-\H(\vz') &\stackrel{\eqref{eq10}}\preceq& 2L_2L \|\vz-\vz'\| \I \\
            &\stackrel{\eqref{eq6}}\preceq& \frac{2L_2L}{\mu^2} \|\vz-\vz'\| \H(\vw)\\
            &=& \frac{2\kappa^2L_2}{L}\|\vz-\vz'\| \H(\vw).
\end{eqnarray*}
\end{proof}

\subsection{The Proof of Corollary \ref{co:concordant}}
\begin{proof}
Taking interchanging of $\z'$ and $\z$ and letting $\w=\z$ in \eqref{eq-self-concordant}, we have
\begin{equation*}
    \H(\z')-\H(\z)\preceq M\|\z-\z'\|\H(\z)=Mr \|\H(\z)\|.
\end{equation*}
Then taking $\w=\z'$, we have
\begin{equation*}
    \H(\z)-\H(\z')\preceq M\|\z-\z'\|\H(\z') =Mr \|\H(\z')\|.
\end{equation*}
We finish the proof by combing above results.
\end{proof}

\subsection{The Proof of Lemma \ref{lm-G-H-neq}}
\begin{proof}
Using Corollary \ref{co:concordant}, we have 
\begin{equation*}
    \H(\z_{+})\overset{\eqref{eq-H-concordant-app}}{\preceq} (1+Mr)\H(\z) \overset{\eqref{eq-G-H-neq-0}}{\preceq}(1+Mr)\G=\tilde{\G}
\end{equation*}
and
\begin{equation*}
    \tilde{\G} = (1+Mr)\G\overset{\eqref{eq-G-H-neq-0}}{\preceq}(1+Mr)\eta\H(\z)\overset{\eqref{eq-H-concordant-app}}{\preceq}(1+Mr)^2\eta\H(\z_{+}),
\end{equation*}
which means
\begin{equation*}
    \H(\z_+)\preceq \tilde{\G}\preceq (1+Mr)^2\eta\H(\z_{+}).
\end{equation*}
Then we obtain \eqref{eq-G-H-neq} by Lemma \ref{lmBroydupdate}.
\end{proof}

\subsection{The Proof of Lemma \ref{thm-lambda-general}}
\begin{proof}

We rewrite $\nabla f(\z_{k+1})-\nabla f(\z_k)$ below
\begin{align*}
\nabla f(\z_{k+1})-\nabla f(\z_k)& = \int_{0}^1\nabla^2 f(\z_k+s(\z_{k+1}-\z_k)) (\z_{k+1}-\z_k)ds \\
&=\int_{0}^1\left[\nabla^2 f(\z_k+s(\z_{k+1}-\z_k))-\nabla^2 f(\z_k)\right] (\z_{k+1}-\z_k)ds + \nabla^2f(\z_k)(\z_{k+1}-\z_k)\\
&=\int_{0}^1\left[\nabla^2 f(\z_k+s(\z_{k+1}-\z_k))-\nabla^2 f(\z_k)\right] (\z_{k+1}-\z_k)ds -\nabla^2f(\z_k)\G_k^{-1}\nabla^2f(\z_k)\nabla f(\z_k),
\end{align*}
which means
\begin{align*}
    \g_{k+1}=\underbrace{(\I -\hH_{k}\G_k^{-1}\hH_{k})\g_k}_{\va_k}+\underbrace{\int_{0}^1\left[\hH(\z_k+s(\z_{k+1}-\z_k))-\hH(\z_k)\right] (\z_{k+1}-\z_k)ds}_{\vb_k}.
\end{align*}
We first bound the term $\|\va_k\|$ by Lemma \ref{lm-matrix-neq}
\begin{equation}
    \|\va_k\|\leq\|\I-\hH_k\G_k^{-1}\hH_k\|\|\g_k\|\leq \left(1-\frac{1}{\eta_k}\right)\lambda_k.
\end{equation}
Before we bound $\vb_k$, we first try to bound $\hH_k\G_k^{-2}\hH_k$
\begin{align*}
        \hH_k\G_k^{-2}\hH_{k}&= (\hH_k\G_k^{-1/2})\G_k^{-1}(\G_k^{-1/2}\hH_k)\preceq (\hH_k\G_k^{-1/2})\frac{1}{\mu^2}\I(\G_k^{-1/2}\hH_k)=\frac{1}{\mu^2}\hH_k\G_k^{-1}\hH_k \\
      &\preceq \frac{1}{\mu^2} \hH_k\H_k^{-1}\hH_k \preceq \frac{1}{\mu^2} \I.
\end{align*}
And $\|\vb_k\|$ can be bounded by the $L_2$-Lipschitz continuous of the object function that
\begin{align*}
    \|\vb_k\|&\leq \int_{0}^1\Norm{\left(\hH\left(\z_k+s(\z_{k+1}-\z_k)\right)-\hH(\z_k)\right)(\z_{k+1}-\z_{k})}ds\\
&\stackrel{\eqref{eq-Hessian-continous}}\leq \int_{0}^{1}\Norm{ L_2 s(\z_{k+1}-\z_{k})}ds \leq \frac{L_2}{2} \|\z_{k+1}-\z_{k}\|^2\\
&=\frac{L_2}{2}\|\G_k^{-1}\hH_k\g_k\|^2= \frac{L_2}{2}\inner{\g_k}{\hH_k\G_k^{-2}\hH_k\g_k}\leq \frac{L_2}{2\mu^2}\|\g_k\|^2=\beta\lambda_k^2.\end{align*}
Combining the above results, we have
\begin{equation}
    \lambda_{k+1}\leq \|\va_k\|+\|\vb_k\|\leq\left(1-\frac{1}{\eta_k}\right)\lambda_k+\beta\lambda_k^2
\end{equation}
The relation of $r_k$ and $\lambda_k$ can be directly prove by the update formula
\begin{equation}
    r_k=\|\z_{k+1}-\z_k\|=\|\G_k^{-1}\hH_k\g_k\|= \inner{\g_k}{\hH_k\G_k^{-2}\hH_k\g_k}^{1/2}\leq\frac{1}{\mu}\lambda_k
\end{equation}
\end{proof}

\subsection{The proof of Theorem \ref{thm:G-linear-convergence}}
\begin{proof}
We use induction to prove the following statements 
\begin{equation}\label{eq64}
    \H_k\preceq\G_k\preceq \exp\left(2\sum_{i=0}^
    {k-1}\rho_i\right)\kappa^2\H_k\preceq b\kappa^2\H_k,
\end{equation}
\begin{equation}\label{eq65}
    \lambda_k \leq\left(1-\frac{1}{2b\kappa^2}\right)^{k}\lambda_{0},
\end{equation}
\begin{equation}
\label{eq66}
    \eta_k \defeq \exp{\left(\sum_{i=0}^
    {k-1}2\rho_i\right)}\kappa^2\leq b\kappa^2
\end{equation}
hold for all $k\geq 0$. 
The initial assumption promise that $\lambda_0$ is small enough such that
\begin{equation}
    \label{eq:initialass1}
    \frac{M}{\mu}\lambda_0\leq \frac{\ln b}{4b\kappa^2} 
\end{equation}
For $k=0$, the initialization $\G_0=L^2\I$ leads to $\H_0\preceq\G_0\preceq \kappa^2\H_0$ and $\eta_0=\kappa^2$, which  satisfy \eqref{eq64}, \eqref{eq65} and \eqref{eq66}.
Then we prove these results for $k'=k+1$.

The induction assumption means $\eta_k\leq b\kappa^2$ and $\H_k\preceq \G_k\leq \eta_k\H_k$. Using Lemma \ref{thm-lambda-general}, we have
\begin{align*}
    \lambda_{k+1} \leq \left(1-\frac{1}{b\kappa^2}\right)\lambda_k + \beta \lambda_k^2 
    \overset{\eqref{eq65}}{\leq}  \left(1-\frac{1}{b\kappa^2}+\beta\lambda_0\right)\lambda_{k}
    \overset{\eqref{eq:initialass1}}{\leq} \left(1-\frac{1}{2b\kappa^2}\right)\lambda_k 
    \overset{\eqref{eq65}}{\leq}  \left(1-\frac{1}{2b\kappa^2}\right)^{k+1}\lambda_0.
\end{align*}
Recall that we've defined $\rho_i=M\lambda_i/\mu$. Based on the fact $e^{x}\geq x+1$ and Lemma \ref{lm-G-H-neq}, we have
\begin{eqnarray*}
    \H_{k+1}\preceq\G_{k+1}  &\overset{\eqref{eq-G-H-neq}}{\preceq}& (1+Mr_k)^2 \eta_k \H_{k+1}\preceq(1+M\lambda_k/\mu)^2 \eta_k \H_{k+1}\\
    &{=}& (1+\rho_k)^2\eta_k \H_{k+1} \preceq e^{2\rho_k}\eta_k \H_{k+1} \\
    &\overset{\eqref{eq64}}{\preceq}& \exp{\left(2\sum_{i=0}^{k}\rho_i\right)}\kappa^2 \H_{k+1};
\end{eqnarray*}
and the term $\sum_{i=0}^{k}\rho_i$ can be bounded by
\begin{eqnarray}
\label{eq-gamma-control}
    \sum_{i=0}^{k}\rho_i &{=}& \frac{M}{\mu}\sum_{i=0}^{k}\lambda_i \overset{\eqref{eq65}}{\leq} \frac{M}{\mu}\lambda_0\sum_{i=0}^{k}\left(1-\frac{1}{2b\kappa^2}\right)^{i} \\&\leq&\frac{2bML^2}{\mu^3}\lambda_0 \overset{\eqref{eq:initialass1}}{\leq}\frac{\ln b}{2}.\nonumber
\end{eqnarray}
Hence, for $k+1$, we have
\begin{eqnarray*}
    \G_{k+1}\preceq \exp{\left(2\sum_{i=1}^
    {k}\rho_i\right)}\kappa^2\H_{k+1}\preceq b\kappa^2\H_{k+1}
\quad\text{and}\quad
    \eta_{k+1}= \exp{\left(2\sum_{i=1}^
    {k}\rho_i\right)}\kappa^2\leq b\kappa^2.
\end{eqnarray*}
Thus we have complete the proof for statements \eqref{eq64}, \eqref{eq65} and \eqref{eq66} by induction.
\end{proof}

\subsection{The Proof of Lemma \ref{lm:Broydensigma-update}}
\begin{proof}
The proof of two algorithms are similar, the only difference is because two update formulas have different convergence rate.
We first give the proof for BFGS method.
\paragraph{BFGS Method:}
Recall that we use the update rule
\begin{equation*}
    \G_{k+1} = {\rm BFGS}\left(\tilde{\G}_k,\H_{k+1},\u_k\right).
\end{equation*}
Using Lemma \ref{thmBFGSupdate}, we have
\begin{equation*}
\EB_{\u_k}[\sigma_{k+1}] = \EB_{\u_k}[\sigma_{\H_{k+1}}(\G_{k+1})]\overset{\eqref{eqbfgsthm1}}{\leq} \left(1-\frac{1}{n}\right) \sigma_{\H_{k+1}}(\tilde{\G}_{k}).
\end{equation*}
Then we bound the term $\sigma_{\H_{k+1}}(\tilde{\G}_k)$ by $\sigma_k$ as follows
\begin{eqnarray}
\sigma_{\H_{k+1}}(\tilde{\G}_k)&\overset{\eqref{eq-def-sigma-H}}{=}&\langle \H_{k+1}^{-1},\tilde{\G}_k\rangle-n \nonumber\\
&\overset{\eqref{eq-G-update}}{=}&(1+Mr_k)\langle\H_{k+1}^{-1}, \G_k \rangle -n \nonumber\\
&\overset{\eqref{eq-H-concordant-app}}{\leq}&(1+Mr_k)^2\langle \H_k^{-1}, \G_k \rangle - n\nonumber\\
&=&(1+Mr_k)^2\sigma_k +n((1+Mr_k)^2-1)\nonumber\\
&=&(1+Mr_k)^2\sigma_k+2nMr_k\left(1+\frac{Mr_k}{2}\right)\nonumber\\
\label{last-sigma}
&\leq& (1+Mr_k)^2\left(\sigma_k + \frac{2nMr_k}{1+Mr_k}\right).
\end{eqnarray}
Thus we obtain \eqref{eq-sigma-update-2} by combining above results.
\paragraph{Broyden Family Method:}
If we use the Broyden family update
\begin{equation*}
    \G_{k+1} = {\rm Broyd}_{\tau_k}(\tilde{\G}_k,\H_{k+1},\u_k)
\end{equation*}
instead of BGFS, Lemma \ref{thmbroydenupdate} means
\begin{equation*}
    \EB_{\u_k}[\sigma_{k+1}]\overset{\eqref{eqbroythm1}}{\leq}\left(1-\frac{1}{n\kappa^2}\right)\sigma_{\H_{k+1}}(\tilde{\G}_k).
\end{equation*}
Combining with \eqref{last-sigma} which holds for both Broyden case and BFGS case, we obtain \eqref{eq-sigma-update-1}.
\end{proof}

\subsection{The Proof of Theorem \ref{thm-Broyden-superlinear}}
We give the proof for the BFGS and Broyden Family methods in Section \ref{app:saddle-thm-bfgs-broyden} and the proof for the SR1 method in Section \ref{app:saddle-thm-sr1}.
Taking $b=2$, all algorithms have $M\lambda_0/\mu \leq \frac{\ln2}{8\kappa^2}$, which implies the properties of $\lambda_k$ shown in Theorem \ref{thm:G-linear-convergence}.
\subsubsection{The Proofs of BFGS and Broyden Methods}\label{app:saddle-thm-bfgs-broyden}
\paragraph{BFGS Method:}
First we give the proof for the BFGS method.
For each $k$, we have
\begin{equation}
\label{eq-fth-1111}
    \H_{k}\preceq\G_k \overset{\eqref{eq-sigmaproperty-1}}{\preceq}(1+\sigma_k)\H_k.
\end{equation}
From Theorem \ref{thm-lambda-general}, we have
\begin{equation}
\label{eq-fth-1}
    \lambda_{k+1}\leq \left(1-\frac{1}{1+\sigma_k}\right)\lambda_k+\beta\lambda_k^{2},
\end{equation}
and
\begin{equation}
\label{eq:rk-lambdak}
    r_k \leq \lambda_k/\mu.
\end{equation}
From Lemma \ref{lm:Broydensigma-update}, we obtain
\begin{equation}
\label{eq-fth-2}
    \EB_{\u_k}[\sigma_{k+1}] \leq \left(1-\frac{1}{n}\right)(1+Mr_k)^2\left(\sigma_k + \frac{2nMr_k}{1+M r_k}\right).
\end{equation}
Since we have defined $\rho_k=M\lambda_k/\mu$ and the constant $\beta,~M$ satisfy $\frac{\beta}{M}<\frac{1}{4}$, it holds that
\begin{eqnarray}
\label{eq-fth-3}
\rho_k \overset{\eqref{eq-fth-1}}{\leq} \frac{\sigma_k}{1+\sigma_k}\rho_k +\frac{\beta}{2M}\rho_k^2 \leq \sigma_k\rho_k +\frac{\beta}{2M}\rho_k^2<\sigma_k\rho_k+\frac{1}{8}\rho_k^2,
\end{eqnarray}
and
\begin{eqnarray}
\label{eq-fth-4}
\EBP{\sigma_{k+1}}&\overset{\eqref{eq-fth-2},\eqref{eq:rk-lambdak}}{\leq}&\left(1-\frac{1}{n}\right)\EBP{(1+\rho_k)^2\left(\sigma_k+\frac{2n\rho_k}{1+\rho_k}\right)}.
\end{eqnarray}
We set 
\begin{equation*}
    \theta_k \defeq \sigma_k+2n\rho_k
    \end{equation*}
and consider Theorem \ref{thm:G-linear-convergence} with $b=2$. 
Then the convergence result of \eqref{eq-gamma-control} and the initial assumption of $\z_0$ implies that
\begin{eqnarray}
\label{eq:rhokbound}
        \rho_k \leq \left(1-\frac{1}{4\kappa^2}\right)^k\rho_0,
\end{eqnarray}
and 
\begin{eqnarray}
\label{eq:initialrho1}
        \rho_0 \leq \frac{\ln2}{8\kappa^2(1+2n)}
\end{eqnarray}
We now use induction to show that
\begin{equation}
\label{eq:induction}
    \EBP{\theta_k}\leq \left(1-\frac{1}{n}\right)^{k}2n\kappa^2
\end{equation}
In the case of $k=0$, we have
\begin{eqnarray}
\label{eq:initialboundtheta0}
    \sigma_0+2n\rho_0 
    \overset{\eqref{eq-def-sigma-H}}{=} \langle\H_0^{-1},\G_0\rangle-n +2n\rho_0
    \leq\langle\H_0^{-1},\kappa^2\H_0\rangle-n+2n\rho_0
    =n\left(\kappa^2-1\right)+2n\rho_0\leq n\kappa^2.
\end{eqnarray}
Thus for $k=0$, \eqref{eq:induction} is satisfied.

Suppose inequality \eqref{eq:induction} holds for $0\leq k'\leq k$. 
For $k+1$, using the inequality $e^x\geq1+x$, we have
\begin{align}\label{eq-fth-sigma}
\begin{split}
\EBP{\sigma_{k+1}} \overset{\eqref{eq-fth-4}}{\leq}& \left(1-\frac{1}{n}\right)\EBP{(1+\rho_k)^2\left(\sigma_k+\frac{2n\rho_k}{1+\rho_k}\right)} \\
\leq& \left(1-\frac{1}{n}\right)\EBP{(1+\rho_k)^2(\sigma_k+2n\rho_k)}\\
=& \left(1-\frac{1}{n}\right)\EBP{(1+\rho_k)^2\theta_k} \\
\leq&\left(1-\frac{1}{n}\right)\EBP{e^{2\rho_k}\theta_k},
\end{split}
\end{align}
and
\begin{eqnarray}
\label{eq-fth-gammak-update}
    \rho_{k+1} \leq \rho_k\left(\sigma_k+\frac{1}{8}\rho_k\right) \leq \rho_k\left(\sigma_k+2n\rho_k\right) \leq \left(1-\frac{1}{n}\right) 2\exp{\left(2\rho_k\right)} \theta_k \rho_k
\end{eqnarray}
Thus we obtain
\begin{eqnarray*}
  \EBP{ \sigma_{k+1} +2n\rho_{k+1}} &\overset{\eqref{eq-fth-gammak-update},\eqref{eq-fth-sigma}}{\leq}&\left(1-\frac{1}{n}\right)\EBP{\exp{\left(2\rho_k\right)}\theta_k}+\left(1-\frac{1}{n}\right) 4n\EBP{ \exp{\left(2\rho_k\right)} \theta_k \rho_k}\\
                &\leq&\left(1-\frac{1}{n}\right)\EBP{\exp{\left(2\rho_k\right)}\theta_k(1+4n\rho_k)}\\
                &\leq&\left(1-\frac{1}{n}\right) \EBP{\exp{\left(2\rho_k\right)}\exp{\left(4n\rho_k\right)}\theta_k}\\
                &\leq&\left(1-\frac{1}{n}\right)\EBP{\exp{\left(2(1+2n)\rho_k\right)}\theta_k}\\
               &\overset{\eqref{eq:rhokbound}}{\leq}&\left(1-\frac{1}{n}\right) \EBP{\exp{\left(2(1+2n)\left(1-\frac{1}{4\kappa^2}\right)^{k}\rho_0\right)}\theta_k}\\
              &=& \left(1-\frac{1}{n}\right) \exp{\left(2(1+2n)\left(1-\frac{1}{4\kappa^2}\right)^{k}\rho_0\right)} \EBP{\theta_k}
\end{eqnarray*}
Therefore, we have
\begin{eqnarray}
        \EBP{\theta_{k+1}} &\leq& \left(1-\frac{1}{n}\right) \exp{\left(2(1+2n)\left(1-\frac{1}{4\kappa^2}\right)^{k}\rho_0\right)} \EBP{\theta_k} \\
        &\leq& \left(1-\frac{1}{n}\right)^{k+1} \exp{\left(2(1+2n)\rho_0\sum_{i=0}^{k}\left(1-\frac{1}{4\kappa^2}\right)^{i}\right)} \EBP{\theta_0}\\
        &\leq& \left(1-\frac{1}{n}\right)^{k+1} \exp{\left(8\kappa^2(1+2n)\rho_0\right)} \EBP{\theta_0}\\
        &\overset{\eqref{eq:initialrho1},\eqref{eq:initialboundtheta0}}{\leq}& \left(1-\frac{1}{n}\right)^{k+1} 2n\kappa^2
\end{eqnarray}
which proves \eqref{eq:induction}.
Hence, for any $k\geq0$, we have 
\begin{equation*}
    \EBP{\sigma_{k}}\leq \EBP{\theta_k} {\leq}\left(1-\frac{1}{n}\right)^{k}2n\kappa^2.
\end{equation*} 
which implies 
\begin{equation}
  \EBP{\frac{\lambda_{k+1}}{\lambda_k}}=\EBP{\frac{ \rho_{k+1}}{\rho_k}} \overset{\eqref{eq-fth-gammak-update}}{\leq}  \EBP{\sigma_k+2n\rho_k}\leq\EBP{\theta_k}\leq \left(1-\frac{1}{n}\right)^k2n\kappa^2.
\end{equation}

\paragraph{Broyden Family Method:}
The proof for the Broyden method is almost the same as the one in the BFGS method. The reason that produce the different convergence result between Broyden and BGGS is that Lemma \ref{lm:Broydensigma-update} only provides a slower convergence rate
\begin{equation*}
  \EB[\sigma_{k+1}] \leq \left(1-\frac{1}{n\kappa^2}\right)(1+Mr_k)^2\left(\sigma_k + \frac{2nMr_k}{1+M r_k}\right)
\end{equation*}
for Broyden method, rather than 
\begin{equation*}
      \EB[\sigma_{k+1}] \leq \left(1-\frac{1}{n}\right)(1+Mr_k)^2\left(\sigma_k + \frac{2nMr_k}{1+M r_k}\right)
\end{equation*}
for BFGS. 
Thus we can directly replace the term $\left(1-\frac{1}{n}\right)$ to the term $\left(1-\frac{1}{n\kappa^2}\right)$ in the proof of BFGS method and obtain the convergence result of Broyden method
\begin{equation*}
        \EBP{\sigma_{k}}\leq \left(1-\frac{1}{n\kappa^2}\right)^{k}2n\kappa^2,
\end{equation*}
and 
\begin{equation*}
      \EBP{\frac{\lambda_{k+1}}{\lambda_k}}\leq \left(1-\frac{1}{n\kappa^2}\right)^k2n\kappa^2
\end{equation*}
in \eqref{eq-g-fth-result-2}.

\subsubsection{The Proof of SR1 Method}\label{app:saddle-thm-sr1}
Note that we use $\tau_k$ to replace $\sigma_k$ for measuring how well does $\G_k$ approximate $\H_k$. We modify the proof in \citet[Theorem 18]{lin2021faster} to obtain our results.

\paragraph{SR1 Method:}
First, we define the random sequence $\big\{\eta_k \big\}$ as follows
\begin{equation}
\label{eq:defetak}
    \eta_k\defeq \frac{{\rm tr}(\G_k-\H_k)}{{\rm tr}(\H_k)}.
\end{equation}
Since $\G_{k+1}=\rm{SR1}(\widetilde{\G}_k,\H_{k+1},\u_k)$,
we have
\begin{eqnarray*}
\EB_{\u_k}[{\rm tr}(\G_{k+1}-\H_{k+1})]&\overset{\eqref{eq:sr1tau}}{\leq}& (1-\frac{1}{n}){\rm tr}(\widetilde{\G}_k-\H_{k+1})\\
&\overset{\eqref{eq-H-concordant-app}}{\leq}& \left(1-\frac{1}{n}\right){\rm tr}\left((1+Mr_k)\G_k-\frac{1}{1+Mr_k}\H_k\right)\\
&\overset{\eqref{eq:defetak}}{=}&\left(1-\frac{1}{n}\right)\left((1+Mr_k)(1+\eta_k)-\frac{1}{1+Mr_k}\right){\rm tr}(\H_{k})\\
&\leq&\left(1-\frac{1}{n}\right)\left((1+Mr_k)^2(1+\eta_k)-1\right){\rm tr}(\H_{k+1}).
\end{eqnarray*}
Thus, we obtain
\begin{eqnarray*}
\EB_{\u_k}\left[\eta_{k+1}\right]&\leq&\left(1-\frac{1}{n}\right)((1+Mr_k)^2(1+\eta_k)-1)\\
&\leq&\left(1-\frac{1}{n}\right)\left[(1+Mr_k)^2\eta_k+Mr_k(Mr_k+2)\right]\\
&\leq& \left(1-\frac{1}{n}\right)(1+Mr_k)^2(\eta_k+2Mr_k).
\end{eqnarray*}
The last inequality comes from the fact that 
\begin{eqnarray*}
        t(t+2)=t^2+2t\leq2t+4t^2+2t^3=(1+t^2)2t
\end{eqnarray*}
for all $t>0$.

Since we have $\mu^2\I\preceq\H_k\preceq L^2\I$, then
\begin{equation*}
    \sigma_k={\rm tr}((\G_k-\H_k)\H_k^{-1})\leq\frac{1}{\mu^2}{\rm tr}(\G_k-\H_k)=\frac{\eta_k}{\mu^2}{\rm tr}(\H_k)\leq n\eta_k\kappa^2.
\end{equation*}
Using the result of \eqref{eq-sigmaproperty-1}, we have $\G_k \preceq (1+\sigma_k)\H_k$, which implies
\begin{align}\label{ieq:eta-sigma}
    \eta_k \leq 1+\sigma_k.
\end{align}
According to Theorem \ref{thm-lambda-general}, we have
\begin{equation*}
    \lambda_{k+1} \overset{\eqref{ieq:eta-sigma}}{\leq} \left(1-\frac{1}{1+\sigma_k}\right) \lambda_k +\beta \lambda_k^2\leq \sigma_k\lambda_k +\beta\lambda_k^2 \leq (n\kappa^2\eta_k)\lambda_k + \beta \lambda_k^2.
\end{equation*}
And it holds that
\begin{equation*}
    r_k\leq \lambda_k/\mu.
\end{equation*}
Recall that $\rho_k= M\lambda_k/\mu$, and we have
\begin{eqnarray}
    2\rho_{k+1} &\leq& (2n\kappa^2\eta_k)\rho_k +\frac{1}{4}\rho_k^2 \nonumber\\
    \label{eq:sr1gammaupdate}
    &\leq& 2n\kappa^2\rho_k(\eta_k+2\rho_k)\\
    &\leq&\left(1-\frac{1}{n}\right)4n\kappa^2\rho_k(\eta_k+2\rho_k)\nonumber\\
    \label{eq:sr1gamma}
    &\leq&\left(1-\frac{1}{n}\right)(1+\rho_k)^24n\kappa^2\rho_k(\eta_k+2\rho_k),
\end{eqnarray}
and 
\begin{eqnarray}
\label{eq:sr1etak}
    \EBP{\eta_{k+1}}\leq \left(1-\frac{1}{n}\right)\EBP{(1+\rho_k)^2(\eta_k+2\rho_k)}.
\end{eqnarray}
Combing above results, we obtain
\begin{eqnarray}
    \EBP{\eta_{k+1}+2\rho_{k+1}} &\overset{\eqref{eq:sr1etak},\eqref{eq:sr1gamma}}{\leq}&\left(1-\frac{1}{n}\right)\EBP{(1+\rho_k)^2(1+4n\kappa^2\rho_k)(\eta_k+2\rho_k)}\\
    \label{eq:thetakinduc}
    &\leq& \left(1-\frac{1}{n}\right)\EBP{\exp{\left(2\rho_k+4n\kappa^2\rho_k\right)}(\eta_k+2\rho_k)}.
\end{eqnarray}
Let $\theta_k \defeq \eta_k+2\rho_k$, then
\begin{equation*}
    \EBP{\theta_{k+1}}\leq \left(1-\frac{1}{n}\right)\EBP{\exp{\left(2(1+2n\kappa^2)\rho_k\right)}\theta_k}.
\end{equation*}
The initial condition means that
\begin{equation}
\label{eq-sr1-initial_0}
    \rho_0 \leq \frac{\ln2}{8(1+2n\kappa^2)\kappa^2},
\end{equation}
In the following, we use induction to prove the fact that
\begin{equation}
    \EBP{\theta_k}\leq \left(1-\frac{1}{n}\right)^k2\kappa^2.
\end{equation}
For $k=0$, the initial condition $\G_0\preceq\kappa^2\H_0$ means $\eta_0\leq\kappa^2-1$. Thus we obtain
\begin{eqnarray}
\label{eq:theta0bound1}
    \theta_0&=&2\rho_0+\eta_0
            \overset{\eqref{eq-sr1-initial}}{\leq} \frac{\ln2}{4(1+2n\kappa^2)\kappa^2}+(\kappa^2-1)\leq \kappa^2.
\end{eqnarray}
For $k\geq 1$, we have
\begin{align}
    \EBP{\theta_{k+1}}&\overset{\eqref{eq:thetakinduc}}{\leq}\left(1-\frac{1}{n}\right)\EBP{\exp{\left(2(1+2n\kappa^2)\rho_k\right)}\theta_k}\\
    &\overset{\eqref{eq:rhokbound}}{\leq} \left(1-\frac{1}{n}\right)\EBP{\exp{\left(2(1+2n\kappa^2)(1-1/(4\kappa^2))^{k}\rho_0\right)}\theta_k}
  \\
    &=\left(1-\frac{1}{n}\right)\exp{\left(2(1+2n\kappa^2)(1-1/(4\kappa^2))^{k}\rho_0\right)}\EBP{\theta_k}\\
            &\leq \left(1-\frac{1}{n}\right)^{k+1} \exp{\left(2(1+2n\kappa^2)\rho_0\sum_{i=0}^{k}\left(1-\frac{1}{4\kappa^2}\right)^{i}\right)} \EBP{\theta_0}\\
        &\leq \left(1-\frac{1}{n}\right)^{k+1} \exp{\left(8\kappa^2(1+2n\kappa^2)\rho_0\right)} \EBP{\theta_0}\\
         \label{eq:sr1thetabound}
        &\overset{\eqref{eq-sr1-initial_0},\eqref{eq:theta0bound1}}{\leq} \left(1-\frac{1}{n}\right)^{k+1} 2\kappa^2
\end{align}
which implies
\begin{equation*}
   \EBP{ \eta_k} \leq \EBP{\theta_k} \leq \left(1-\frac{1}{n}\right)^k2\kappa^2.
\end{equation*}
Finally, we have
\begin{equation}
\label{eq:sr1lambda}
    \EBP{\frac{\lambda_{k+1}}{\lambda_k}}=\EBP{\frac{\rho_{k+1}}{\rho_k}}\overset{\eqref{eq:sr1gammaupdate}}{\leq} \EBP {n\kappa^2\theta_k}\overset{\eqref{eq:sr1thetabound}}{\leq}\EBP{\left(1-\frac{1}{n}\right)^{k}2n\kappa^4}.
\end{equation}
which is equivalent to \eqref{eq:sr1thm}.

\subsection{The proof of Corollary \ref{co:Broydenconvergence-g}}
\label{sec:proofofco-g}
\begin{proof}
We split the proofs of three algorithms into different subsections.

\subsubsection{Greedy Broyden Family Method}
To obtain the local superlinear convergence rate, we desire the first $k_0$ iterations guarantee $z_{k_0}$ falls into a sufficient small neighbour of the saddle point. Concretely, the condition of $\eqref{eq-general-initial-2}$ means we require
\begin{equation}\label{eq-broyden-local-condition0}
    \left(1-\frac{1}{4\kappa^2}\right)^{k_0}\leq \frac{1}{1+2n}.
\end{equation}
In the view of \eqref{eq-g-fth-result-2}, we also require
\begin{equation}\label{eq-broyden-local-condition}
    \left(1-\frac{1}{n\kappa^2}\right)^{k_0}2n\kappa^2 \leq \frac{1}{2}.
\end{equation}
It is easy to verified that $k_0=n\kappa^2\ln(4n\kappa^2)$ satisfies \eqref{eq-broyden-local-condition0} and \eqref{eq-broyden-local-condition} simultaneously.

For the first $k_0$ iterations, we have
\begin{eqnarray}
\label{eq-broyden-k0-iter}
\lambda_k \leq \left(1-\frac{1}{4\kappa^2}\right)^k\lambda_0.
\end{eqnarray}
According to Theorem \ref{thm-Broyden-superlinear}, we have
\begin{equation*}
    \lambda_{k+k_0+1} \overset{\eqref{eq-g-fth-result-2}}{\leq} \left(1-\frac{1}{n\kappa^2}\right)^{k}\left(1-\frac{1}{n\kappa^2}\right)^{k_0}2n\kappa^2\lambda_{k+k_0}\overset{\eqref{eq-broyden-local-condition}}{\leq} \frac{1}{2}\left(1-\frac{1}{n\kappa^2}\right)^{k}\lambda_{k+k_0}.
\end{equation*}
Thus we obtain
\begin{eqnarray*}
\lambda_{k+k_0}&\leq&\lambda_{k_0}\prod_{i=0}^{k-1}\left[\frac{1}{2}\left(1-\frac{1}{n\kappa^2}\right)^{i}\right]\\
               &=& \left(1-\frac{1}{n\kappa^2}\right)^{\frac{k(k-1)}{2}}\left(\frac{1}{2}\right)^k\lambda_{k_0}    \\
              &\overset{\eqref{eq-broyden-k0-iter}}{\leq} &\left(1-\frac{1}{n\kappa^2}\right)^{\frac{k(k-1)}{2}}\left(\frac{1}{2}\right)^k \left(1-\frac{1}{4\kappa^2}\right)^{k_0}\lambda_0.
\end{eqnarray*}
for all $k\geq 0$.

\subsubsection{Greedy BFGS Method}
Similar to the analysis for Broyden family method, we require the point $\vz_{k_0}$ sufficient close to the saddle point such that
\begin{equation*}
    \left(1-\frac{1}{4\kappa^2}\right)^{k_0}\leq \frac{1}{1+2n}
\qquad \text{and} \qquad
    \left(1-\frac{1}{n}\right)^{k_0}2n\kappa^2 \leq \frac{1}{2},
\end{equation*}
which can be guaranteed by setting $k_0=\max\left\{n,4\kappa^2 \right\}\ln(4n\kappa^2)$.

The remainder of the proof can follow the analysis in last subsection. 
We only need to replace all the term of $\left(1-\frac{1}{n\kappa^2}\right)$ to  $\left(1-\frac{1}{n}\right)$. The reason is that Theorem \ref{thm-Broyden-superlinear}  provides a faster convergence result
\begin{align*}
    \lambda_{k+1}\leq\left(1-\frac{1}{n}\right)^k2n\kappa^2\lambda_{k}
\end{align*}
for BFGS update, rather than 
\begin{align*}
    \lambda_{k+1}\leq\left(1-\frac{1}{n\kappa^2}\right)^k2n\kappa^2\lambda_{k}  
\end{align*}
for Broyden case.

\subsubsection{Greedy SR1 Method}
Similarly, we require the point $\z_{k_0}$ satisfies
\begin{equation*}
    \left(1-\frac{1}{4\kappa^2}\right)^{k_0}\leq \frac{1}{1+2n\kappa^2}
\qquad\text{and}\qquad
    \left(1-\frac{1}{n}\right)^{k_0}2n\kappa^4 \leq \frac{1}{2}.
\end{equation*}
which can be guaranteed by setting $k_0=\max\{n,4\kappa^2\}\ln(4n\kappa^4)$. Following the ideas in previous subsections, we can prove the result for SR1 method by applying Theorem \ref{thm-Broyden-superlinear}.
\end{proof}

\subsection{The Proof of Corollary
\ref{co:Broydenconvergence-random}}
\label{sec:proofofco-r}
\begin{proof}
We split the proofs of three algorithms into different subsections.

We consider the random variable $\X_k=\lambda_{k+1}/\lambda_k, \forall k \geq 0 $ in the following derivation.
\subsubsection{Random Broyden Family Method}
The proof is modified from the proof of Corollary 10 in \citet{lin2021faster},
	Note that $\mX_k \geq 0$, using Markov's inequality, we have for any $\epsilon>0$, 
	\begin{equation}\label{eq:pp}
	\BP\left(\X_k \geq \frac{2n\kappa^2}{\epsilon} \left(1-\frac{1}{n\kappa}\right)^{k} \right) \leq \frac{\EBP{\ \mX_k}}{\frac{2n\kappa^2}{\epsilon} \left(1-\frac{1}{n\kappa^2}\right)^{k}} \stackrel{\eqref{eq-g-fth-result-2}}{\leq} \epsilon.
	\end{equation}
	Choosing $\epsilon_k = \delta (1-q)q^k$ for some positive $q < 1$, then we have 
	\begin{equation*}
	\begin{aligned}
	\BP\left(\mX_k \geq \frac{2n\kappa^2}{\epsilon_k} \left(1-\frac{1}{n\kappa^2}\right)^{k}, \exists  k\in\BN\right) 
	\leq&  \sum_{k=0}^\infty \BP\left(\mX_k
	\geq \frac{2n\kappa^2}{\epsilon_k}\left(1-\frac{1}{n\kappa^2}\right)^{k} \right) \\
	\stackrel{(\ref{eq:pp})}{\leq}&  \sum_{k=0}^\infty \epsilon_k = \sum_{k=0}^\infty \delta (1-q)q^k = \delta.
	\end{aligned}
	\end{equation*}
	Therefore, we obtain with probability $1-\delta$,
	\begin{equation*}
	\mX_k \leq \left(\frac{1-\frac{1}{n\kappa^2}}{q}\right)^{k} \cdot \frac{2n\kappa^2}{(1-q)\delta}, \ \forall k\in\BN.
	\end{equation*}
	If we set $q = 1 - \frac{1}{n^2\kappa^4}$, we could obtain with probability $1-\delta$, for all $k \in \BN$,
	\begin{equation*}
	\mX_k \leq \frac{2n^3\kappa^6}{\delta} \left(1+\frac{1}{n\kappa^2}\right)^{-k} = \frac{2n^3\kappa^6}{\delta} \left(1-\frac{1}{n\kappa^2+1}\right)^{k}.
	\end{equation*}
	Furthermore, it holds with probability $1-\delta$ that
	\begin{equation}\label{eq:r_tel}
	\frac{\lambda_{k+1}}{\lambda_k} \le \frac{4n^3\kappa^6}{\delta} \left(1-\frac{1}{n\kappa^2+1}\right)^{k}, \forall k \in \BN, 
	\end{equation}
	Telescoping from $k$ to $0$ in Eq.~\eqref{eq:r_tel}, we get
	\begin{align*}
	\lambda_{k} &= \lambda_0 \cdot \prod_{i=1}^k \frac{\lambda_i}{\lambda_{i-1}} \stackrel{\eqref{eq:r_tel}}{\le} \lambda_0 \cdot \left(\frac{4n^3\kappa^6}{\delta}\right)^{k} \prod_{i=1}^k \left(1 - \frac{1}{n\kappa^2+1}\right)^{i-1} \\
	&= \left(\frac{4n^3\kappa^6}{\delta}\right)^{k}\left(1-\frac{1}{n\kappa^2+1}\right)^{k(k-1)/2} \lambda_0.
	\end{align*}
	
	Now we combine this result with Theorem \ref{thm:G-linear-convergence},we give the entire period convergence estimator.
	Denote by $k_1\geq 0$ the number of the first iteration, for which
	\[ \left(1-\frac{1}{4\kappa^2}\right)^{k_1} \leq  \frac{1}{2n+1}. \]
	Clearly, $k_1 \leq 4 \kappa^2 \ln(2n+1)$. 
	
	In view of \ref{eq:r_tel} denote by $k_2\geq 0$ the number of the first iteration, for which
	\[ \frac{4n^3\kappa^6}{\delta}\left(1-\frac{1}{n\kappa^2+1}\right)^{k_2} \leq \frac{1}{2}. \]
	Clearly, $k_2 \leq (n\kappa^2+1) \ln(8n^3\kappa^6/\delta)$. 
	
    Thus for all $k\geq 0$, we have
	\[ \lambda_{k_1+k_2+k+1} \leq \frac{4n^3\kappa^6}{\delta} \left(1-\frac{1}{n\kappa^2+1}\right)^{k_2+k} \, \lambda_{k_1+k_2+k} \leq \frac{1}{2} \left(1-\frac{1}{n\kappa^2+1}\right)^{k} \, \lambda_{k_1+k_2+k}. \]
	Therefore,
	\[ \lambda_{k_1+k_2+k} \leq \left(1-\frac{1}{n\kappa^2+1}\right)^{k(k-1)/2} \left(\frac{1}{2}\right)^k \, \lambda_{k_1+k_2}, \]
	and 
	\[ \lambda_{k_1+k_2} \leq \left(1-\frac{1}{4\kappa^2}\right)^{k_1+k_2}\lambda_0. \]
	Finally, choose $k_0=k_1+k_2 = \OM\left(n\kappa^2\ln(n\kappa/\delta)\right)$, we obtain
	\[ \lambda_{k_0+k} \leq  \left(1-\frac{1}{n\kappa^2+1}\right)^{k(k-1)/2} \cdot \left(\frac{1}{2}\right)^k  \cdot \left(1-\frac{1}{4\kappa^2}\right)^{k_0}\lambda_0. \]

\subsubsection{Random BFGS Method}
Similar to the analysis for random Broyden family method, 
we obtain with probability $1-\delta$,
	\begin{equation*}
	\X_{k} \leq \left(\frac{1-\frac{1}{n}}{q}\right)^k \cdot \frac{2n\kappa^2}{(1-q)\delta}, \forall k\in\BN.
	\end{equation*}
	If we set $q = 1 - 1 / n^2$, we could obtain with probability $1-\delta$,
	\begin{equation}
	\label{eq:lambda-randomsuper}
 \lambda_{k+1} \leq  \frac{4n^3\kappa^2}{\delta}\left(1-\frac{1}{n+1}\right)^{k} \lambda_k, \text{ for all } k \in \BN.
	\end{equation}
we require the point $\vz_{k_0}$ where $k_0=k_1+k_2$ sufficient close to the saddle point such that
\begin{equation*}
    \left(1-\frac{1}{4\kappa^2}\right)^{k_1}\leq \frac{1}{1+2n}
\qquad \text{and} \qquad
    \left(1-\frac{1}{n+1}\right)^{k_2}4n^3\kappa^2/\delta \leq \frac{1}{2},
\end{equation*}
which can be guaranteed by setting $k_0=k_1+k_2=\OM\left(\max\left\{n,\kappa^2 \right\}\ln(n\kappa/\delta)\right)$.

The remainder of the proof can follow the analysis in random Broyden methdos. 
We only need to replace all the term of $\left(1-\frac{1}{n\kappa^2+1}\right)$ to  $\left(1-\frac{1}{n+1}\right)$. The reason is that \eqref{eq:lambda-randomsuper}  provides a faster convergence result for random BFGS update, rather than \eqref{eq:r_tel}
for random Broyden case.

\subsubsection{Random SR1 Method}
	Similarly, for SR1 update,
	\begin{equation*}
	\X_{k} \leq \left(\frac{1-\frac{1}{n}}{q}\right)^k \cdot \frac{2n\kappa^4}{(1-q)\delta}, \forall k\in\BN.
	\end{equation*}
	If we set $q = 1 - 1 / n^2$, we could obtain with probability $1-\delta$,
	\begin{equation*}
	\lambda_{k+1} \leq  \frac{4n^3\kappa^4}{\delta}\left(1-\frac{1}{n+1}\right)^{k} \lambda_k, \text{ for all } k \in \BN.
	\end{equation*}
And we require the point $\z_{k_0}$ where $\z_0=\z_1+\z_2$ satisfies
\begin{equation*}
    \left(1-\frac{1}{4\kappa^2}\right)^{k_1}\leq \frac{1}{1+2n\kappa^2}
\qquad\text{and}\qquad
    \left(1-\frac{1}{n+1}\right)^{k_0}\frac{4n^3\kappa^4}{\delta} \leq \frac{1}{2}.
\end{equation*}
which can be guaranteed by setting $k_0=\OM\left(\max\{n,\kappa^2\}\ln(n\kappa/\delta)\right)$. Following the ideas in previous subsections, we can prove the result for random SR1 method.
\end{proof}

\section {The Proofs for Section \ref{sec:extension}}
This section provides the proofs of all results in Section \ref{sec:extension}.
We first introduce the following notations
\begin{align*}
    L' \defeq 2L,\qquad \mu' \defeq \sqrt{2}\mu \qquad\text{and}\qquad \kappa' \defeq \frac{L'}{\mu'}=\frac{2\sqrt{2}L}{\mu}.
\end{align*}
The definition of $\g_k$, $\hH_k$ and $\H$ are similar to
Section \ref{sec-general-convergence}, that is 
\begin{align*}
    \g_k \defeq \g(\z_k),\qquad \hH_k \defeq \hH(\z_k) \qquad\text{and}\qquad \H_k \defeq \big(\hH_k\big)^2.
\end{align*}
The only difference is they are not dependent on specific saddle point problems.

Besides, we also 
use $\sigma_{\max}(\cdot)$ and $\sigma_{\min}(\cdot)$ the present the largest singular value and the smallest singular value of the given matrix respectively. And we define the local neighbor of the solution $\vz^*$ as follows 
\begin{align*}
\Omega^*\defeq\left\{\z: \|\z-\z^*\|\leq D\right\},
\quad\text{where}~~ D=\dfrac{\mu^2}{8L_2L}.
\end{align*}

\subsection{The Proof of Lemma \ref{lem:localNLS}}
\begin{proof}
Assumption \ref{ass:neqsysass2} implies $\sigma_{\min}\left(\H(\z^*)\right)=\mu^2$ and $\sigma_{\max}\left(\H(\z^*)\right)=L^2$. 

Since $\H(\z^*)$ is invertible and symmetric, we have $\H(\z^*)\succ0$ and

Since we have restricted $\z,\z'\in \Omega^*$, it holds that
\begin{equation}
\label{eq:hHnorm}
\big\|\hH(\z)\big\| 
\leq \big\|\hH(\z)-\hH(\z^*)\big\| + \big\|\hH(\z^*)\big\| 
\leq L_2\Norm{\z-\z^*} + \big\|\hH(\z^*)\big\| \leq L_2D+L,
\end{equation}
which implies
\begin{equation}\label{ieq:bound-H-noninlear}
    \H(\z) \preceq (L_2D+L)^2 \I
\end{equation}
for all $\z\in\Omega^*$.
According to \eqref{ieq:bound-H-noninlear}, we have
\begin{eqnarray*}
\|\H(\z)-\H(\z')\|
&=& \Norm{\hH(\z)^2-\hH(\z')^2} \\
&\leq&\Norm{\hH(\z)\left(\hH(\z)-\hH(\z')\right)}+\Norm{\left(\hH(\z)-\hH(\z')\right)\hH(\z')}\\
&\leq& L_2\left(\big\|\hH(\z)\big\|+\big\|\hH(\z')\big\|\right)\|\z-\z'\|\\
&\overset{\eqref{eq:hHnorm}}{\leq}&2L_2(L_2D+L)\|\z-\z'\|.
\end{eqnarray*}
Combining above result with the Weyl's inequality for singular values \cite[Theorem 3.3.16]{horn1994topics}, we have
\begin{eqnarray*}
\sigma_{\min}(\H(\z^*)) - 2L_2(L_2D+L)\|\z-\z^*\| \leq \sigma_{\min}(\H(\z)).
\end{eqnarray*}
Since the definition of $D$ indicates $2L_2(L_2D+L)D\leq\frac{\mu^2}{2}$ and $L_2D\leq L$, we have
\begin{equation*}
   \frac{\mu^2}{2} \I\preceq\H(\z) \preceq 4L^2 \I,
\end{equation*}
and
\begin{equation*}
    \|\H(\z)-\H(\z')\|\leq 4L_2L\|\z-\z'\|
\end{equation*}
\end{proof}
\subsection{The Proof of Lemma \ref{lm-self-concordant2}}
\begin{proof}
According to Lemma \ref{lem:localNLS}, we have
\begin{eqnarray*}
\H(\vz)-\H(\vz') &\stackrel{\eqref{eq:nonlinar-obj-conti}}\preceq& 4L_2L \|\vz-\vz'\| \I \\
            &\stackrel{\eqref{eq:nonlinear-obj-bound}}\preceq& \frac{8L_2L}{\mu^2} \|\vz-\vz'\| \H(\vw)\\
            &=& \frac{8\kappa^2L_2}{L}\|\vz-\vz'\| \H(\vw).
\end{eqnarray*}
for all $\z,\z'\in \Omega^*$.
\end{proof}

\subsection{The Proof of Lemma \ref{lm:zallin}}
\begin{proof}
We prove this lemma by induction.
For $k=0$, it is obviously. 
Suppose the statement holds for all $k'\leq k$. Then for all $k'=0,\dots, k$, we have $\z_{k'}\in \Omega^*$ and $\z_{k'}$ holds that \eqref{eq:nonlinear-obj-bound}, \eqref{eq:nonlinar-obj-conti} and \eqref{eq:Nonlinear-concor}.
By Theorem \ref{thm:G-linear-convergence}, we guarantee 
\begin{equation}\label{nonlinear-matrix-order}
    \H_{k'}\leq \G_{k'}\leq 16\kappa^2 \H_{k'}.
\end{equation}
For $k'=k+1$, according to the proof of Lemma \ref{thm-lambda-general}, we have
\begin{eqnarray}
  \g_{k+1}=\underbrace{(\I -\hH_{k}\G_k^{-1}\hH_{k})\g_k}_{\va_k}+\underbrace{\int_{0}^1\left[\hH(\z_k+s(\z_{k+1}-\z_k))-\hH(\z_k)\right] (\z_{k+1}-\z_k)ds}_{\vb_k}.
\end{eqnarray}
Using Lemma \ref{lm-matrix-neq} and the result of \eqref{nonlinear-matrix-order}, we have
\begin{eqnarray}
\|\I-\hH_k\G_k^{-1}\hH_k\|\leq 1-\frac{1}{16\kappa^2}
\end{eqnarray}
and 
\begin{eqnarray}
\|\vb_k\|\leq\frac{L_2}{\mu^2}\lambda_k^2.
\end{eqnarray}
Combing above results,  we have
\begin{equation}
    \lambda_{k+1}\leq \left(1-\frac{1}{16\kappa^2}\right)\lambda_k+\frac{L_2}{\mu^2}\lambda_{k}^2 \leq \left(1-\frac{1}{32\kappa^2}\right)\lambda_k.
\end{equation}
Thus we always have $\|\z_{k+1}-\z_k\|\leq \frac{1}{\mu}\lambda_{k+1}\leq\lambda_0\leq D$.
By induction, we finish the proof.
\end{proof}

\subsection{The Proof of Theorem \ref{thm:Broydenconvergence-Nonlinear} and Theorem \ref{co::Broydenconvergence-Nonlinear-random}}
The initial condition on $\lambda_0$ and Lemma \ref{lm:zallin} means all the points $\z_k$ generated from our algorithms are located in $\Omega^*$. Thus we can directly use the results of Corollary \ref{co:Broydenconvergence-g} and \ref{co:Broydenconvergence-random} by replacing $\kappa$ and $\mu$ of the Corollary \ref{co:Broydenconvergence-g} and \ref{co:Broydenconvergence-random} into $\kappa'=\frac{L'}{\mu'}=2\sqrt{2}\kappa$ and $\mu'=\frac{1}{\sqrt{2}}\mu$ here.

\section{Experimental Details}\label{sec:expapp}

We provide some details for our experiments in this section.
Our experiments are conducted on a work station with 56 Intel(R) Xeon(R) Gold 6132 CPU @ 2.60GHz and 256GB memory. We use MATLAB 2021a to run the code and the operating system is Ubuntu 20.04.2. 

\subsection{AUC Maximization}
The gradient of the object function at $\z=[\x;\y]=[\w;u;v;y]$ is
\begin{align*}
    \g(\z)=\nabla f(\z)=\frac{1}{m}\sum_{i=1}^{m}\begin{bmatrix}
 \nabla_{\w} f_i(\z)\\
  \nabla_{u}f_i(\z)\\
  \nabla_{v}f_i(\z)\\
 \nabla_{y}f_i(\z)
    \end{bmatrix},
\end{align*}
where
\begin{align*}
\nabla_{\w} f_i(\z) &= \lambda \w +2(1-p)(\w^{\top}\va_i-u-1-y))\va_i\mathbb{I}_{b_i=1}+2p(\w_i^{\top}\va_i-v+1+y)\va_i\mathbb{I}_{b_i=-1},\\
\nabla_{u} f_i(\z) &= \lambda u -2(1-p)(\w^{\top}\va_i-u)\mathbb{I}_{b_i=1}, \\
\nabla_{v} f_i(\z) &= \lambda v - 2p(\w^{\top}\va_i-v)\mathbb{I}_{b_i=-1},\\
\nabla_{y} f_i(\z) &= -2p(1-p)y+2p\w^{\top}\va_i\mathbb{I}_{b_i=-1}-2(1-p)\w^{\top}\va_i\mathbb{I}_{b_i=1}.
\end{align*}

The Hessian-vector of the object function is
\begin{align*}
    \nabla^{2}f(\z) \h =  \hH(\vz)\h=\frac{1}{m}\sum_{i=1}^{m}\begin{bmatrix}
 (\nabla^2f_i(\z)\h)_{\w}\\
    (\nabla^2f_i(\z)\h)_{u}\\
    (\nabla^2f_i(\z)\h)_{v}\\
    (\nabla^2f_i(\z)\h)_{y}
    \end{bmatrix},
\end{align*} 
where $\h=[\h_\vw; \h_u; \h_v; \h_y]$ such that
\begin{align*}
(\nabla^2f_i(\z)\h)_{\w}&= \lambda\h_{\w}+2(1-p)(\inner{\va_i}{\h_{\vw}}-\h_{u}-\h_y)\va_i\mathbb{I}_{b_i=1}+2p(\inner{\va_i}{\h_{\w}}-\h_{v}+\h_{y})\va_i\mathbb{I}_{b_i=-1}, \\
(\nabla^2f_i(\z)\h)_{u}&= -2(1-p)\va_i^{\top} \h_{\vw}\mathbb{I}_{b_i=1}+(\lambda+2(1-p)\mathbb{I}_{b_i=1})\h_{u},\\
(\nabla^2f_i(\z)\h)_{v}&=-2p\va_i^{\top}\h_{\vw}\mathbb{I}_{b_i=-1}+(\lambda+2p\mathbb{I}_{b_i=-1})\h_{v}, \\
(\nabla^2f_i(\z)\h)_{y}&=-2p(1-p)y+2p\va_i^{\top}\h_{\w}\mathbb{I}_{b_i=-1}-2(1-p)\va_i^{\top}\h_{\w}\mathbb{I}_{b_i=1}.
\end{align*}
Note that the Hessian-vector can be achieved in $\OM(nm)$ flops and it guarantees $\OM(nm+n^2)$ complexity for each iteration.

For baseline method extragradient (Algorithm \ref{alg:EG}), we tune the stepsize from $\{0.01,0.05,0.1,0.5\}$. 
For RaBFGSv1-Q (Algorithm \ref{alg:BroydenQuadratic}), we let $\G_0=3\I$ for ``a9a'', ``w8a'' and $\G_0=30\I$ for ``sido0''. 
For RaBFGSv2-Q (Algorithm \ref{alg:BFGSQuadratic}), we let $\G_0=3\I$ for ``a9a'' and  ``w8a''. We do not run RaBFGSv2-Q on ``sido0'' because this algorithm is not efficient for high-dimensional problem as we have mentioned in Remark \ref{rm:BFGS}.  
For RaSR1-Q  (Algorithm \ref{alg:SR1Quadratic}), we let $\G_0=5\I$ for ``a9a'', ``w8a'' and $\G_0=30\I$ for ``sido0''.

The dataset ``sido0'' comes from Causality Workbench~\cite{guyon2008design} and the other datasets can be downloaded from LIBSVM repository \cite{CC01a}.

\begin{algorithm}[t]
\caption{Extragradient Method}\label{alg:EG}
\begin{algorithmic}[1]
\STATE \textbf{Input:} $\vz_0\in\BR^n$, and $\eta>0$. \\[0.15cm]
\STATE \textbf{for} $k=0,1,\dots$ \\[0.15cm] 
\STATE \quad $\vx_{k+1/2}=\vx_k - \eta\nabla_\vx f(\vx_k,\vy_k)$\\[0.15cm]
\STATE \quad $\vy_{k+1/2}=\vy_k + \eta\nabla_\vy f(\vx_k,\vy_k)$\\[0.15cm]
\STATE \quad $\vx_{k+1}=\vx_k - \eta\nabla_\vx f(\vx_{k+1/2},\vx_{k+1/2})$\\[0.15cm]
\STATE \quad $\vy_{k+1}=\vy_k + \eta\nabla_\vy f(\vx_{k+1/2},\vx_{k+1/2})$\\[0.15cm]
\STATE \textbf{end for}
\end{algorithmic}
\end{algorithm}

\subsection{Adversarial Debiasing}
The minimax formulation can be rewritten as 
\begin{equation*}
    \min_{\x\in\RB^d}\max_{y\in\RB}\frac{1}{m}\sum_{i=1}^{m}f_i(\x,y;\va_i,b_i,c_i,\lambda,\gamma,\beta)
\end{equation*}
where $f_i$ is defined as
\begin{equation*}
    f_i(\x,y;\va_i,b_i,c_i,\lambda,\gamma,\beta)=\log(1+\exp(-b_i\va_i^{\top}\x))-\beta \log(1+\exp(-c_i\va_i^{\top}\x y))+\lambda\|\x\|^2-\gamma y^{2}\end{equation*}.
We define
\begin{eqnarray*}
p_i = \frac{1}{1+\exp(b_i\va_i^{\top}\x)} \qquad \text{and}\qquad q_i =\frac{1}{1+\exp(c_iy\va_i^{\top}\x)}.
\end{eqnarray*}
Then the gradient of the object function at $\z=[\x;y]$ is
\begin{eqnarray*}
    \g(\z)=\nabla f(\z)=\frac{1}{m}\sum_{i=1}^{m}\begin{bmatrix}
 \nabla_{\x} f_i(\z)\\
  \nabla_{y}f_i(\z)
    \end{bmatrix},
\end{eqnarray*}
where
\begin{align*}
\nabla_\x f_i(\z) = -p_ib_i\va_i+\beta q_ic_iy\va_i+2\lambda\x
\qquad \text{and}\qquad
\nabla_{y} f_i(\z) = \beta c_i \va_i^{\top} xq_i-2\gamma y.
\end{align*}
The Hessian-vector of the object function is
\begin{equation*}
\nabla^{2}f(\z) \h =  \hH(\vz)\h=\frac{1}{m}\sum_{i=1}^{m}\begin{bmatrix}
 (\nabla^2f_i(\z)\h)_{\x}\\
(\nabla^2f_i(\z)\h)_{y}
\end{bmatrix}.
\end{equation*} 
where $\h=[\h_\vx; \h_y]$ such that
\begin{align*}
(\nabla^2f_i(\z)\h)_{\x} & = (p_i(1-p_i)\va_i^{\top}\h_{\x}-q_i(1-q_i)\beta y^2\va_i^{\top}\h_{\vx})\va_i+2\lambda \h_{\x}-(q_i(1-q_i)\beta y\va_i^{\top}\x-q_i\beta c)\h_{y}\va_i,  \\
(\nabla^2f_i(\z)\h)_{y} & =-\beta y q_i(1-q_i)\va_i^{\top}\x\va_i^{\top}\h_{\x} +q_i\beta c\va_i^{\top}\h_{\x}-q_i(1-q_i)\beta(\va_i^{\top}\x)^2\h_y-2\gamma\h_y.
\end{align*}
The Hessian-vector can be achieved in $\OM(nm)$ flops and it guarantees $\OM(nm+n^2)$ complexity for each iteration.

The experiments are based on the datasets of fairness aware machine learning~\cite{quy2021survey}. 
Following the preprocessing of \citet{platt1998fast,CC01a},
we convert the features of the original datasets into binary for our experiments. Concretely, the continuous features are discretized into quantiles, and each quantile is represented by a binary feature. Also, a categorical feature with $C$ categories is converted to $C$ binary features. More specifically, for the ``adults'' dataset, we transform the 13 features of it into 122 binary features and choose the feature of ``gender'' as the protected feature. For the ``law school'' dataset, we transform the 11 features of it into 379 binary features and choose the feature of ``gender'' as the protected feature. For the ``bank marketing'' dataset, we transform 16 features of it into 3879 binary features and choose ``marital'' as the protected feature.

For experiments, we tune the stepsize of EG from $\{0.01,0.05,0.1,0.5\}$ and run it with 4000, 20000 and 40000 iterations for the ``adults'', ``law school'' and ``Bank market''  respectively as warm up to obtain $\vz_0$ as initial point. 
Then we evaluate all algorithms (including the baseline algorithm EG)  by starting with $\vz_0$ and achieve the result shown in Figure \ref{fig:fairepoch}. 
Since each algorithm has the identical behavior in the warm up stage,
we only present the curves of iterations vs. $\norm{\vg(\vz)}$ and CPU time vs. $\norm{\vg(\vz)}$ after warm up stage in Figure \ref{fig:fairepoch}.

In addition,  we set $\G_0=2\I$ for the proposed quasi-Newton methods.
For RaBFGSv1-G (Algorithm  \ref{alg:BroydenSPP}), we set $M=0$ for ``adults'', $M=1$ for ``law school'' and ``bank market''. For RaBFGSv2-G (Algorithm \ref{alg:BFGSSPP}), we set $M=1$ for ``adults'' and ``law school''.
For RaSR1-G (Algorithm \ref{alg:SR1SPP}), we set $M=1$ for all three datasets.

\end{document}